%% file: modules_arxiv.tex
\documentclass[12pt]{article}
\usepackage{amsmath,amsfonts,amsthm,yhmath}
\usepackage[square]{natbib}
\usepackage{multirow}
\usepackage{textcomp}
\usepackage[utf8]{inputenc}
\usepackage[english]{babel}
\usepackage{graphicx,psfrag}
\usepackage{pgfplots}
\usepackage{floatflt}
\usepackage{pict2e,wrapfig}
\usepackage{caption}
\usepackage{authblk}
\input{defs_p.tex}

\usepackage[off]{auto-pst-pdf}

\title{Phase retrieval for the Cauchy wavelet transform}
\author[1]{St\'ephane Mallat}
\author[1]{Ir\`ene Waldspurger}
\affil[1]{D\'epartement d'informatique, \'Ecole Normale Sup\'erieure, Paris}
\date{}

\begin{document}

\maketitle

\begin{abstract}
We consider the phase retrieval problem in which one tries to reconstruct a function from the modulus of its wavelet transform. We study the unicity and stability of the reconstruction.

In the case where the wavelets are Cauchy wavelets, we prove that the modulus of the wavelet transform uniquely determines the function up to a global phase. We show that the reconstruction operator is continuous but not uniformly continuous. We describe how to construct pairs of functions which are far away in $L^2$-norm but whose wavelet transforms are very close, in modulus. The principle is to modulate the wavelet transform of a fixed initial function by a phase which varies slowly in both time and frequency. This construction seems to cover all the instabilities that we observe in practice; we give a partial formal justification to this fact.

Finally, we describe an exact reconstruction algorithm and use it to numerically confirm our analysis of the stability question.
\end{abstract}

\section{Introduction}

A phase retrieval problem consists in reconstructing an unknown object $f$ from a set of phaseless linear measurements. More precisely, let $E$ be a complex vector space and $\{L_i\}_{i\in I}$ a set of linear forms from $E$ to $\C$. We are given the set of all $\left|L_i(f)\right|,i\in I$, for some unknown $f\in E$ and we want to determine $f$.

\noindent This problem can be studied under three different viewpoints:
\begin{itemize}
\item Is $f$ uniquely determined by $\left\{\left|L_i(f)\right|\right\}_{i\in I}$ (up to a global phase)?
\item If the answer to the previous question is positive, is the inverse application $\left\{\left|L_i(f)\right|\right\}_{i\in I}\to f$ ``stable''? For example, is it continuous? Uniformly Lipschitz?
\item In practice, is there an efficient algorithm which recovers $f$ from $\left\{\left|L_i(f)\right|\right\}_{i\in I}$?
\end{itemize}

\nl
The most well-known example of a phase retrieval problem is the case where the $L_i$ represent the Fourier transform. The unknown object is some compactly-supported function $f\in L^2(\R,\C)$ and the problem is:
\begin{equation*}
\mbox{reconstruct }f\mbox{ from }|\hat f|
\end{equation*}
Because of its important applications in physics, this problem has been extensively studied from the 50's. Unfortunately, \citet{akutowicz} and \citet{walther} have shown that it is not solvable. Indeed, for any $f$, there generally exists an infinite number of compactly-supported $g$ such that $|\hat f|=|\hat g|$.

\nl

We are interested in the problem which consists in reconstructing $f\in L^2(\R)$ from the modulus of its wavelet transform.

A wavelet is a (sufficiently regular) function $\psi:\R\to\C$ such that $\int_\R\psi(x)dx=0$. For any $j\in\Z$, we define $\psi_j(x)=a^{-j}\psi(a^{-j}x)$, which is equivalent to $\hat\psi_j(x)=\hat\psi(a^jx)$. The number $a$ may be any real in $]1;+\infty[$. The wavelet transform of a function $f\in L^2(\R)$ is:
\begin{equation*}
\{f\star\psi_j\}_{j\in\Z}\in(L^2(\R))^\Z
\end{equation*}
Our problem is then the following:
\begin{equation}\label{eq:prob_statement}
\mbox{reconstruct }f\in L^2(\R)\mbox{ from }\{|f\star\psi_j|\}_{j\in\Z}
\end{equation}
It can be seen as a collection of phase retrieval subproblems, where the linear form of each subproblem is the Fourier transform. Provided that $\psi$ is not pathological and $f$ is uniquely determined by $\{f\star\psi_j\}$, the problem is indeed equivalent to:
\begin{equation*}
\mbox{reconstruct }\{\hat f.\hat\psi_j\}_{j\in\Z}\mbox{ from }
\left\{\left|\mathcal{F}\left(\hat{f}.\hat{\psi}_j\right)\right|\right\}_{j\in\Z}
\end{equation*}
where $\mathcal{F}$ is another notation for the Fourier transform.

Even if, for any given $j$, it is impossible to reconstruct $\hat f.\hat \psi_j$ from $\left|\mathcal{F}\left(\hat{f}.\hat{\psi}_j\right)\right|$ only, the reconstruction \eqref{eq:prob_statement} may be possible: the $\hat f.\hat \psi_j$ are not independent one from the other and we can use this information for reconstruction.

We consider here the case of Cauchy wavelets. In this case, the relations between the $\hat f.\hat\psi_j$ may be expressed in terms of holomorphic functions. This allows us to study the problem \eqref{eq:prob_statement} with the same tools as in \citep{akutowicz}. We show that $f$ is uniquely determined by $\{|f\star\psi_j|\}_{j\in\Z}$ and we are able to study the stability of the reconstruction. We show that, when the wavelet transform does not have too many small values, the reconstruction is stable, up to modulation of the different frequency bands by low-frequency phases.

\nl
This problem of reconstructing a signal from the modulus of its wavelet transform is intersesting in practice because of its applications in audio processing.

Indeed, a common way to represent audio signals is to use the modulus of some time-frequency representation, either the short-time Fourier transform (spectrogram) or the wavelet transform (scalogram, \citep{shamma,anden}). Numerical results strongly indicate that the loss of phase does not induce a loss of perceptual information. Thus, some audio processing tasks can be achieved by modifying directly the modulus, without taking the phase into account, and then reconstructing a new signal from the modified modulus (\citep{griffin_lim},\citep{balan}), which requires to solve a phase retrieval problem.



\nl
The interest of the phase retrieval problem in the case of the wavelet transform is also theoretical.

A lot of work has been devoted to finding or characterizing systems of linear measurements whose modulus suffices to uniquely determine an unknown vector. If the underlying vector space is of finite dimension $n$, it is known that $4n-4$ generic linear forms are enough to guarantee the unicity (\citep{conca}). Specific examples of such linear forms have been given by \citet{bodmann} and \citet{fickus}. \citet{candes2} and \citet{candes_li2} have constructed random measurements systems for which unicity holds with high probability and their reconstruction algorithm \textit{PhaseLift} is guaranteed to succeed. These examples either rely on randomization techniques or have been carefully designed by means of algebraic tricks to guarantee the unicity of the reconstruction. By contrast, the (Cauchy) wavelet transform is a natural and deterministic system of linear measurements, for which unicity results can be proved.

Most of the research in phase retrieval has at first focused on the unicity of the reconstruction or on the algorithmic part. The question of whether the reconstruction is stable to measurement noise is more recent. \citet{bandeira_stability} and \citet{balan_wang} gave a necessary and sufficient condition for stability in the case where the unknown vector $x$ is real but it only partially extends to the complex case. For several random measurement systems, it has been proved that, with high probability, all signals are determined by the modulus of the linear measurements, in a way which is stable to noise (see for example \citep{candes_li2} and \citep{eldar}). Again, our measurement system presents the interest of being, on the contrary, totally deterministic. Moreover, to our knowledge, it is the first case where the question of stability does not have a binary answer (the reconstruction is ``partially stable'') and where we are able to precisely describe the instabilities.

\subsection{Outline and results}

In the section \ref{s:unicity}, we prove that a function is uniquely determined by the modulus of its Cauchy wavelet transform. Precisely, if $(\psi_j)_{j\in\Z}$ is a family of Cauchy wavelets, we have the following theorem:
\begin{thm*}
If $f,g\in L^2(\R)$ are two functions such that $\hat f(\omega)=\hat g(\omega)=0$ for any $\omega<0$ and if $|f\star\psi_j|=|g\star\psi_j|$ for all $j$, then:
\begin{equation*}
f=e^{i\phi}g\mbox{ for some }\phi\in\R
\end{equation*}
\end{thm*}
The proof uses harmonic analysis tools similar to the ones used by \citet{akutowicz}.

We also give a version of this result for finite signals. The proof is similar but easier. We show that it implies a unicity result for a system of $4n-2$ linear measurements.

\nl
Then, in the section \ref{s:weak_stability}, we prove (theorem \ref{thm:weak_stability}) that the reconstruction operator is continuous.

\nl
In the section \ref{s:non_uniform_continuity}, we explain why this operator is not uniformly continuous: there exist functions $f,g$ such that $||f-g||_2\not\ll||f||_2$ and $|f\star\psi_j|\approx|g\star\psi_j|$ for all $j$. In the light of \citep{bandeira_stability}, we give simple examples of such $(f,g)$. We then describe a more general construction of pairs $(f,g)$. The principle of this construction is to multiply the wavelet transform of a fixed signal $f$ by a ``slow-varying'' phase. Projecting this modified wavelet transform on the set of admissible wavelet transforms yields a new signal $g$. For each $j$, we have $|f\star\psi_j|\approx|g\star\psi_j|$, but we may have $f\not\approx g$.

\nl
In the section \ref{s:strong_stability_result}, we give explicit reconstruction formulas. We use them to prove a local form of stability of the reconstruction problem (theorems \ref{thm:stability_dyadic_case} and \ref{thm:stability_a_less_than_2}). Our result is approximately the following:
\begin{thm*}
Let $f,g\in L^2(\R)$ be such that $\hat f(\omega)=\hat g(\omega)=0$ for any $\omega<0$.

Let $j\in\Z,K\in\N^*$ be fixed.

We assume that, for each $l=j+1,...,j+K$, we have, for all $x$ in some interval:
\begin{gather*}
|f\star\psi_l(x)|\approx |g\star\psi_l(x)|\\
|f\star\psi_l(x)|, |g\star\psi_l(x)|\not\approx 0
\end{gather*}
Then, for some low-frequency function $h$:
\begin{equation*}
h.(f\star\psi_j)\approx g\star\psi_j
\end{equation*}
\end{thm*}
This implies that, if the modulus of the Fourier transform does not have too small values, all the instabilities of the reconstruction operator are of the form described in section \ref{s:non_uniform_continuity}.

\nl
Finally, in the section \ref{s:numerical_experiments}, we present an algorithm which exactly recovers a function from the modulus of its Cauchy wavelet transform (and a low-frequency component). This algorithm uses the explicit formulas derived in the section \ref{s:strong_stability_result}. It may fail when the wavelet transform is too close to zero at some points but otherwise it almost always succeeds. It does not get stuck into local minima, like most classical algorithms (for example \citet{gerchberg}), and it is stable to noise.


\subsection{Notations}

For any $f\in L^1(\R)$, we denote by $\hat f$ or $\mathcal{F}(f)$ the Fourier transform of $f$:
\begin{equation*}
\hat f(\omega)=\int_\R f(x)e^{-i\omega x}dx
\quad \forall \omega\in\R
\end{equation*}
We extend this definition to $L^2$ by continuity.

We denote by $\mathcal{F}^{-1}:L^2(\R)\to L^2(\R)$ the inverse Fourier transform and recall that, for any $f\in L^1\cap L^2(\R)$:
\begin{equation*}
\mathcal{F}^{-1}(f)(x)=\frac{1}{2\pi}\int_\R f(\omega)e^{i\omega x}d\omega
\end{equation*}

\nl
We denote by $\H$ the Poincaré half-plane: $\H=\{z\in\C\mbox{ s.t. }\Im z>0\}$.

\section{Unicity of the reconstruction for Cauchy wavelets\label{s:unicity}}

\subsection{Definition of the wavelet transform and comparison with Fourier\label{s:comparison}}

The most important phase retrieval problem, which naturally arises in several physical settings, is the case of the Fourier transform:
\begin{equation*}
\mbox{reconstruct }f\in L^2(\R)\mbox{ from }|\hat f|
\end{equation*}
Without additional assumptions over $f$, the reconstruction is clearly impossible: any choice of phase $\phi:\R\to\R$ yield a signal $g=\mathcal{F}^{-1}(|\hat f|e^{i\phi})\in L^2(\R)$ such that $|\hat g|=|\hat f|$.

To avoid this problem, one may for example require that $f$ is compactly supported. However, \citet{akutowicz_compact,walther} showed that, even with this constraint, the reconstruction was still not possible.

More precisely, their result is the following one. If $f\in L^2(\R)$ is a compactly supported function, then its Fourier transform $\hat f$ admits a holomorphic extension $F$ over all $\C$: $F(z)=\int_\R f(x)e^{-izx}dx$. If $g\in L^2(\R)$ is another compactly supported function and $G$ is this holomorphic extension of its Fourier transform, the equality $|\hat f|=|\hat g|$ happens to be equivalent to:
\begin{equation*}
\forall z\in\C,\quad
F(z)\overline{F(\overline{z})}=G(z)\overline{G(\overline{z})}
\end{equation*}
This in turn is essentially equivalent to:
\begin{equation}\label{eq:equality_zeros}
\{z_n\}\cup\{\overline{z}_n\}=\{z'_n\}\cup\{\overline{z}'_n\}
\end{equation}
where the $(z_n)$ and $(z'_n)$ are the respective zeros of $F$ and $G$ over $\C$, counted with multiplicity. This means that $F$ and $G$ must have the same zeros, up to symmetry with respect to the real axis.

Conversely, for every choice of $\{z'_n\}$ satisfying \eqref{eq:equality_zeros}, it is possible to find a compactly supported $g$ such that the zeroes of $G$ are the $z'_n$, which implies $|\hat f|=|\hat g|$.


A similar result can be established in the case where the function $f\in L^2(\R)$ is assumed to be identically zero on the negative real line \citep{akutowicz} instead of compactly supported.

\nl
Let us know define the wavelet transform and compare it with the Fourier transform.

Let $\psi\in L^1\cap L^2(\R)$ be a wavelet, that is a function such that $\int_\R\psi(x)dx=0$. Let $a>1$ be fixed; we call $a$ the \textit{dilation factor}. We define a family of wavelets by:
\begin{equation*}
\forall x\in\R\quad
\psi_j(x)=a^{-j}\psi(a^{-j}x)
\quad\quad
\Leftrightarrow
\quad\quad
\forall \omega\in\R\quad
\hat\psi_j(\omega)=\hat\psi(a^j\omega)
\end{equation*}
The wavelet transform operator is:
\begin{equation*}
f\in L^2(\R)\to\{f\star\psi_j\}_{j\in\Z}\in (L^2(\R))^\Z
\end{equation*}
This operator is unitary if the so-called Littlewood-Paley condition is satisfied:
\begin{equation}\label{eq:littlewood_paley}
\left(\underset{j}{\sum}|\hat\psi_j(\omega)|^2=1,\forall\omega\in\R\right)
\quad\quad\Rightarrow
\quad\quad
\left(||f||_2^2=\underset{j}{\sum}||f\star\psi_j||_2^2\quad\forall f\in L^2(\R)\right)
\end{equation}

The phase retrieval problem associated with this operator is:
\begin{equation*}
\mbox{reconstruct }f\in L^2(\R)\mbox{ from }\{|f\star\psi_j|\}_{j\in\Z}
\end{equation*}

\nl

This problem may or may not be well-posed, depending on which wavelet family we use.

The simplest case is the one where the wavelets are Shannon wavelets:
\begin{equation*}
\hat\psi=1_{[1;a]}\hskip 1cm
\Rightarrow
\hskip 1cm\forall j\in\Z,\quad
\hat \psi_j(\omega)=1_{[a^{-j};a^{-j+1}]}
\end{equation*}
Reconstructing $f$ amounts to reconstruct $\hat f1_{[a^{-j};a^{-j+1}]}=\hat f\hat\psi_j$ for all $j$. For each $j$, we have only two informations about $\hat f\hat\psi_j$: its support is included in $[a^{-j};a^{-j+1}]$ and the modulus of its inverse Fourier transform is $|f\star\psi_j|$. From the results of the Fourier transform case, it is not enough to determine uniquely $\hat f\hat\psi_j$. Thus, for Shannon wavelets, the phase retrieval problem is as ill-posed as for the Fourier transform.

In this example, the problem comes from the fact that the $\hat\psi_j$ have non-overlapping supports. Thus, reconstructing $f$ is equivalent to reconstructing independantly each $f\star\psi_j$, which is not possible.

\nl

However, in general, the $\hat\psi_j$ have overlapping supports and the $f\star\psi_j$ are not independent for different values of $j$. They satisfy the following relation:
\begin{equation}\label{eq:convolution_commutes}
(f\star\psi_j)\star\psi_k=(f\star\psi_k)\star\psi_j
\quad\forall j,k\in\Z
\end{equation}
Thus, there is ``redundancy'' in the wavelet decomposition of $f$. We can hope that this redundancy compensates the loss of phase of $|f\star\psi_j|$. In the following, we show that, at least for specific wavelets, it is the case.

\subsection{Unicity theorem for Cauchy wavelets\label{ss:unicity_theorem}}

In this paragraph, we consider wavelets of the following form:
\begin{gather}\label{eq:def_wavelets}
\hat\psi(\omega)=\rho(\omega)\omega^pe^{-\omega}1_{\omega>0}\\
\hat\psi_j(\omega)=\hat\psi(a^j\omega)
\quad \forall \omega\in\R
\nonumber
\end{gather}
where $p> 0$ and $\rho\in L^\infty(\R)$ is such that $\rho(a\omega)=\rho(\omega)$ for almost every $\omega\in\R$ and $\rho(\omega)\ne 0,\forall\omega$.

The presence of $\rho$ allows some flexibility in the choice of the family. In particular, if it is properly chosen, the Littlewood-Paley condition \eqref{eq:littlewood_paley} may be satisfied. However, the proofs are the same with or without $\rho$.

\nl
When $\rho=1$, the wavelets of the form \eqref{eq:def_wavelets} are called \textit{Cauchy wavelets of order $p$}. The figure \ref{fig:cauchy_wavelets} displays an example of such wavelets. For these wavelets, the wavelet transform has the property to be a set of sections of a holomorphic function along horizontal lines.

If $f\in L^2(\R)$, its analytic part $f_+$ is defined by:
\begin{equation}\label{eq:def_analytic_part}
\hat f_+(\omega)=2\hat f(\omega)1_{\omega>0}
\end{equation}
We define:
\begin{equation}\label{eq:def_holomorphic_extension}
F(z)=\frac{1}{2\pi}\int_\R \omega^p\hat f_+(\omega)e^{i\omega z}d\omega
\quad\quad
\forall z\mbox{ s.t. }\Im z>0
\end{equation}
When $f_+$ is sufficiently regular, $F$ is the holomorphic extension of its $p$-th derivative.

For each $y>0$, if we denote by $F(.+iy)$ the function $x\in\R\to F(x+iy)$:
\begin{equation*}
F(.+iy)=\mathcal{F}^{-1}\left(2\omega^p\hat f(\omega)1_{\omega>0}e^{-y\omega} \right)
\end{equation*}
Consequently, for each $j\in\Z$:
\begin{equation}
\frac{a^{pj}}{2}F(.+ia^j)=f\star\psi_j\quad\forall j\in\Z
\label{eq:F_f_star_psi}
\end{equation}
So $f\star\psi_j$ is the restriction of $F$ to the horizontal line $\R+ia^j$. In this case, the relation \eqref{eq:convolution_commutes} is equivalent to the fact that, for all $j,k$, $f\star\psi_j$ and $f\star\psi_k$ are the restrictions of the \emph{same} holomorphic function to the lines $\R+ia^j$ and $\R+ia^k$.

Reconstructing $f_+$ from $\{|f\star\psi_j|\}_{j\in\Z}$ now amounts to reconstruct the holomorphic function $F:\H=\{z\in\C,\Im z>0\}\to\C$ from its modulus on an infinite set of horizontal lines. The figure \ref{fig:lines} shows these lines for $a=2$. Our phase retrieval problem thus reduces to a harmonic analysis problem. Actually, knowing $|F|$ on only two lines is already enough to recover $F$ and one of the two lines may even be $\R$, the boundary of $\H$.

\begin{figure}
\begin{minipage}[b]{0.45\textwidth}
\psfrag{x}[t][b]{$\omega$}
\includegraphics[width=\textwidth]{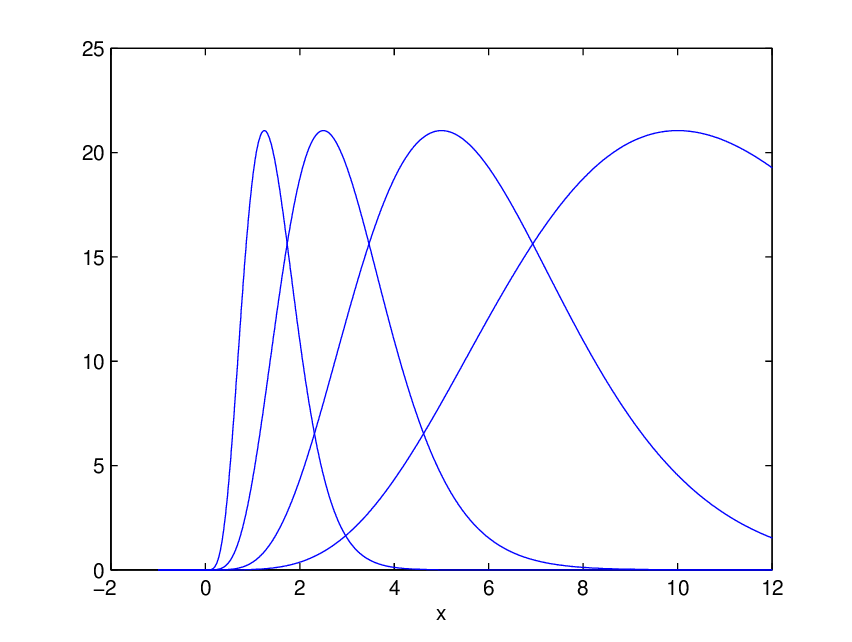}
\caption{\label{fig:cauchy_wavelets}Cauchy wavelets of order $p=5$ for $j=2,1,0,-1$, $a=2$}
\end{minipage}
\hskip 0.3cm
\begin{minipage}[b]{0.45\textwidth}
\psfrag{r}[t][b]{\small Real line}
\psfrag{ir}[b][t]{\small Imaginary line}
\includegraphics[width=\textwidth]{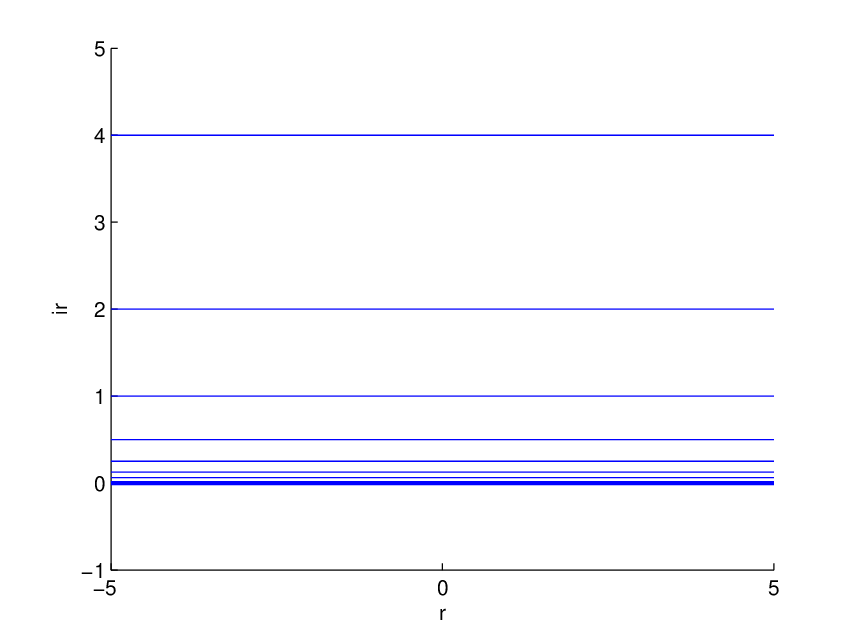}
\caption{\label{fig:lines}Lines in $\C$ over which $|F|$ is known (for $a=2$)}
\end{minipage}
\end{figure}

\begin{thm}\label{thm:unicity_hol}
Let $\alpha>0$ be fixed. Let $F,G:\H\to\C$ be holomorphic functions such that, for some $M>0$:
\begin{equation}\label{eq:L2_uniform}
\int_\R|F(x+iy)|^2dx<M
\quad\mbox{and}\quad
\int_\R|G(x+iy)|^2dx<M
\quad\quad
\forall y>0
\end{equation}
We suppose that:
\begin{gather*}
|F(x+i\alpha)|=|G(x+i\alpha)|\mbox{ for a.e. }x\in\R\\
\underset{y\to 0^+}{\lim}|F(x+iy)|=\underset{y\to 0^+}{\lim}|G(x+iy)|
\mbox{ for a.e. }x\in\R
\end{gather*}

Then, for some $\phi\in\R$:
\begin{equation}\label{eq:unicity_hol}
F=e^{i\phi} G
\end{equation}
\end{thm}
The proof is given in section \ref{ss:dem}.

\begin{cor}\label{cor:unicity_wavelets}
We consider wavelets $(\psi_j)_{j\in\Z}$ of the form \eqref{eq:def_wavelets}. Let $f,g\in L^2(\R)$ be such that, for some $j,k\in\Z$ with $j\ne k$:
\begin{equation}\label{eq:equality_moduli}
|f\star\psi_j|=|g\star\psi_j|
\quad\mbox{and}\quad
|f\star\psi_k|=|g\star\psi_k|
\end{equation}
We denote by $f_+$ and $g_+$ the analytic parts of $f$ and $g$ (as defined in \eqref{eq:def_analytic_part})

There exists $\phi\in\R$ such that:
\begin{equation}
f_+=e^{i\phi} g_+
\label{eq:equality_pos_freq}
\end{equation}
\end{cor}
\begin{proof}
We may assume that $j<k$. We define $F$ and $G$ as in \eqref{eq:def_holomorphic_extension}, with the additional $\rho$:
\begin{equation*}
F(z)=\frac{1}{2\pi}\int_\R\omega^p\rho(\omega)\hat f_+(\omega)e^{i\omega z}d\omega
\quad\quad
G(z)=\frac{1}{2\pi}\int_\R\omega^p\rho(\omega)\hat g_+(\omega)e^{i\omega z}d\omega
\quad\quad
\forall z\in\H
\end{equation*}
For each $y>0$, $F(.+iy)=\mathcal{F}^{-1}(2\omega^p\rho(\omega)e^{-y\omega}1_{\omega>0}\hat f(\omega))$. For $y=a^j$ and $y=a^k$, it implies $F(.+ia^j)=\frac{2}{a^{jp}}f\star\psi_j$ and $F(.+ia^k)=\frac{2}{a^{kp}}f\star\psi_k$. From \eqref{eq:equality_moduli}:
\begin{gather*}
|F(.+ia^j)|=\frac{2}{a^{pj}}|f\star\psi_j|=\frac{2}{a^{pj}}|g\star\psi_j|=|G(.+ia^j)|\\
|F(.+ia^k)|=\frac{2}{a^{pk}}|f\star\psi_k|=\frac{2}{a^{pk}}|g\star\psi_k|=|G(.+ia^k)|
\end{gather*}
So the functions $F(.+ia^j)$ and $G(.+ia^j)$ coincide in modulus on two horizontal lines: $\R$ and $\R+i(a^k-a^j)$. From theorem \ref{thm:unicity_hol}, they are equal up to a global phase. As $\rho$ does not vanish, it implies that $f_+$ and $g_+$ are equal up to this global phase.

So that we can apply theorem \ref{thm:unicity_hol}, we must verify that the condition \eqref{eq:L2_uniform} holds for $F(.+ia^j)$ and $G(.+ia^j)$. For any $y>a^j$:
\begin{align*}
F(.+iy)&=\mathcal{F}^{-1}\left(2\omega^p\rho(\omega)\hat f(\omega)e^{-y\omega}\right)\\
\Rightarrow \quad||F(.+iy)||_2^2&=\frac{1}{2\pi}||2\omega^p\rho(\omega)\hat f(\omega)e^{-y\omega}1_{\omega\geq 0}||_2^2\\
&\leq\frac{1}{2\pi}||2\omega^p\rho(\omega)\hat f(\omega)e^{-a^j\omega}1_{\omega\geq 0}||_2^2\\
&=\left(\frac{2}{a^{jp}}\right)^2||f\star\psi_j||_2^2
\end{align*}
The same inequality holds for $G$: the condition \eqref{eq:L2_uniform} is true for $M=\left(\frac{2}{a^{jp}}\right)^2||f\star\psi_j||_2^2$.
\end{proof}

We have just proved that the modulus of the wavelet transform uniquely determines, up to a global phase, the analytic part of a function, that is its positive frequencies. On the contrary, as wavelets are analytic ($\hat \psi_j(\omega)=0$ if $\omega<0$), the wavelet transform contains no information about the negative frequencies. In practice, signals are often real so negative frequencies are determined by positive ones and this latter limitation is not really important.
\begin{cor}
Let $f,g\in L^2(\R)$ be real-valued functions; $f_+$ and $g_+$ are their analytic parts. We assume that, for some $j,k\in\Z$ such that $j\ne k$:
\begin{equation*}
|f\star\psi_j|=|g\star\psi_j|\quad\mbox{and}\quad
|f\star\psi_k|=|g\star\psi_k|
\end{equation*}
Then, for some $\phi\in\R$:
\begin{equation*}
f_+=e^{i\phi}g_+\quad
\Leftrightarrow
\quad
f=\Re(e^{i\phi}g_+)
\end{equation*}
\end{cor}

\begin{rem}
Although the corollary \ref{cor:unicity_wavelets} holds for only two wavelets and does not require $|f\star\psi_s|=|g\star\psi_s|$ for each $s\in\Z$, the reconstruction of $f$ from only two components, $|f\star\psi_j|$ and $|f\star\psi_k|$, is very unstable in practice. Indeed, $\hat\psi_j$ and $\hat\psi_k$ are concentrated around characteristic frequencies of order $2^{-j}$ and $2^{-k}$. Thus, from $f\star\psi_j$ and $f\star\psi_k$ (and even more so from $|f\star\psi_j|$ and $|f\star\psi_k|$), reconstructing the frequencies of $f$ which are not close to $2^{-j}$ or $2^{-k}$ is numerically impossible. It is an ill-conditioned deconvolution problem.
\end{rem}

Before ending this section, let us note that, with a proof similar to the one of the corollary \ref{cor:unicity_wavelets}, the theorem \ref{thm:unicity_hol} also implies the following result.
\begin{cor}\label{cor:unicity_alpha}
Let $\alpha>0$ be fixed. Let $f,g\in L^2(\R)$ be such that $f(\omega)=g(\omega)=0$ for every $\omega<0$.

If $|\hat f|=|\hat g|$ and $|\widehat {f(t)e^{-\alpha t}}|=|\widehat {g(t)e^{-\alpha t}}|$, then, for some $\phi\in\R$:
\begin{equation*}
f=e^{i\phi} g
\end{equation*}
\end{cor}
This says that there is unicity in the phase retrieval problem associated to the masked Fourier transform, in the case where there are two masks, $t\to 1$ and $t\to e^{-\alpha t}$.

\subsection{Discrete case}

Naturally, the functions we have to deal with in practice are generally not in $L^2(\R)$. They are instead discrete finite signals. In this section, we explain how to switch from the continuous to the discrete finite setting. As we will see, all results derived in the continuous case have a discrete equivalent but proofs become simpler because they use polynomials instead of holomorphic functions.

Let $f\in\C^n$ be a discrete function. We assume $n$ is even. The discrete Fourier transform of $f$ is:
\begin{equation*}
\hat f[k]=\underset{s=0}{\overset{n-1}{\sum}}f[s]e^{-\frac{2\pi isk}{n}}
\quad\quad
\mbox{for }k=-\frac{n}{2}+1,...,\frac{n}{2}
\end{equation*}
The analytic part of $f$ is $f_+\in\C^n$ such that:
\begin{align*}
\hat f_+[k]&=0 \mbox{ if }-\frac{n}{2}+1\leq k<0\\
\hat f_+[k]&=\hat f[k] \mbox{ if }k=0\mbox{ or }k=\frac{n}{2}\\
\hat f_+[k]&=2\hat f[k] \mbox{ if }0< k<\frac{n}{2}
\end{align*}
When $f$ is real, $f=\Re(f_+)$.

We consider wavelets of the following form, for $p>0$ and $a>1$:
\begin{equation}\label{eq:def_wavelets_discrete}
\hat\psi_j[k]=\rho(a^jk)(a^jk)^pe^{-a^jk}1_{k\geq 0}
\quad
\mbox{for all }j\in\Z,k=-\frac{n}{2}+1,...,\frac{n}{2}
\end{equation}
where $\rho:\R^+\to\C$ is such that $\rho(ax)=\rho(x)$ for every $x$ and $\rho$ does not vanish.

\nl
As in the continuous case, the set $\{|f\star\psi_j|\}_{j\in\Z}$ almost uniquely determines $f_+$. Naturally, the global phase still cannot be determined. The mean value of $f_+$ can also not be determined, because $\hat\psi_j[0]=0$ for all $j$. To determine the mean value and the global phase, we would need some additional information, for example the value of $f\star\phi$ for some low frequency signal $\phi$.

\begin{thm}[Discrete version of \ref{cor:unicity_wavelets}]
\label{thm:unicity_wavelets_discrete}
Let $f,g\in\C^n$ be discrete signals and $(\psi_j)_{j\in\Z}$ a family of wavelets of the form \eqref{eq:def_wavelets_discrete}. Let $j,l\in\Z$ be two distinct integers. Then:
\begin{equation}\label{eq:equality_moduli_discrete}
|f\star\psi_j|=|g\star\psi_j|
\quad\mbox{and}\quad
|f\star\psi_l|=|g\star\psi_l|
\end{equation}
if and only if, for some $\phi\in\R,c\in\C$:
\begin{equation*}
f_+=e^{i\phi}g_++c
\end{equation*}
\end{thm}

\begin{proof}
We first assume $f_+=e^{i\phi}g_++c$. Taking the Fourier transform of this equality yields:
\begin{equation*}
\hat f[k]=e^{i\phi}\hat g[k]\quad\quad
\mbox{for all }k=1,...,\frac{n}{2}
\end{equation*}
As $\hat\psi_j[k]=0$ for $k=-\frac{n}{2}+1,...,0$:
\begin{gather*}
\hat f[k]\hat\psi_j[k]=e^{i\phi}\hat g[k]\hat\psi_j[k]\quad\quad
\mbox{for all }k=-\frac{n}{2}+1,...,\frac{n}{2}\\
\Rightarrow\quad (f\star\psi_j=e^{i\phi}(g\star\psi_j))
\end{gather*}
So $|f\star\psi_j|=|g\star\psi_j|$ and, similarly, $|f\star\psi_l|=|g\star\psi_l|$.

\nl
We now suppose conversely that $|f\star\psi_j|=|g\star\psi_j|$ and $|f\star\psi_l|=|g\star\psi_l|$. We define:
\begin{equation*}
F(z)=\frac{1}{n}\,\underset{k=1}{\overset{n/2}{\sum}}\hat f[k]\rho(k)k^pz^k\quad\quad
G(z)=\frac{1}{n}\,\underset{k=1}{\overset{n/2}{\sum}}\hat g[k]\rho(k)k^pz^k\quad\quad
\forall z\in\C
\end{equation*}
These polynomials are the discrete equivalents of functions $F$ and $G$ used in the proof of \ref{cor:unicity_wavelets}. For all $s=-\frac{n}{2}+1,...,\frac{n}{2}$:
\begin{align*}
F(e^{-a^j}e^{\frac{2\pi is}{n}})&=\frac{1}{n}\,\underset{k=1}{\overset{n/2}{\sum}}\hat f[k]\rho(k)k^pe^{-a^jk}e^{\frac{2\pi iks}{n}}\\
&=a^{-jp}\frac{1}{n}\,\underset{k=-n/2+1}{\overset{n/2}{\sum}}\hat f[k]\hat\psi_j[k]e^{\frac{2\pi iks}{n}}\\
&=a^{-jp}\left(f\star\psi_j[s]\right)
\end{align*}
Similarly, $G(e^{-a^j}e^{\frac{2\pi is}{n}})=a^{-jp}(g\star\psi_j[s])$ for all $s=-\frac{n}{2}+1,...,\frac{n}{2}$.

Thus, $f\star\psi_j$ and $g\star\psi_j$ can be seen as the restrictions of $F$ and $G$ to the circle of radius $e^{-a^j}$. This is similar to the continuous case, where $f\star\psi_j$ and $g\star\psi_j$ were the restrictions of functions $F,G$ to horizontal lines.

The equality \eqref{eq:equality_moduli_discrete} implies:
\begin{gather*}
\left|F(e^{-a^j}e^{\frac{2\pi is}{n}})\right|^2=\left|G(e^{-a^j}e^{\frac{2\pi is}{n}})\right|^2
\quad\quad
\mbox{for all }s=-\frac{n}{2}+1,...,\frac{n}{2}\\
\Leftrightarrow
F(e^{-a^j}e^{\frac{2\pi is}{n}})\overline{F}(e^{-a^j}e^{-\frac{2\pi is}{n}})
=G(e^{-a^j}e^{\frac{2\pi is}{n}})\overline{G}(e^{-a^j}e^{-\frac{2\pi is}{n}})
\quad\quad
\mbox{for all }s=-\frac{n}{2}+1,...,\frac{n}{2}\\
\end{gather*}
The functions $z\to F(e^{-a^j}z)\overline{F}(e^{-a^j}\frac{1}{z})$ and $z\to G(e^{-a^j}z)\overline{G}(e^{-a^j}\frac{1}{z})$ are polynomials of degree $n-2$ (up to multiplication by $z^{n/2-1}$). They share $n$ common values so they are equal. The same is true for $l$ instead of $j$ so:
\begin{gather}
F(e^{-a^j}z)\overline{F}\left(e^{-a^j}\frac{1}{z}\right)= G(e^{-a^j}z)\overline{G}\left(e^{-a^j}\frac{1}{z}\right)\quad\quad\forall z\in\C\label{eq:eq_poly_1}\\
F(e^{-a^l}z)\overline{F}\left(e^{-a^l}\frac{1}{z}\right)= G(e^{-a^l}z)\overline{G}\left(e^{-a^l}\frac{1}{z}\right)\quad\quad\forall z\in\C\label{eq:eq_poly_2}
\end{gather}
If we show that these equalities imply $F=e^{i\phi}G$ for some $\phi\in\R$, the proof will be finished. Indeed, from the definition of $F$ and $G$, we will then have $\hat f[k]=e^{i\phi}\hat g[k]$ for all $k=1,...,\frac{n}{2}$ so $\hat f_+[k]=e^{i\phi}\hat g_+[k]$ for all $k\ne 0$. It implies $f_+=e^{i\phi}g_++c$ for $c=\frac{1}{n}\left(\hat f_+[0]-e^{i\phi}\hat g_+[0]\right)$.

It suffices to show that $F$ and $G$ have the same roots (with multiplicity) because then, they will be proportional and, from \eqref{eq:eq_poly_1}, \eqref{eq:eq_poly_2}, the proportionality constant must be of modulus $1$.

For each $z\in\C$, let $\mu_F(z)$ (resp. $\mu_G(z)$) be the multiplicity of $z$ as a root of $F$ (resp. $G$). The polynomials of \eqref{eq:eq_poly_1} are of respective degree $n-2\mu_F(0)$ and $n-2\mu_G(0)$ so $\mu_F(0)=\mu_G(0)$.

For all $z\ne 0$, the multiplicity of $e^{a^j}z$ as a zero of \eqref{eq:eq_poly_1} is:
\begin{equation*}
\mu_F(z)+\mu_F\left(\frac{e^{-2a^j}}{\overline{z}}\right)
=\mu_G(z)+\mu_G\left(\frac{e^{-2a^j}}{\overline{z}}\right)
\end{equation*}
and the multiplicity of $e^{2a^j-a^l}z$ as a zero of \eqref{eq:eq_poly_2} is:
\begin{equation*}
\mu_F(e^{2(a^j-a^l)} z)+\mu_F\left(\frac{e^{-2a^j}}{\overline{z}}\right)
=\mu_G(e^{2(a^j-a^l)} z)+\mu_G\left(\frac{e^{-2a^j}}{\overline{z}}\right)
\end{equation*}
Substracting this last equality to the previous one implies that, for all $z$:
\begin{equation*}
\mu_F(z)-\mu_G(z)=\mu_F(e^{2(a^j-a^l)}z)-\mu_G(e^{2(a^j-a^l)}z)
\end{equation*}
By applying this equality several times, we get, for all $n\in\N$:
\begin{align*}
\mu_F(z)-\mu_G(z)&=\mu_F(e^{2(a^j-a^l)}z)-\mu_G(e^{2(a^j-a^l)}z)\\
&=\mu_F(e^{4(a^j-a^l)}z)-\mu_G(e^{4(a^j-a^l)}z)\\
&=...\\
&=\mu_F(e^{2n(a^j-a^l)}z)-\mu_G(e^{2n(a^j-a^l)}z)
\end{align*}
As $F$ and $G$ have a finite number of roots, $\mu_F(e^{2n(a^j-a^l)}z)-\mu_G(e^{2n(a^j-a^l)}z)=0$ if $n$ is large enough. So $\mu_F(z)=\mu_G(z)$ for all $z\in\C$.
\end{proof}

As in the section \ref{ss:unicity_theorem}, a very similar proof gives a unicity result for the case of the Fourier transform with masks, if the masks are well-chosen.
\begin{thm}[Discrete version of \ref{cor:unicity_alpha}]
Let $\alpha>0$ be fixed. Let $f,g\in\C^{2n-1}$ be two discrete signals with support in $\{0,...,n-1\}$:
\begin{equation*}
f[s]=g[s]=0\mbox{ for }s=n,...,2n-2
\end{equation*}
If $|\hat f|=|\hat g|$ and $|\widehat{f[s]e^{-s\alpha}}|=|\widehat{g[s]e^{-s\alpha}}|$, then, for some $\phi\in\R$:
\begin{equation*}
f=e^{i\phi}g
\end{equation*}
\end{thm}
Remark that this theorem describes systems of $4n-2$ linear measurements whose moduli are enough to recover each complex signal of dimension $n$. As discussed in the introduction, it is known that $4n-4$ generic measurements always achieve this property (\citep{conca}). However, it is in general difficult to find deterministic systems for which it can be proven.

\subsection{Proof of theorem \ref{thm:unicity_hol}\label{ss:dem}}

\begin{thm*}[\ref{thm:unicity_hol}]
Let $\alpha>0$ be fixed. Let $F,G:\H\to\C$ be holomorphic functions such that, for some $M>0$:
\begin{equation*}
\tag{\ref{eq:L2_uniform}}
\int_\R|F(x+iy)|^2dx<M
\quad\mbox{and}\quad
\int_\R|G(x+iy)|^2dx<M
\quad\quad
\forall y>0
\end{equation*}
We suppose that:
\begin{gather*}
|F(x+i\alpha)|=|G(x+i\alpha)|\mbox{ for a.e. }x\in\R\\
\underset{y\to 0^+}{\lim}|F(x+iy)|=\underset{y\to 0^+}{\lim}|G(x+iy)|
\mbox{ for a.e. }x\in\R
\end{gather*}

Then, for some $\phi\in\R$:
\begin{equation}\label{eq:unicity_hol}
F=e^{i\phi} G
\end{equation}
\end{thm*}

\begin{proof}[Proof of theorem \ref{thm:unicity_hol}]
This demonstration relies on the ideas used by \citet{akutowicz}.

If $F= 0$, the theorem is true: $G$ is null over a whole line and, as $G$ is holomorphic, $G=0$. The same reasoning holds if $G=0$. We now assume $F\ne 0,G\ne 0$.

The central point of the proof is to factorize the functions $F,F(.+i\alpha),G,G(.+i\alpha)$ as in the following lemma.
\begin{lem}{\rm\citep{kryloff}\footnote{Non russian speaking readers may also deduce this theorem from \citet[Thm 17.17]{rudin}: functions over $\H$ may be turned into functions over $D(0,1)$ by composing them with the conformal application $z\in D(0,1)\to\frac{1-z}{1+z}i\in\H$. The main difficulty is to show that if $H:\H\to\C$ satisfies \eqref{eq:L2_uniform}, then $\tilde H:z\in D(0,1)\to H\left(\frac{1-z}{1+z}i\right)\in\C$ is of class $H^2$ and Rudin's theorem can be applied.}}\label{lem:factorization}
The function $F$ admits the following factorization:
\begin{equation*}
F(z)=e^{ic+i\beta z}B(z)D(z)S(z)
\end{equation*}
Here, $c$ and $\beta$ are real numbers. The function $B$ is a Blaschke product. It is formed with the zeros of $F$ in the upper half-plane $\H$. We call $(z_k)$ these zeros, counted with multiplicity, with the exception of $i$. We call $m$ the multiplicity of $i$ as zero.
\begin{equation}\label{eq:def_blaschke}
  B(z)=\left(\frac{z-i}{z+i}\right)^{m}\underset{k}{\prod}\frac{|z_k-i|}{z_k-i}\frac{|z_k+i|}{z_k+i}\frac{z-z_k}{z-\overline{z}_k}
\end{equation}
This product converges over $\H$, which is equivalent to:
\begin{equation}\label{eq:condition_convergence}
\underset{k}{\sum}\frac{\Im z_k}{1+|z_k|^2}<+\infty
\end{equation}

The functions $D$ and $S$ are defined by:
\begin{gather}
D(z)=\exp\left(\frac{1}{\pi i}\int_{\R}\frac{1+tz}{t-z}\frac{\log|F(t)|}{1+t^2}dt\right)\label{eq:def_D}\\
S(z)=\exp\left(\frac{i}{\pi}\int_{\R}\frac{1+tz}{t-z}dE(t)\right)\label{eq:def_S}
\end{gather}
In the first equation, $|F(t)|$ is the limit of $|F|$ on $\R$. In the second one, $dE$ is a positive bounded measure, singular with respect to Lebesgue measure.

Both integrals converge absolutely for any $z\in\H$.
\end{lem}

The same factorization can be applied to $F(.+i\alpha),G$ and $G(.+i\alpha)$:
\begin{equation*}
\begin{array}{cc}
F(z)=e^{ic_F+i\beta_F z}B_F(z)D_F(z)S_F(z)&
\quad
G(z)=e^{ic_G+i\beta_G z}B_G(z)D_G(z)S_G(z)\\
F(z+i\alpha)=e^{i\tilde c_F+i\tilde\beta_F z}\tilde B_F(z)\tilde D_F(z)\tilde S_F(z)&
\quad
G(z+i\alpha)=e^{i\tilde c_G+i\tilde\beta_G z}\tilde B_G(z)\tilde D_G(z)\tilde S_G(z)
\end{array}
\end{equation*}

As $F(.+i\alpha)$ and $G(.+i\alpha)$ are analytic on the real line, they actually have no singular part $S$. The proof may be found in \citep[Thm 6.3]{garnett}; it is done for functions on the unit disk but also holds for functions on $\H$.
\begin{equation}\label{eq:tilde_S_1}
\tilde S_F=\tilde S_G=1
\end{equation}

Because $\underset{y\to 0^+}{\lim}|F(.+iy)|=\underset{y\to 0^+}{\lim}|G(.+iy)|$ and $|F(.+i\alpha)|=|G(.+i\alpha)|$, we have $D_F=D_G$ and $\tilde D_F=\tilde D_G$. We show that it implies a relation between the $B$'s, that is, a relation between the zeros of $F$ and $G$. From this relation, we will be able to prove that $F$ and $G$ have the same zeros and that, up to a global phase, they are equal.

For all $z\in\H$:
\begin{align*}
\frac{e^{ic_F+i\beta_F(z+i\alpha)}B_F(z+i\alpha)D_F(z+i\alpha)S_F(z+i\alpha)}{e^{i\tilde c_F+i\tilde\beta_F z}\tilde B_F(z)\tilde D_F(z)}
&=\frac{F(z+i\alpha)}{F(z+i\alpha)}=1\\
&=\frac{G(z+i\alpha)}{G(z+i\alpha)}\\
&=\frac{e^{ic_G+i\beta_G(z+i\alpha)}B_G(z+i\alpha)D_G(z+i\alpha)S_G(z+i\alpha)}{e^{i\tilde c_G+i\tilde\beta_G z}\tilde B_G(z)\tilde D_G(z)}
\end{align*}
\begin{gather}
\Rightarrow \quad
\frac{B_F(z+i\alpha)\tilde B_G(z)}{B_G(z+i\alpha)\tilde B_F(z)}
=e^{iC+iBz}\frac{S_G(z+i\alpha)}{S_F(z+i\alpha)}\label{eq:relation_zeros}
\\
\mbox{for some }C,B\in\R\nonumber
\end{gather}
Equality \eqref{eq:relation_zeros} holds only for $z\in\H$. It is a priori not even defined for $z\in\C-\H$. Before going on, we must show that \eqref{eq:relation_zeros} is meaningful and still valid over all $\C$. This is the purpose of the two following lemmas, whose proofs may be found in appendix \ref{app:lemmas_unicity_hol}.

For $z\in\H$, we denote by $\mu_F(z)$ (resp. $\mu_G(z)$) the multiplicity of $z$ as a zero of $F$ (resp. $G$).

\begin{lem}\label{lem:prolongation}
There exists a meromorphic function $B_w:\C\to\C$ such that:
\begin{equation*}
B_w(z)=\frac{B_F(z+i\alpha)\tilde B_G(z)}{B_G(z+i\alpha)\tilde B_F(z)}\quad\quad\forall z\in\H
\end{equation*}

Moreover, for all $z\in\H$, the multiplicity of $\overline{z}-i\alpha$ as a pole of $B_w$ is:
\begin{equation}\label{eq:formule_mult}
(\mu_F(z)-\mu_G(z))-(\mu_F(z+2i\alpha)-\mu_G(z+2i\alpha))
\end{equation}
\end{lem}
\begin{lem}\label{lem:S_1}
For all $z\in\H$, $\frac{S_G(z+i\alpha)}{S_F(z+i\alpha)}=1$.
\end{lem}

The equation \eqref{eq:relation_zeros} and the lemmas \ref{lem:prolongation} and \ref{lem:S_1} give, for all $z\in\H$ and thus all $z\in\C$ (because functions are meromorphic):
\begin{equation*}
B_w(z)=e^{iC+iBz}\quad\quad\forall z\in\C
\end{equation*}
The function $e^{iC+iB z}$ has no zero nor pole so, from \eqref{eq:formule_mult}, for all $z\in\H$:
\begin{equation*}
(\mu_F(z)-\mu_G(z))-(\mu_F(z+2i\alpha)-\mu_G(z+2i\alpha))
=0
\end{equation*}

So if $\mu_F(z)\ne\mu_G(z)$ for some $z$, we may by symmetry assume that $\mu_F(z)>\mu_G(z)$ and, in this case, for all $n\in\N^*$:
\begin{align*}
\mu_F(z+2ni\alpha)-\mu_G(z+2ni\alpha)
&=...\\
&=\mu_F(z+2i\alpha)-\mu_G(z+2i\alpha)\\
&=\mu_F(z)-\mu_G(z)>0
\end{align*}
In particular, $z+2ni\alpha$ is a zero of $F$ for all $n\in\N^*$. But this is impossible because, if it is the case, $\frac{\mbox{\footnotesize Im} (z+2ni\alpha)}{1+|z+2ni\alpha|^2}\sim\frac{1}{2n\alpha}$ and:
\begin{equation*}
\underset{k}{\sum}\frac{\Im z_k}{1+|z_k|^2}=+\infty
\end{equation*}
where the $(z_k)$ are the zeros of $F$ over $\H$. It is in contradiction with \eqref{eq:condition_convergence}.

So for all $z\in\H$, $\mu_F(z)=\mu_G(z)$. This implies that $B_F=B_G$ and $\tilde B_F=\tilde B_G$. So, for all $z\in\H$:
\begin{gather*}
F(z+i\alpha)=e^{i\tilde c_F+i\tilde\beta_Fz}\tilde B_F(z)\tilde D_F(z)
=e^{i\tilde c_F+i\tilde\beta_Fz}\tilde B_G(z)\tilde D_G(z)
=e^{i\gamma+i\delta z} G(z+i\alpha)\\
\mbox{with }\gamma=\tilde c_F-\tilde c_G\mbox{ and }
\delta=\tilde\beta_F-\tilde\beta_G
\end{gather*}
The functions $F$ and $G$ are meromorphic over $\H$ so the last equality actually holds over all $\{z\in\C\mbox{ s.t. }\Im z>-\alpha\}$.
\begin{align*}
|\underset{y\to 0^+}{\lim}F(x+iy)|&=|\underset{y\to 0^+}{\lim}e^{i\gamma+i\delta(x+iy-i\alpha)} G(x+iy)|\\
&=e^{\delta\alpha}|\underset{y\to 0^+}{\lim} G(x+iy)|
\end{align*}
Consequently, because $\delta$ is real and $\alpha\ne 0$, $\delta=0$. So:
\begin{equation*}
F(z)=e^{i\gamma} G(z)\quad\quad
\forall z\in\H
\end{equation*}
\end{proof}

\section{Weak stability of the reconstruction\label{s:weak_stability}}

In the previous section, we proved that the operator $U:f\to\{|f\star\psi_j|\}$ was injective, up to a global phase, for Cauchy wavelets. So we can theoretically reconstruct any function $f$ from $U(f)$. However, if we want the reconstruction to be possible in practice, we also need it to be stable to a small amount of noise:
\begin{equation*}
\left(U(f_1)\approx U(f_2)\right)
\quad\Rightarrow\quad
\left(f_1\approx f_2\right)
\end{equation*}

In this section, we show that it is, in some sense, the case: $U^{-1}$ is continuous.

Contrarily to the ones of the previous section, this result is not specific to Cauchy wavelets: it holds for all reasonable wavelets, as soon as $U$ is injective.



\subsection{Definitions}

As in the previous section, we consider only functions without negative frequencies:
\begin{equation*}
L^2_+(\R)=\{f\in L^2(\R)\mbox{ s.t. }\hat f(\omega)=0\mbox{ for a.e. }\omega<0\}
\end{equation*}
As the reconstruction is always up to a global phase, we need to define the quotient $L^2_+(\R)/S^1$:
\begin{equation*}
f=g\mbox{ in }L^2_+(\R)/S^1
\quad\Leftrightarrow\quad
f=e^{i\phi}g\mbox{ for some }\phi\in\R
\end{equation*}
The set $L^2_+(\R)/S^1$ is equipped with a natural metric:
\begin{equation*}
||f-g||_{2,S^1}=\underset{\phi\in \R}{\inf}\,||f-e^{i\phi} g||_2
\end{equation*}

\noindent We also define:
\begin{gather*}
L^2_{\Z}(\R)=\left\{(h_j)_{j\in\Z}\in L^2(\R)^\Z\mbox{ s.t. }\underset{j}{\sum}||h_j||_2^2<+\infty\right\}\\
\left|\left|(h_j)-(h'_j) \right|\right|_2=\sqrt{\underset{j\in\Z}{\sum}||h_j-h'_j||_2^2}\quad
\mbox{for any }(h_j),(h'_j)\in L^2_\Z(\R)
\end{gather*}
We are interested in the operator $U$:
\begin{equation}\label{eq:def_U}
\begin{array}{cccc}
U:&L^2_+(\R)/S^1&\to& L^2_{\Z}(\R)\\
&f&\to&(|f\star\psi_j|)_{j\in\Z}
\end{array}
\end{equation}
We require two conditions over the wavelets. They must be analytic:
\begin{equation}\label{eq:analyticity}
\hat\psi_j(\omega)=0 \mbox{ for a.e. }\omega<0,j\in\Z
\end{equation}
and satisfy an approximate Littlewood-Paley inequality:
\begin{equation}\label{eq:approx_littlewood_paley}
A\leq\underset{j\in\Z}{\sum}|\hat\psi_j(\omega)|^2\leq B
\quad\quad
\mbox{for a.e. }\omega>0,\mbox{for some }A,B>0
\end{equation}
This last inequality implies:
\begin{equation}\label{eq:U_bounds}
\forall f\in L^2_+(\R)/S^1,\quad
\sqrt{A}||f||_{2,S^1}\leq||U(f)||_2\leq\sqrt{B}||f||_{2,S^1}
\end{equation}
In particular, it ensures the continuity of $U$.

\subsection{Weak stability theorem}

\begin{thm}\label{thm:weak_stability}
We suppose that, for all $j\in\Z$, $\psi_j\in L^1(\R)\cap L^2(\R)$ and that \eqref{eq:analyticity} and \eqref{eq:approx_littlewood_paley} hold. We also suppose that $U$ is injective. Then:

\hskip 0.5cm
(i) The image of $U$, $I_U=\{U(f)\mbox{ s.t. }f\in L^2_+(\R)/S^1\}$ is closed in $L^2_\Z(\R)$.

\hskip 0.5cm
(ii) The application $U^{-1}:I_U\to L^2_+(\R)/S^1$ is continuous.
\end{thm}
\begin{proof}

What we have to prove is the following: if $(U(f_n))_{n\in\N}$ converges towards a limit $v\in L^2_\Z(\R)$, then $v=U(g)$ for some $g\in L^2_+(\R)/S^1$ and $f_n\to g$ in $L^2_+(\R)/S^1$.

So let $(U(f_n))_{n\in\N}$ be a sequence of elements in $I_U$, which converges in $L^2_\Z(\R)$. Let $v=(h_j)_{j\in\Z}\in L^2_{\mathbb{Z}}(\R)$ be the limit. We show that $v\in I_U$.

\begin{lem}\label{lem:tot_bound}
For all $j\in\Z$, $\{f_n\star\psi_j\}_{n\in\N}$ is relatively compact in $L^2(\R)$ (that is, the closure of this set in $L^2(\R)$ is compact).
\end{lem}

The proof of this lemma is given in \ref{app:lemmas_weak_stability}. It uses the Riesz-Fr\'echet-Kolmogorov theorem, which gives an explicit characterization of the relatively compact subsets of $L^2(\R)$.

For every $j\in\Z$, $\{f_n\star\psi_j\}_{n\in\N}$ is thus included in a compact subset of $L^2(\R)$. In a compact set, every sequence admits a converging subsequence: there exists $\phi:\N\to\N$ injective such that $(f_{\phi(n)}\star\psi_j)_{n\in\N}$ converges in $L^2(\R)$. Actually, we can choose $\phi$ such that $(f_{\phi(n)}\star\psi_j)_n$ converges for any $j$ (and not only for a single one). We donote by $l_j$ the limits.

\begin{lem}[Proof in \ref{app:lemmas_weak_stability}]\label{lem:existence_g}
There exists $g\in L^2_+(\R)$ such that $l_j=g\star\psi_j$ for every $j$. Moreover, $f_{\phi(n)}\to g$ in $L^2(\R)$.
\end{lem}

As $U$ is continuous, $U(g)=\underset{n}{\lim}\,U(f_{\phi(n)})=v$. So $v$ belongs to $I_U$.

The $g$ such that $U(g)=v$ is uniquely defined in $L^2_+(\R)/S^1$ because $U$ is injective (it does not depend on the choice of $\phi$). We must now show that $f_n\to g$.

From the lemma \ref{lem:existence_g}, $(f_n)_n$ admits a subsequence $(f_{\phi(n)})$ which converges to $g$. By the same reasoning, every subsequence $(f_{\psi(n)})_{n}$ of $(f_n)_n$ admits a subsequence which converges to $g$. This implies that $(f_n)_n$ globally converges to $g$.




\end{proof}

\begin{rem}
The same demonstration gives a similar result for wavelets on $\R^d$, of the form $(\psi_{j,\gamma})_{j\in\Z,\gamma\in\Gamma}$, for $\Gamma$ a finite set of parameters.
\end{rem}

\section{The reconstruction is not uniformly continuous\label{s:non_uniform_continuity}}

The theorem \ref{thm:weak_stability} states that the operator $U:f\to\{|f\star\psi_j|\}_{j\in\Z}$ has a continuous inverse $U^{-1}$, when it is invertible. However, $U^{-1}$ is not uniformly continuous. Indeed, for any $\epsilon>0$, there exist $g_1,g_2\in L^2_+(\R)/S^1$ such that:
\begin{equation}\label{eq:def_non_uniform}
||U(g_1)-U(g_2)||<\epsilon
\quad\mbox{but}\quad
||g_1-g_2||\geq 1
\end{equation}

In this section, we describe a way to construct such ``instable'' pairs $(g_1,g_2)$: we start from any $g_1$ and modulate each $g_1\star \psi_j$ by a low-frequency phase. We then (approximately) invert this modified wavelet transform and obtain $g_2$.

 This construction seems to be ``generic'' in the sense that it includes all the instabilities that we have been able to observe in practice.

\subsection{A simple example}


To begin with, we give a simple example of instabilities and relate it to known results about the stability in general phase retrieval problems.

In phase retrieval problems with (a finite number of) real measurements, the stability of the reconstruction operator is characterized by the following theorem (\citep{bandeira_stability}, \citep{balan_wang}).
\begin{thm}
Let $A\in\R^{m\times n}$ be a measurement matrix. For any $S\subset\{1,...,m\}$, we denote by $A_S$ the matrix obtained by discarding the rows of $A$ whose indexes are not in $S$. We call $\lambda^2_S$ the lower frame bound of $A_S$, that is, the largest real number such that:
\begin{equation*}
||A_Sx||_2^2\geq\lambda^2_S||x||_2^2
\quad\quad
\forall x\in\R^n
\end{equation*}
Then, for any $x,y\in\R^n$:
\begin{equation*}
||\,|Ax|-|Ay|\,||_2\geq\left(\underset{S}{\min}\sqrt{\lambda^2_S+\lambda^2_{S^c}}\right).\min(||x-y||_2,||x+y||_2)
\end{equation*}
Moreover, $\underset{S}{\min}\sqrt{\lambda^2_S+\lambda^2_{S^c}}$ is the optimal constant.
\end{thm}
In the complex case, one can only show a weaker result.
\begin{thm}
Let $A\in\C^{m\times n}$ be a measurement matrix. There exist $x,y\in\C^n$ such that:
\begin{equation*}
||\,|Ax|-|Ay|\,||_2\leq\left(\underset{S}{\min}\sqrt{\lambda_S^2+\lambda_{S^c}^2}\right).\underset{|\eta|=1}{\min}(||x-\eta y||_2)
\end{equation*}
\end{thm}

Consequently, if the set of measurements can be divided in two parts $S$ and $S^c$ such that $\lambda_S^2$ and $\lambda_{S^c}^2$ are very small, then the reconstruction is not stable.

Such a phenomenon occurs in the case of the wavelet transform. We define:
\begin{equation*}
S=\{\psi_j\mbox{ s.t. }j\geq 0\}\mbox{ and }S^c=\{\psi_j\mbox{ s.t. }j<0\}
\end{equation*}
Let us fix a small $\epsilon>0$. We choose $f_1,f_2\in L^2(\R)$ such that:
\begin{gather*}
\hat f_1(x)=0\mbox{ if }|x|<1/\epsilon
\quad\mbox{ and }\quad
\hat f_2(x)=0\mbox{ if }x\notin[-\epsilon;\epsilon]
\end{gather*}
For every $\psi_j\in S$, $f_1\star\psi_j\approx 0$ because the characteristic frequency of $\psi_j$ is smaller than $1$ and $f_1$ is a very high frequency function. So:
\begin{equation*}
|(f_1+f_2)\star\psi_j|\approx|f_2\star\psi_j|=|-f_2\star\psi_j|\approx|(f_1-f_2)\star\psi_j|
\end{equation*}
And similarly, for $\psi_j\in S^c$, $f_2\star\psi_j\approx 0$ and:
\begin{equation*}
|(f_1+f_2)\star\psi_j|\approx|f_1\star\psi_j|\approx|(f_1-f_2)\star\psi_j|
\end{equation*}
As a consequence:
\begin{equation*}
\{|(f_1+f_2)\star\psi_j|\}_{j\in\Z}\approx\{|(f_1-f_2)\star\psi_j|\}_{j\in\Z}
\end{equation*}
Nevertheless, $f_1+f_2$ and $f_1-f_2$ may not be close in $L^2(\R)/S^1$: $g_1=f_1+f_2$ and $g_2=f_1-f_2$ satisfy \eqref{eq:def_non_uniform}.

\begin{figure}
\psfrag{Titre1}{(a)}
\psfrag{Titre2}{(b)}
\psfrag{Titre3}{(c)}
\psfrag{Titre4}{(d)}
\includegraphics[scale = 0.55]{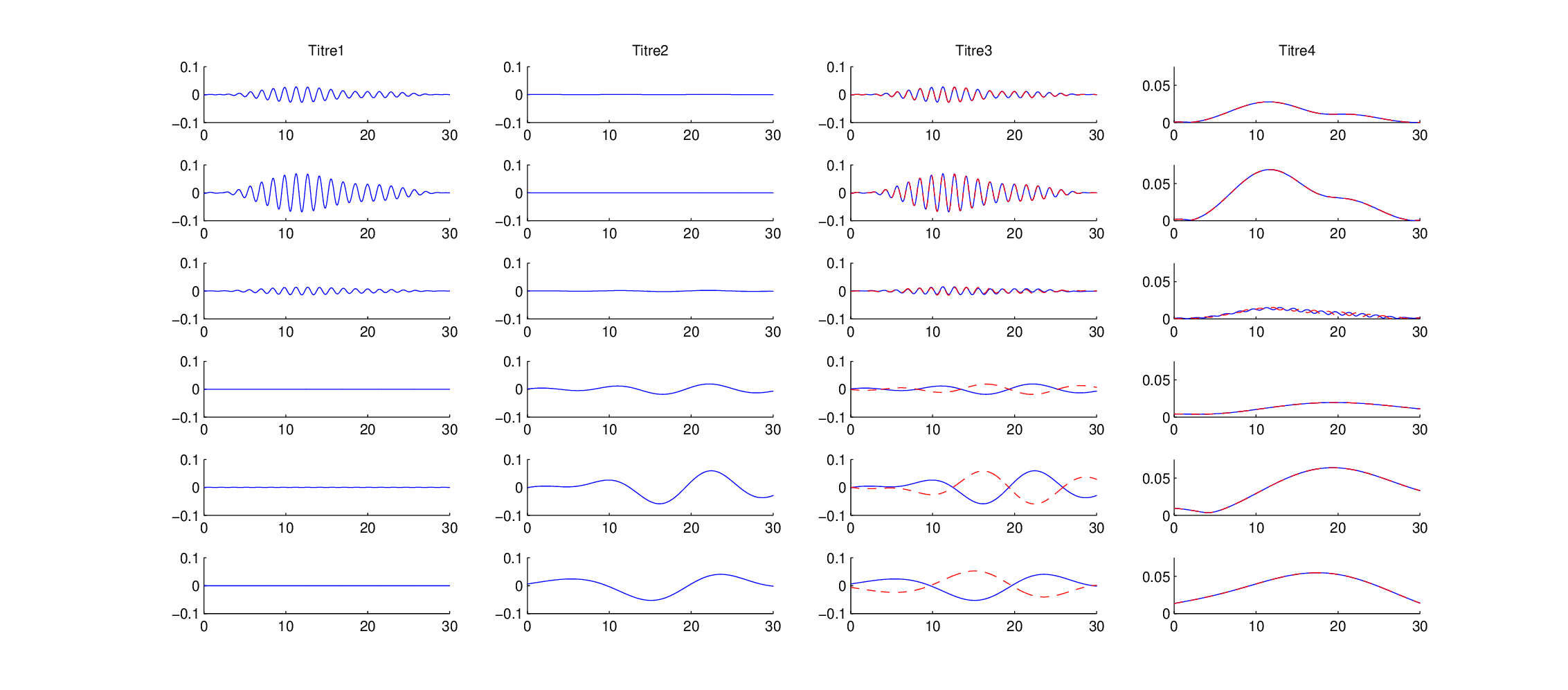}
\caption{(a) Wavelet transform of $f_1$ (b) Wavelet transform of $f_2$ (c) Wavelet transform of $f_1+f_2$ (solid blue) and $f_1-f_2$ (dashed red) (d) Modulus of the wavelet transforms of $f_1+f_2$ and $f_1-f_2$; the two modulus are almost equal\\
In each column, each graph corresponds to a specific frequency; the highest frequency is on top and the lowest one at bottom.
For complex functions, only the real part is displayed.\label{fig:inst}}
\end{figure}

The figure \ref{fig:inst} displays an example of this kind.

\subsection{A wider class of instabilities\label{ss:a_wider_class}}

We now describe the construction of more general ``instable'' pairs $(g_1,g_2)$.


\noindent Let $g_1\in L^2(\R)$ be any function. We aim at finding $g_2\in L^2(\R)$ such that, for all $j\in\Z$:
\begin{equation}\label{eq:quasi_equality_in_modulus}
(g_1\star\psi_j)e^{i\phi_j}\approx  g_2\star\psi_j
\end{equation}
for some real functions $\phi_j$.

In other words, we must find phases $\phi_j$ such that $(g_1\star\psi_j)e^{i\phi_j}$ is approximately equal to the wavelet transform of some $g_2\in L^2(\R)$. Any phases $\phi_j(t)$ which vary slowly both in $t$ and in $j$ satisfy this property.

Indeed, if the $\phi_j(t)$ vary ``slowly enough'', we set:
\begin{equation*}
g_2=\underset{j\in\Z}{\sum}\left((g_1\star\psi_j)e^{i\phi_j}\right)\star\tilde{\psi}_j
\end{equation*}
where $\{\tilde\psi_j\}_{j\in\Z}$ are the dual wavelets associated to $\{\psi_j\}$.

Then, for all $k\in\Z,t\in\R$:
\begin{align*}
g_2\star\psi_k(t)&=\underset{j\in\Z}{\sum}\left((g_1\star\psi_j)e^{i\phi_j}\right)\star\tilde\psi_j\star\psi_k(t)\\
&=\underset{j\in\Z}{\sum}\int_\R e^{i\phi_j(t-u)}(g_1\star\psi_j)(t-u)(\tilde\psi_j\star\psi_k)(u)\,du
\end{align*}
\begin{align*}
(g_1\star\psi_k(t))e^{i\phi_k(t)}&=e^{i\phi_k(t)}\underset{j\in\Z}{\sum}(g_1\star\psi_j)\star(\tilde{\psi_j}\star\psi_k)(t)\\
&=\underset{j\in\Z}{\sum}\int_\R e^{i\phi_k(t)}(g_1\star\psi_j)(t-u)(\tilde\psi_j\star\psi_k)(u)\,du
\end{align*}
So:
\begin{equation}
\label{eq:approx_error}
g_2\star\psi_k(t)-(g_1\star\psi_k(t))e^{i\phi_k(t)}
=\underset{j\in\Z}{\sum}\int_\R (e^{i\phi_j(t-u)} -e^{i\phi_k(t)})(g_1\star\psi_j)(t-u)(\tilde\psi_j\star\psi_k)(u)\,du
\end{equation}

The function $\tilde\psi_j\star\psi_k(u)$ is negligible if $j$ is not of the same order as $k$ or if $u$ is too far away from $0$. It means that, for some $C\in\N,U\in\R$ (which may depend on $k$):
\begin{equation*}
g_2\star\psi_k(t)-(g_1\star\psi_k(t))e^{i\phi_k(t)}
\approx
\underset{|j-k|\leq C}{\sum}\int_{[-U;U]} (e^{i\phi_j(t-u)} -e^{i\phi_k(t)})(g_1\star\psi_j)(t-u)(\tilde\psi_j\star\psi_k)(u)\,du
\end{equation*}
If $\phi_j(t-u)$ does not vary much over $[k-C;k+C]\times [-U;U]$, it gives the desired relation:
\begin{equation*}
g_2\star\psi_k(t)-(g_1\star\psi_k(t))e^{i\phi_k(t)}
\approx 0
\end{equation*}
which is \eqref{eq:quasi_equality_in_modulus}.

To summarize, we have described a way to construct $g_1,g_2\in L^2(\R)$ such that $|g_1\star\psi_j|\approx|g_2\star\psi_j|$ for all $j$. The principle is to multiply the wavelet transform of $g_1$ by any set of phases $\{e^{i\phi_j(t)}\}_{j\in\Z}$ whose variations are slow enough in $j$ and $t$.

How slow the variations must be depends on $g_1$. Indeed, at the points $(j,t)$ where $g_1\star\psi_j(t)$ is small, the phase may vary more rapidly because, then, the presence of $g_1\star\psi_j(t-u)$ in \eqref{eq:approx_error} compensates for a bigger $(e^{i\phi_j(t-u)}-e^{i\phi_k(t)})$.

All instabilities $g_1,g_2$ that we were able to observe in practice were of the form we described: each time, the wavelet transforms of $g_1$ and $g_2$ were equal up to a phase whose variation was slow in $j$ and $t$, except at the points where $g_1\star\psi_j$ was small.

\section{Strong stability result\label{s:strong_stability_result}}

The goal of this section is to give a partial formal justification to the fact that has been non-rigorously discussed in section \ref{ss:a_wider_class}: when two functions $g_1,g_2$ satisfy $|g_1\star\psi_j|\approx|g_2\star\psi_j|$ for all $j$, then the wavelet transforms $\{g_1\star\psi_j(t)\}_j$ and $\{g_2\star\psi_j(t)\}_j$ are equal up to a phase whose variation is slow in $t$ and $j$, except eventually at the points where $|g_1\star\psi_j(t)|$ is small.

In the whole section, we consider $f^{(1)},f^{(2)}$ two non-zero functions. We denote by $F^{(1)},F^{(2)}$ the holomorphic extensions defined in \eqref{eq:def_holomorphic_extension}. We recall that, for all $j\in\Z$:
\begin{equation}\label{eq:recall_F_f_psi_j}
f\star\psi_j(x)=\frac{a^{pj}}{2}F(x+ia^j)
\quad\quad\forall x\in\R
\end{equation}
We define:
\begin{equation*}
N_j=\underset{x\in\R,s=1,2}{\sup}|f^{(s)}\star\psi_j(x)|
\end{equation*}

\subsection{Main principle\label{ss:main_principle}}

From $|f\star\psi_j|$, one can calculate $|f\star\psi_j|^2$ and thus, from \eqref{eq:recall_F_f_psi_j}, $|F(x+ia^j)|^2$, for all $x\in\R$. But this last function coincides with $G_j(z)=F(z+ia^j)\overline{F(\overline{z}+ia^j)}$ on the horizontal line $\Im z=0$. As $G_j$ is holomorphic, it is uniquely determined by its values on one line. Consequently, $G_j$ is uniquely determined from $|f\star\psi_j|$.


Combining the functions $G_j$ for different values of $j$ allows to write explicit reconstruction formulas. The stability of these formulas can be studied, to obtain relations of the following form, for $K>0$:
\begin{align*}
\Big(|f^{(1)}\star\psi_k|&\approx|f^{(2)}\star\psi_k|\quad\forall k\in\Z\Big)
\\&\Rightarrow
\Big((f^{(1)}\star\psi_j)(\overline{f^{(1)}\star\psi_{j+K}})
\approx (f^{(2)}\star\psi_j)(\overline{f^{(2)}\star\psi_{j+K}})\quad
\forall j\in\Z
\Big)
\end{align*}
These relations imply that, for each $j$, the phases of $f^{(1)}\star\psi_j$ and $f^{(2)}\star\psi_j$ are approximately equal up to multiplication by the phase of $\frac{\overline{f^{(1)}\star\psi_{j+K}}}{\overline{f^{(2)}\star\psi_{j+K}}}$. If $K$ is not too small, this last phase is low-frequency, compared to the phase of $f^{(1)}\star\psi_j$ and $f^{(2)}\star\psi_j$.

\nl

The results we obtain are local, in the sense that if the approximate equality $|f^{(1)}\star\psi_k|\approx|f^{(2)}\star\psi_k|$ only holds on a (large enough) interval of $\R$, the equality $(f^{(1)}\star\psi_j)(\overline{f^{(1)}\star\psi_{j+K}})\approx (f^{(2)}\star\psi_j)(\overline{f^{(2)}\star\psi_{j+K}})$ still holds (also on an interval of $\R$).

\nl
Our main technical difficulty was to handle properly the fact that the $G_j$'s may have zeros (which is a problem because we need to divide by $G_j$ in order to get reconstruction formulas). We know that, when the wavelet transform has a lot of zeros, the reconstruction becomes unstable. On the other hand, if they are only a few isolated zeros, the reconstruction is stable and this must appear in our theorems.

They are several ways to write reconstruction formulas, which give different stability results. In the dyadic case $(a=2)$, there is a relatively simple method. We present it first. Then we handle the case where $a<2$. We do not consider the case where $a>2$. Indeed, it has less practical interest for us. Moreover, when the value of $a$ increases, the reconstruction becomes much less stable.

\subsection{Case $\boldsymbol{a=2}$}

In the dyadic case, we only assume that two consecutive moduli are approximately known, on an interval of $\R$: $|f\star\psi_j|$ and $|f\star\psi_{j+1}|$. We also assume that, on this interval, the moduli are never too close to $0$. Then we show these moduli stabily determine:
\begin{equation*}
\frac{f\star\psi_{j+2}}{f\star\psi_{j+1}}
\end{equation*}

\begin{thm}\label{thm:stability_dyadic_case}
Let $\epsilon,c,\lambda\in]0;1[,M>0$ be fixed, with $c\geq\epsilon$.

We assume that, for all $x\in[-M2^j;M2^j]$:
\begin{gather*}
\left||f^{(1)}\star\psi_j(x)|^2-|f^{(2)}\star\psi_j(x)|^2\right|\leq \epsilon N_j^2\\
\left||f^{(1)}\star\psi_{j+1}(x)|^2-|f^{(2)}\star\psi_{j+1}(x)|^2\right|\leq \epsilon N_{j+1}^2
\end{gather*}
and:
\begin{gather*}
|f^{(1)}\star\psi_j(x)|^2,|f^{(2)}\star\psi_j(x)|^2\geq c N_j^2\\
|f^{(1)}\star\psi_{j+1}(x)|^2,|f^{(2)}\star\psi_{j+1}(x)|^2\geq c N_{j+1}^2
\end{gather*}

Then, for all $x\in[-\lambda^2M2^j;\lambda^2M2^j]$:
\begin{equation*}
\left|\frac{f^{(1)}\star\psi_{j+2}}{f^{(1)}\star\psi_{j+1}}(x)
-\frac{f^{(2)}\star\psi_{j+2}}{f^{(2)}\star\psi_{j+1}}(x)
\right|\leq
\frac{A}{c}\left(\frac{N_{j-1}}{N_{j+1}}\right)^{4/3}\epsilon^{(1/3-\alpha_M)(4/5-\alpha_M')}
\end{equation*}
if $1/3-\alpha_M>0$ and $4/5-\alpha_M'>0$, where:
\begin{itemize}
\item $A$ is a constant which depends only on $p$.
\item $\alpha_M,\alpha'_M\to 0$ exponentially when $M\to+\infty$.
\end{itemize}
\end{thm}

\begin{proof}[Principle of the proof]
Here, we only give a broad outline of the proof. A rigorous one is given in the appendix \ref{app:stability_dyadic_case}, with all the necessary technical details.



As explained in the paragraph \ref{ss:main_principle}, $|f^{(1)}\star\psi_{j+1}|$ uniquely determines the values of $z\to F^{(1)}(z+i2^{j+1})\overline{F^{(1)}(\overline{z}+i2^{j+1})}$ on the line $\Im z=0$. Thus, it uniquely determines all the values (because the function is holomorphic) and in particular (for $z=x+i2^j$):
\begin{equation*}
F^{(1)}(x+i3.2^j)\overline{F^{(1)}(x+i2^j)}\quad
\forall x\in\R
\end{equation*}
Moreover, this determination is a stable operation:
\begin{align*}
\Big(|f^{(1)}\star\psi_{j+1}(x)|^2&\approx|f^{(2)}\star\psi_{j+1}(x)|^2\quad\forall x\in\R
\Big)\\
&\Rightarrow
\Big(F^{(1)}(x+i3.2^j)\overline{F^{(1)}(x+i2^j)}
\approx F^{(2)}(x+i3.2^j)\overline{F^{(2)}(x+i2^j)}
\quad\forall x\in\R
\Big)
\end{align*}
If we divide this last expression by $|F^{(1)}(x+i2^j)|^2\approx|F^{(2)}(x+i2^j)|^2$ (whose values we know from $|f\star\psi_j|^2$):
\begin{equation*}
\frac{F^{(1)}(x+i3.2^j)}{F^{(1)}(x+i2^j)}
\approx
\frac{F^{(2)}(x+i3.2^j)}{F^{(2)}(x+i2^j)}
\quad\mbox{for }x\in\R
\end{equation*}
As previously, using the holomorphy of $F$ allows to replace, in the last expression, the real number $x$ by $x+i2^j$:
\begin{equation*}
\frac{F^{(1)}(x+i2^{j+2})}{F^{(1)}(x+i2^{j+1})}
\approx
\frac{F^{(2)}(x+i2^{j+2})}{F^{(2)}(x+i2^{j+1})}
\quad\mbox{for }x\in\R
\end{equation*}
By \eqref{eq:recall_F_f_psi_j}, this is the same as:
\begin{equation*}
\frac{f^{(1)}\star\psi_{j+2}}{f^{(1)}\star\psi_{j+1}}
\approx
\frac{f^{(2)}\star\psi_{j+2}}{f^{(2)}\star\psi_{j+1}}
\end{equation*}
\end{proof}

From this theorem, if $f^{(s)}\star\psi_{j+2}$ has no small values either on $[-\lambda^2M2^j;\lambda^2M2^j]$, then:
\begin{equation*}
\mbox{phase}(f^{(1)}\star\psi_{j+1})-\mbox{phase}(f^{(2)}\star\psi_{j+1})
\approx
\mbox{phase}(f^{(1)}\star\psi_{j+2})-\mbox{phase}(f^{(2)}\star\psi_{j+2})
\end{equation*}
If more than two consecutive components of the wavelet transform have almost the same modulus (and all these components do not come close to $0$), one can iterate this approximate equality. It gives:
\begin{equation*}
\mbox{phase}(f^{(1)}\star\psi_{j+1})-\mbox{phase}(f^{(2)}\star\psi_{j+1})
\approx
\mbox{phase}(f^{(1)}\star\psi_{j+K})-\mbox{phase}(f^{(2)}\star\psi_{j+K})
\end{equation*}
This holds for any $K\in\N^*$ but with an approximation error that becomes larger and larger as $K$ increases.

When $K$ is large enough, this means that $f^{(1)}\star\psi_{j+1}$ and $f^{(2)}\star\psi_{j+1}$ are equal up to a low-frequency phase.

\subsection{Case $\boldsymbol{a<2}$\label{ss:case_a_less_than_2}}

\textbf{Notations:} for this section, we fix:
\begin{itemize}
\item $j\in\Z$: the frequency of the component whose phase we want to estimate

\item $K\in\N^*$ such that $K\equiv 0[2]$: the number of components of the wavelet transform whose modulus are approximately equal

\item $\epsilon,\kappa\in]0;1[$: they will controle the difference between $|f^{(1)}\star\psi_j|$ and $|f^{(2)}\star\psi_j|$, as well as the minimal value of those functions.

\item $M>0$: we will assume that the approximate equality between the modulus holds on $[-Ma^{j+K};Ma^{j+K}]$.

\item $k\in\N^*$ such that $a^{-k}< 2-a$: this number will controle the stability with which one can derive informations about $f\star\psi_{l-1}$ from $|f\star\psi_l|$. Typically, for $a\leq 1.5$, we may take $k=3$.
\end{itemize}

We define:
\begin{itemize}
\item $J\in[j+K-1;j+K]$ such that $a^J=\frac{2}{a+1}a^{j+K}+\frac{a-1}{a+1}a^j$: we will prove that $f^{(1)}\star\psi_j$ and $f^{(2)}\star\psi_j$ are equal up to a phase which is concentrated around $a^J$ in frequencies (that is, a much lower-frequency phase than the phase of $f\star\psi_j$).

\item $c=1-\frac{a-1}{1-a^{-k}}\in]0;1[$
and $d_M=c-4\frac{e^{-\pi M/(K+2)}}{1-e^{-\pi M/(K+2)}}$, which converges exponentially to $c$ when $\frac{M}{K}$ goes to $\infty$.
\end{itemize}

\begin{thm}\label{thm:stability_a_less_than_2}
We assume that $\kappa\geq \epsilon^{2(1-c)}$.

We assume that, for $x\in[-Ma^{j+K};Ma^{j+K}]$ and $l=j+1,...,j+K$:
\begin{gather}
\left||f^{(1)}\star\psi_l(x)|^2-|f^{(2)}\star\psi_l(x)|^2\right|\leq\epsilon N_l^2\label{eq:thm_stab_hyp_1}\\
|f^{(1)}\star\psi_l(x)|^2,|f^{(2)}\star\psi_l(x)|^2\geq \kappa N_l^2\label{eq:thm_stab_hyp_2}
\end{gather}
Then, for any $x\in\left[- \frac{M a^{j+K}}{2};\frac{ Ma^{j+K}}{2}\right]$, as soon as $d_M<1$:
\begin{align}
\frac{1}{N_JN_j}\left|\left(\overline{f^{(1)}\star\psi_J(x)}\right) \right.&\left.\left(f^{(1)}\star\psi_j(x)\right)-\left(\overline{f^{(2)}\star\psi_J(x)}\right) \left(f^{(2)}\star\psi_j(x)\right)\right|
\leq \frac{C_K}{\kappa^{K/4}}\epsilon^{d_M}
\label{eq:stability}
\end{align}
where $C_K=\frac{6}{1-\sqrt{\kappa}}\underset{s=0}{\overset{K/2-1}{\prod}}\left(a^{p(k-1)}\frac{N_{n_s-1-k}}{N_{n_s-2}}\right)$
\end{thm}

As in the dyadic case $a=2$, this theorem shows that, if two functions $f^{(1)}$ and $f^{(2)}$ have their wavelet transforms almost equal in moduli, then, for each $j$, $f^{(1)}\star\psi_j\approx f^{(2)}\star\psi_j$ up to multiplication by a low-frequency function.

In contrast to the dyadic case, we are not able to show directly that:
\begin{equation*}
\frac{f^{(1)}\star\psi_j}{f^{(2)}\star\psi_j}\approx \frac{f^{(1)}\star\psi_{j+1}}{f^{(2)}\star\psi_{j+1}}
\end{equation*}
Because of that, the inequality we get is less good than in the dyadic case: the bound in \eqref{eq:stability} is exponential in $K$ instead of being proportional to $K$.

With a slightly different method, we could have obtained a better bound, proportional to $K$. This better bound would have been valid for any $a>1$, but under the condition that $f\star\psi_l$ does not come close to $0$ for some explicit non-integer values of $l$, which would have been rather unsatisfying because, in practice, these values of $l$ do not seem to play a particular role.

\begin{proof}[Principle of the proof]
The full proof may be found in appendix \ref{app:stability_a_less_than_2}. Its principle is to show, by induction over $s=0,...,K/2$, that:
\begin{equation}\label{eq:recurrence_formula}
(\overline{f^{(1)}\star\psi_{J_s}})(f^{(1)}\star\psi_{j+K-2s})
\approx
(\overline{f^{(2)}\star\psi_{J_s}})(f^{(2)}\star\psi_{j+K-2s})
\end{equation}
where $J_s$ is an explicit number in the interval $[j+K-1;j+K]$.

For $s=0$, we set $J_s=j+K$ and \eqref{eq:recurrence_formula} just says:
\begin{equation*}
\left|f^{(1)}\star\psi_{j+K}\right|^2\approx
\left|f^{(2)}\star\psi_{j+K}\right|^2
\end{equation*}
which is true by hypothesis.

Then, to go from $s$ to $s+1$, we use the fact that:
\begin{equation}\label{eq:recurrence_principle}
(\overline{f^{(1)}\star\psi_{j+K-2s}})(f^{(1)}\star\psi_l)
\approx
(\overline{f^{(2)}\star\psi_{j+K-2s}})(f^{(2)}\star\psi_l)
\end{equation}
if we choose $l$ such that $a^l=2a^{j+K-2s-1}-a^{j+K-2s}$: we can check that, up to multiplication by a constant, $(\overline{f^{(r)}\star\psi_{j+K-2s}})(f^{(r)}\star\psi_l)$ is the evaluation on the line $a^{j+K-2s}-a^{j+K-2s-1}$ of the holomorphic extension of $|f^{(r)}\star\psi_{j+K-2s-1}|^2$. The holomorphic extension is a stable transformation (in a sense that has to be made precise). As $|f^{(1)}\star\psi_{j+K-2s-1}|^2\approx|f^{(2)}\star\psi_{j+K-2s-1}|^2$, this implies \eqref{eq:recurrence_principle}.

Multiplying \eqref{eq:recurrence_formula} and \eqref{eq:recurrence_principle} and dividing by $|f^{(1)}\star\psi_{j+K-2s}|^2\approx|f^{(2)}\star\psi_{j+K-2s}|^2$ yields:
\begin{equation}\label{eq:recurrence_almost}
(\overline{f^{(1)}\star\psi_{J_s}})(f^{(1)}\star\psi_l)
\approx
(\overline{f^{(2)}\star\psi_{J_s}})(f^{(2)}\star\psi_l)
\end{equation}
If $J_{s+1}$ is suitably chosen, $(f^{(r)}\star\psi_{J_{s+1}})(f^{(r)}\star\psi_{j+K-2(s+1)})$ may be seen as the restriction to a line of the holomorphic extension of $(\overline{f^{(r)}\star\psi_{J_s}})(f^{(r)}\star\psi_l)$. Because, again, taking the holomorphic extension is relatively stable, the relation \eqref{eq:recurrence_almost} implies the recurrence hypothesis \eqref{eq:recurrence_formula} at order $s+1$.

For $s=K/2$, the recurrence hypothesis is equivalent to the stated result.
\end{proof}

\section{Numerical experiments\label{s:numerical_experiments}}

In the previous section, we proved a form of stability for the phase retrieval problem associated to the Cauchy wavelet transform. The proof implicitely relied on the existence of an explicit reconstruction algorithm. In this section, we describe a practical implementation of this algorithm and its performances.

The main goal of our numerical experiments is to investigate the issue of stability. The theorems \ref{thm:stability_dyadic_case} and \ref{thm:stability_a_less_than_2} prove that the reconstruction is, in some sense, stable, at least when the wavelet transform does not have small values. Are these results confirmed by the implementation? To what extent does the presence of small values make the reconstruction unstable?

As we will see, our algorithm fails when large parts of the wavelet transform are close to zero. In all other cases, it seems to succeed and to be stable to noise, even when the amount of noise over the wavelet transform is relatively high ($\sim$ 10\%). The presence of a small number of zeroes in the wavelet transform is not a problem.




\nl
The code is available at \url{http://www.di.ens.fr/~waldspurger/cauchy_phase_retrieval.html}, along with examples of reconstruction for audio signals. It only handles the dyadic case $a=2$ but could easily be extended to other values of $a$.

\subsection{Description of the algorithm}

In practice, we must restrict our wavelet transform to a finite number of components. So we only consider the $|f\star\psi_j|$ for $j\in\{J_{\min},...,J_{\max}\}$. To compensate for the loss of the $|f\star\psi_j|$ with $j>J_{\max}$, we give to our algorithm an additional information about the low-frequency, under the form of $f\star\phi_{J_{\max}}$, where $\hat\phi_{J_{\max}}$ is negligible outside an neighborhood of $0$ of size $\sim a^{-J_{\max}}$.

\nl
The algorithm takes as input the functions $|f\star\psi_{J_{\min}}|,|f\star\psi_{J_{\min}+1}|,...,|f\star\psi_{J_{\max}}|,f\star\phi_{J_{\max}}$, for some unknown $f$, and tries to reconstruct $f$. The input functions may be contaminated by some noise. To simplify the implementation, we have assumed that the probability distribution of the noise was known.

For any real numbers $j,k_1,k_2$ such that $j\in\Z$ and $2.a^j=a^{k_1}+a^{k_2}$, it comes from the reasoning of the previous section that $|f\star\psi_j|$ uniquely determines $(f\star\psi_{k_1}).(\overline{f\star\psi_{k_2}})$. More precisely, we have, for all $\omega\in\R$:
\begin{equation}\label{eq:k1_k2_from_j}
\widehat{(f\star\psi_{k_1}).(\overline{f\star\psi_{k_2}})}(\omega)
=\widehat{|f\star\psi_j|^2}(\omega)e^{(a^{k_2}-a^j)\omega}\frac{a^{k_1+k_2}}{a^{2j}}
\end{equation}
The algorithm begins by fixing real numbers $k_{J_{\min}-1},k_{J_{\min}},...,k_{J_{\max}}$ such that:
\begin{gather}\label{eq:sequence_k_j}
k_{J_{\min}-1}<J_{\min}<k_{J_{\min}}<J_{\min}+1<...<J_{\max}<k_{J_{\max}}\\
\forall j,\quad\quad
2.a^j = a^{k_{j-1}}+a^{k_j}\nonumber
\end{gather}
Then, for all $j$, it applies \eqref{eq:k1_k2_from_j} to determine $g_j\overset{\mbox{\small def}}{=}(f\star\psi_{k_{j-1}}).(\overline{f\star\psi_{k_j}})$. Because of the exponential function present in \eqref{eq:k1_k2_from_j}, the $g_j$ may take arbitrarily high values in the frequency band $\{(a^{k_2}-a^j)\omega\gg 1\}$. To avoid this, we truncate the high frequencies of $g_j$.

The function $f\star\psi_{k_{J_{\max}}}$ may be approximately determined from $f\star\phi_{J_{\max}}$. From this function and the $g_j$, the algorithm estimates all the $f\star\psi_{k_j}$. As this estimation involves divisions by functions which may be close to zero at some points, it is usually not very accurate. In particular, the estimated set $\{f\star\psi_{k_j}\}_j$ do not generally satisfy the constraint that it must belong to the range of the function $f\in L^2(\R)\to\{f\star\psi_{k_j}\}_{J_{\min}-1\leq j\leq J_{\max}}$.

Thus, in a second step, the algorithm refines the estimation. To do this, it attempts to minimize an error function which takes into account both the fact that $(f\star\psi_{k_{j-1}}).(\overline{f\star\psi_{k_j}})$ is known for every $j$ and the fact that $\{f\star\psi_{k_{j-1}}\}_{J_{\min}-1\leq j\leq J_{\max}}$ must belong to the range of $f\in L^2(\R)\to\{f\star\psi_{k_j}\}_{J_{\min}-1\leq j\leq J_{\max}}$. The minimization is performed by gradient descent, using the previously found estimations as initialization.

Finally, we deduce $f$ from the $f\star\psi_{k_{j-1}}$ and refine this estimation one more time by a few steps of the classical Gerchberg-Saxton algorithm (\citep{gerchberg}).

\nl
The principle of the algorithm is summarized by the pseudocode \ref{alg:rec}.

\begin{algorithm}[ht] 
\caption{Reconstruction algorithm} 
\label{alg:rec} 
\begin{algorithmic} [1]
\REQUIRE $\{|f\star\psi_j|\}_{J_{\min}\leq j\leq J_{\max}}$ and $f\star\phi_{J_{\max}}$
\STATE Choose $k_{J_{\min}-1},...,k_{J_{\max}}$ as in \eqref{eq:sequence_k_j}.
\FORALL {$j$}
\STATE Determine $g_j=(f\star\psi_{k_{j-1}}).(\overline{f\star\psi_{k_j}})$ from $|f\star\psi_j|^2$.
\ENDFOR
\STATE Determine $f\star\psi_{k_{J_{\max}}}$ from $f\star\phi_{J_{\max}}$.
\FORALL {$j$}
\STATE Estimate $h_j\approx f\star\psi_{k_j}$.
\ENDFOR
\STATE Refine the estimation with a gradient descent.
\STATE Deduce $f$ from $\{f\star\psi_{k_j}\}_{J_{\min}-1\leq j\leq J_{\max}}$.
\STATE Refine the estimation of $f$ with the Gerchberg-Saxton algorithm.
\ENSURE $f$
\end{algorithmic} 
\end{algorithm} 

\subsection{Input signals}

\begin{figure}
\begin{minipage}[b]{0.3\textwidth}
\begin{center}
\includegraphics[width=\textwidth]{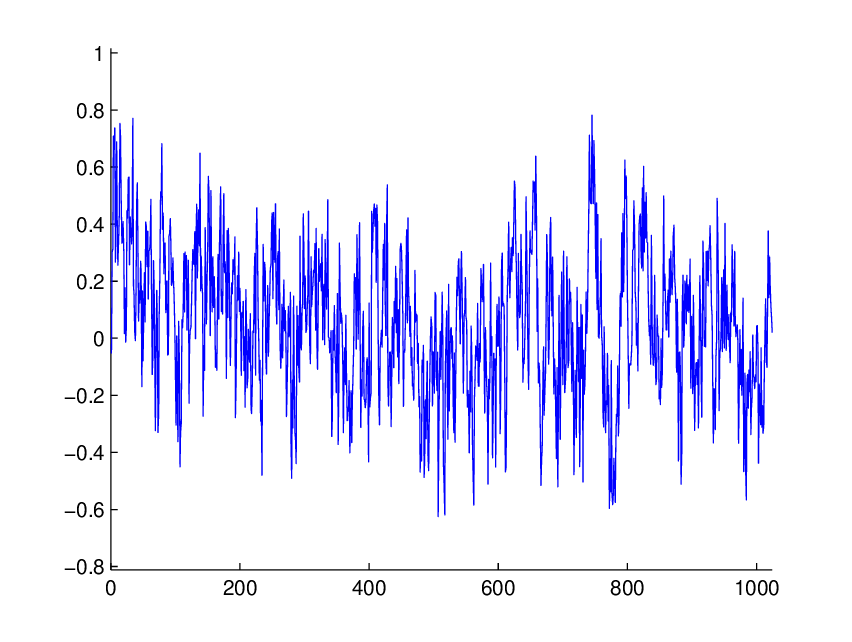}\\
(a)
\end{center}
\end{minipage}
\hskip 0.3cm
\begin{minipage}[b]{0.3\textwidth}
\begin{center}
\includegraphics[width=\textwidth]{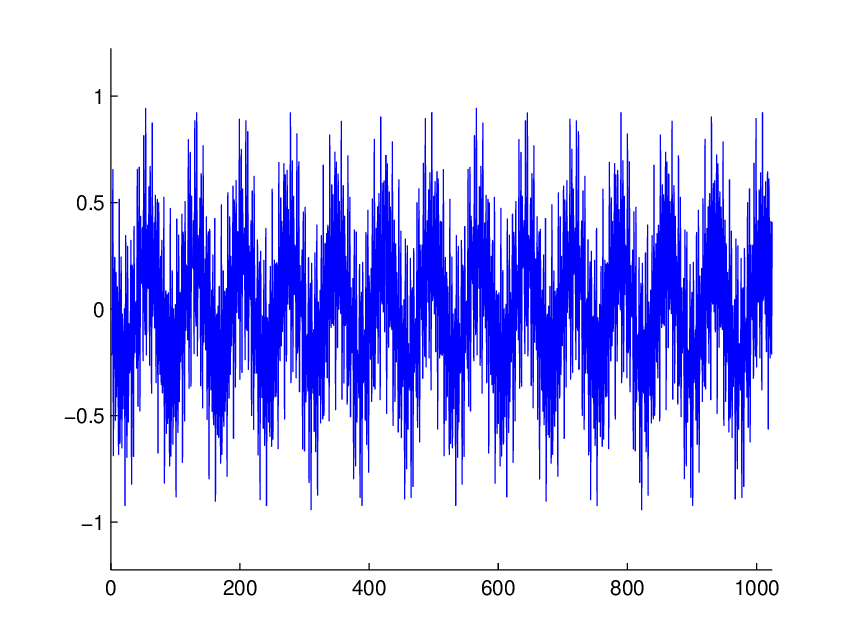}\\
(b)
\end{center}
\end{minipage}
\hskip 0.3cm
\begin{minipage}[b]{0.3\textwidth}
\begin{center}
\includegraphics[width=\textwidth]{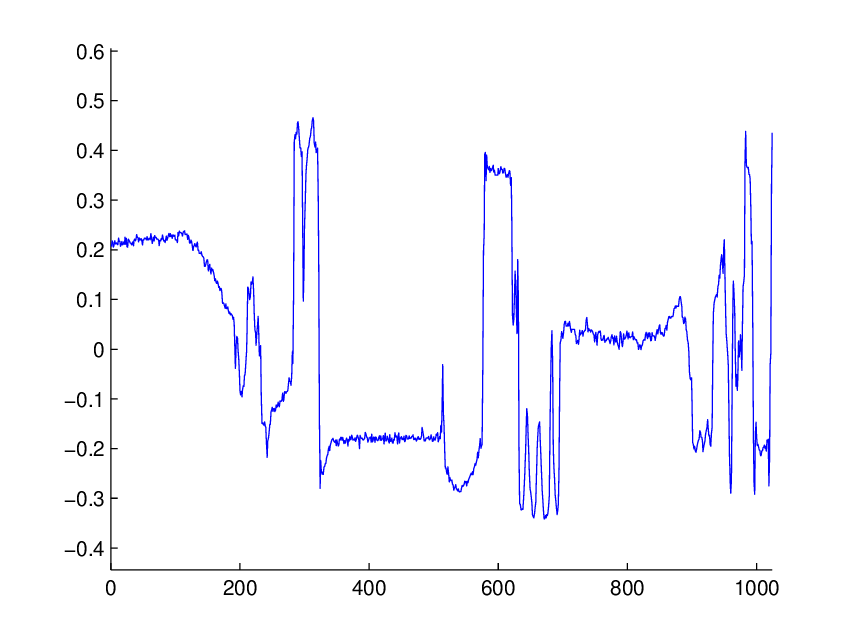}\\
(c)
\end{center}
\end{minipage}
\caption{examples of signals: (a) realization of a gaussian process (b) sum of sinusoids (c) piecewise regular
\label{fig:three_classes}}
\end{figure}

We study the performances of this algorithm on three classes of input signals with finite size $n$. The figure \ref{fig:three_classes} shows an example for each of these three classes.

\nl

The first class contains realizations of gaussian processes with renormalized frequencies. More precisely, the signals $f$ of this class satisfy:
\begin{equation*}
\hat f[n]=\frac{X_n}{\sqrt{n+2}}
\end{equation*}
where the $X_n$ are independant realizations of a gaussian random variable $X\sim \mathcal{N}(0,1)$. The normalization $\frac{1}{\sqrt{n+2}}$ ensures that all dyadic frequency bands contain approximately the same amount of energy.

The second class consists in sums of a few sinusoids. The amplitudes, phases and frequencies of the sinusoids are randomly chosen. In each dyadic frequency band, there is approximately the same mean number of sinusoids (slightly smaller than $1$).

The signals of the third class are random lines extracted from real images. They usually are structured signals, with smooth regular parts and large discontinuities at a small number of points.

\nl
To study the influence of the size of the signals on the reconstruction, we perform tests for signals of size $N=128,N=1024$ and $N=8192$. For each $N$, we used $\log_2(N)-1$ Cauchy wavelets of order $p=3$. Our low-pass filter is a gaussian function of the form $\hat\phi[k]=\exp(-\alpha k^2/2)$, with $\alpha$ independant of $N$.

\subsection{Noise}


The inputs that are provided to the algorithm are not exactly $\{|f\star\psi_j|\},f\star\phi_{J_{\max}}$ but $\{|f\star\psi_j|+n_{\psi,j}\},f\star\phi_{J_{\max}}+n_{\phi}$. The $n_{\psi,j}$ and the $n_\phi$ represent an additive noise. In all our experiments, this noise is white and gaussian.

We measure the amplitude of the noise in relative $l^2$-norm:
\begin{equation*}
\mbox{relative noise }=
\frac{\sqrt{||n_\phi||_2^2+\underset{j}{\sum}||n_{\psi,j}||_2^2}}{\sqrt{||f\star\phi_{J_{\max}}||_2^2+\underset{j}{\sum}||f\star\psi_{j}||_2^2}}
\end{equation*}

\subsection{Results}

The results are displayed on the figure \ref{fig:results}.

The x-axis displays the relative error induced by the noise over the input and the y-axis represents the reconstruction error, both over the reconstructed function and over the modulus of the wavelet transform of the reconstructed function.

For an input signal $f$ and output $f_{rec}$, we define the relative error between $f$ and $f_{rec}$ by:
\begin{equation*}
\mbox{function error }=
\frac{||f-f_{rec}||_2}{||f||_2}
\end{equation*}
and the relative error over the modulus of the wavelet transform by:
\begin{equation*}
\mbox{modulus error }=
\frac{\sqrt{||f\star\phi_{J_{\max}}-f_{rec}\star\phi_{J_{\max}}||_2^2+\underset{j}{\sum}||\,|f\star\psi_j|-|f_{rec}\star\psi_j|\,||_2^2}}{\sqrt{||f\star\phi_{J_{\max}}||_2^2+\underset{j}{\sum}||f\star\psi_{j}||_2^2}}
\end{equation*}

The modulus error describes the capacity of the algorithm to reconstruct a signal whose wavelet transform is close, in modulus, to the one which has been provided as input. The function error, on the other hand, quantifies the intrinsic stability of the phase retrieval problem. If the modulus error is small but the function error is large, it means that there are several functions whose wavelet transforms are almost equal in moduli and the reconstruction problem is ill-posed.

An ideal reconstruction algorithm would yield a small modulus error (that is, proportional to the noise over the input). Nevertheless, the function error could be large or small, depending on the well-posedness of the phase retrieval problem.

\begin{figure}
\pgfplotsset{every axis legend/.append style={
at={(1.02,1)},
anchor=north west}}

\begin{tabular}{cc}
\begin{tikzpicture}[scale=0.7]
\begin{loglogaxis}[
ymin=0.0001, ymax = 3,
title=Gaussian signals,
xlabel={Relative noise},
ylabel={Relative error},
legend entries={moduli (128),function (128),moduli (1024),function (1024),moduli (8192),function (8192)}]
\addplot[blue,mark=*] coordinates{
(1e-3, 2.8e-4)
(1e-2, 2.8e-3)
(1e-1, 2.9e-2)
};
\addplot[red,mark=triangle] coordinates{
(1e-3, 6.9e-4)
(1e-2, 6.9e-3)
(1e-1, 7.9e-2)
};
\addplot[blue,mark=*,dashed] coordinates{
  (1e-3, 2.3e-4)
  (1e-2, 2.5e-3)
  (1e-1, 2.4e-2)
};
\addplot[red,mark=triangle,dashed] coordinates{
  (1e-3, 6.8e-4)
  (1e-2, 8.9e-3)
  (1e-1, 8.8e-2)
};
\addplot[blue,mark=*,dotted,thick] coordinates{
  (1e-3, 2.3e-4)
  (1e-2, 2.5e-3)
  (1e-1, 2.1e-2)
};
\addplot[red,mark=triangle,dotted,thick] coordinates{
  (1e-3, 8.4e-4)
  (1e-2, 1.2e-2)
  (1e-1, 8.9e-2)
};
\end{loglogaxis}
\end{tikzpicture}&
\begin{tikzpicture}[scale=0.7]
\begin{loglogaxis}[
ymin=0.0001, ymax = 3,
title=Gaussian signals (Gerchberg-Saxton),
xlabel={Relative noise},
ylabel={Relative error},
legend entries={moduli (128),function (128),moduli (1024), function (1024),moduli (8192),function (8192)}]
\addplot[blue,mark=*] coordinates{
(1e-3, 2.1e-2)
(1e-2, 2.2e-2)
(1e-1, 4.3e-2)
};
\addplot[red,mark=triangle] coordinates{
(1e-3, 0.14)
(1e-2, 0.14)
(1e-1, 0.20)
};
\addplot[blue,mark=*,dashed] coordinates{
  (1e-3, 3.9e-2)
  (1e-2, 3.9e-2)
  (1e-1, 4.8e-2)
};
\addplot[red,mark=triangle,dashed] coordinates{
  (1e-3, 0.36)
  (1e-2, 0.36)
  (1e-1, 0.38)
};
\addplot[blue,mark=*,dotted,thick] coordinates{
  (1e-3, 5.5e-2)
  (1e-2, 5.5e-2)
  (1e-1, 6.0e-2)
};
\addplot[red,mark=triangle,dotted,thick] coordinates{
  (1e-3, 0.59)
  (1e-2, 0.58)
  (1e-1, 0.59)
};
\end{loglogaxis}
\end{tikzpicture}
\\
\begin{tikzpicture}[scale=0.7]
\begin{loglogaxis}[
ymin=0.0001, ymax = 3,
title=Sum of sinusoids,
xlabel={Relative noise},
ylabel={Relative error},
legend entries={moduli (128),function (128),moduli (1024),function (1024),moduli (8192),function (8192)}]
\addplot[blue,mark=*] coordinates{
(1e-3, 2.3e-2)
(1e-2, 3.0e-2)
(1e-1, 6.7e-2)
};
\addplot[red,mark=triangle] coordinates{
(1e-3, 0.25)
(1e-2, 0.31)
(1e-1, 0.49)
};
\addplot[blue,mark=*,dashed] coordinates{
  (1e-3, 2.3e-2)
  (1e-2, 3.4e-2)
  (1e-1, 6.7e-2)
};
\addplot[red,mark=triangle,dashed] coordinates{
  (1e-3, 0.32)
  (1e-2, 0.42)
  (1e-1, 0.68)
};
\addplot[blue,mark=*,dotted,thick] coordinates{
  (1e-3, 2.1e-2)
  (1e-2, 3.2e-2)
  (1e-1, 6.6e-2)
};
\addplot[red,mark=triangle,dotted,thick] coordinates{
  (1e-3, 0.35)
  (1e-2, 0.48)
  (1e-1, 0.78)
};\end{loglogaxis}
\end{tikzpicture}&
\begin{tikzpicture}[scale=0.7]
\begin{loglogaxis}[
ymin=0.0001, ymax = 3,
title=Sum of sinusoids (Gerchberg-Saxton),
xlabel={Relative noise},
ylabel={Relative error},
legend entries={moduli (128),function (128),moduli (1024), function (1024),moduli (8192),function (8192)}]
\addplot[blue,mark=*] coordinates{
(1e-3, 5.3e-2)
(1e-2, 5.8e-2)
(1e-1, 8.5e-2)
};
\addplot[red,mark=triangle] coordinates{
(1e-3, 0.47)
(1e-2, 0.50)
(1e-1, 0.60)
};
\addplot[blue,mark=*,dashed] coordinates{
  (1e-3, 8.0e-2)
  (1e-2, 8.2e-2)
  (1e-1, 9.7e-2)
};
\addplot[red,mark=triangle,dashed] coordinates{
  (1e-3, 0.83)
  (1e-2, 0.85)
  (1e-1, 0.88)
};
\addplot[blue,mark=*,dotted,thick] coordinates{
  (1e-3, 9.1e-2)
  (1e-2, 9.4e-2)
  (1e-1, 0.10)
};
\addplot[red,mark=triangle,dotted,thick] coordinates{
  (1e-3, 1.0)
  (1e-2, 1.0)
  (1e-1, 1.1)
};\end{loglogaxis}
\end{tikzpicture}
\\
\begin{tikzpicture}[scale=0.7]
\begin{loglogaxis}[
ymin=0.0001, ymax = 3,
title=Piecewise regular functions,
xlabel={Relative noise},
ylabel={Relative error},
legend entries={moduli (128),function (128),moduli (1024),function (1024),moduli (8192),function (8192)}]
\addplot[blue,mark=*] coordinates{
(1e-3, 3.0e-4)
(1e-2, 2.8e-3)
(1e-1, 3.1e-2)
};
\addplot[red,mark=triangle] coordinates{
(1e-3, 7.5e-4)
(1e-2, 6.5e-3)
(1e-1, 7.6e-2)
};
\addplot[blue,mark=*,dashed] coordinates{
(1e-3, 3.1e-4)
(1e-2, 2.7e-3)
(1e-1, 3.1e-2)
};
\addplot[red,mark=triangle,dashed] coordinates{
(1e-3, 1.6e-3)
(1e-2, 1.2e-2)
(1e-1, 8.9e-2)
};
\addplot[blue,mark=*,dotted,thick] coordinates{
  (1e-3, 3.5e-4)
  (1e-2, 2.4e-3)
  (1e-1, 2.9e-2)
};
\addplot[red,mark=triangle,dotted,thick] coordinates{
  (1e-3, 2.9e-3)
  (1e-2, 1.4e-2)
  (1e-1, 8.6e-2)
};
\end{loglogaxis}
\end{tikzpicture}&
\begin{tikzpicture}[scale=0.7]
\begin{loglogaxis}[
ymin=0.0001, ymax = 3,
title= Piecewise regular functions (Gerchberg-Saxton),
xlabel={Relative noise},
ylabel={Relative error},
legend entries={moduli (128),function (128), moduli (1024), function (1024),moduli (8192),function (8192)}]
\addplot[blue,mark=*] coordinates{
(1e-3, 7.8e-3)
(1e-2, 9.8e-3)
(1e-1, 3.6e-2)
};
\addplot[red,mark=triangle] coordinates{
(1e-3, 5.9e-2)
(1e-2, 6.3e-2)
(1e-1, 0.12)
};
\addplot[blue,mark=*,dashed] coordinates{
  (1e-3, 1.1e-2)
  (1e-2, 1.2e-2)
  (1e-1, 3.6e-2)
};
\addplot[red,mark=triangle,dashed] coordinates{
  (1e-3, 0.098)
  (1e-2, 0.10)
  (1e-1, 0.15)
};
\addplot[blue,mark=*,dotted,thick] coordinates{
  (1e-3, 2.7e-2)
  (1e-2, 2.7e-2)
  (1e-1, 4.4e-2)
};
\addplot[red,mark=triangle,dotted,thick] coordinates{
  (1e-3, 0.25)
  (1e-2, 0.25)
  (1e-1, 0.28)
};
\end{loglogaxis}
\end{tikzpicture}
\end{tabular}
\caption{Reconstruction results for the three considered classes of signals. Left column: our algorithm. Right column: alternate projections (Gerchberg-Saxton)\label{fig:results}}
\end{figure}
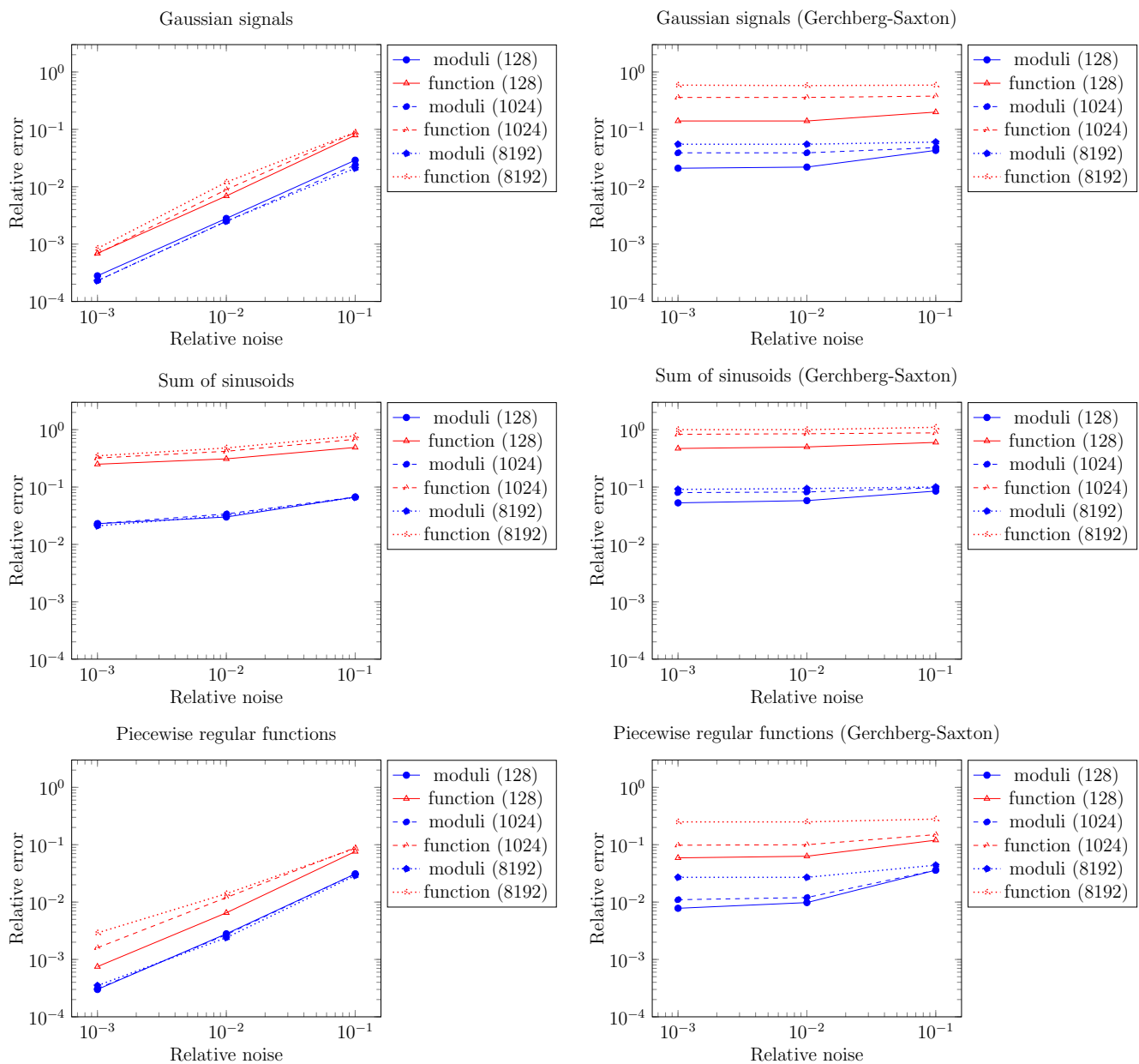

\nl

We expect that our algorithm may fail when the input modulus contain very small values (because the algorithm performs divisions, which become very unstable in presence of zeroes).


\begin{figure}
\begin{minipage}[b]{0.3\textwidth}
\begin{center}
\includegraphics[width=\textwidth]{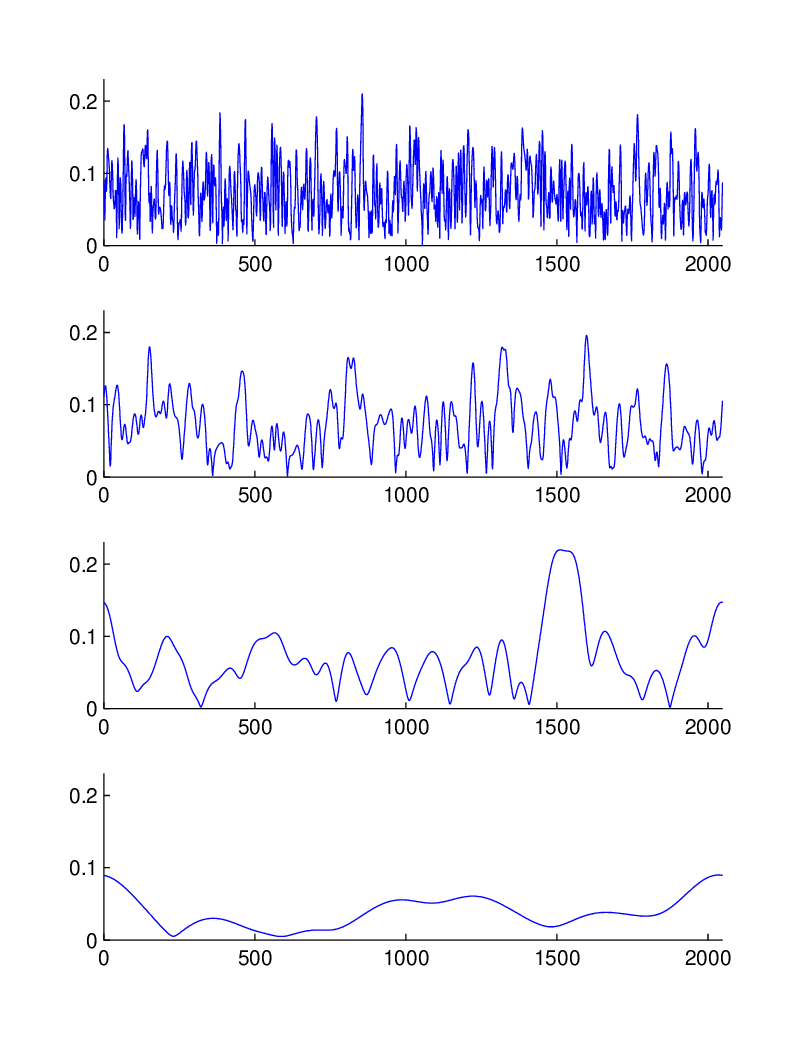}\\
(a)
\end{center}
\end{minipage}
\hskip 0.3cm
\begin{minipage}[b]{0.3\textwidth}
\begin{center}
\includegraphics[width=\textwidth]{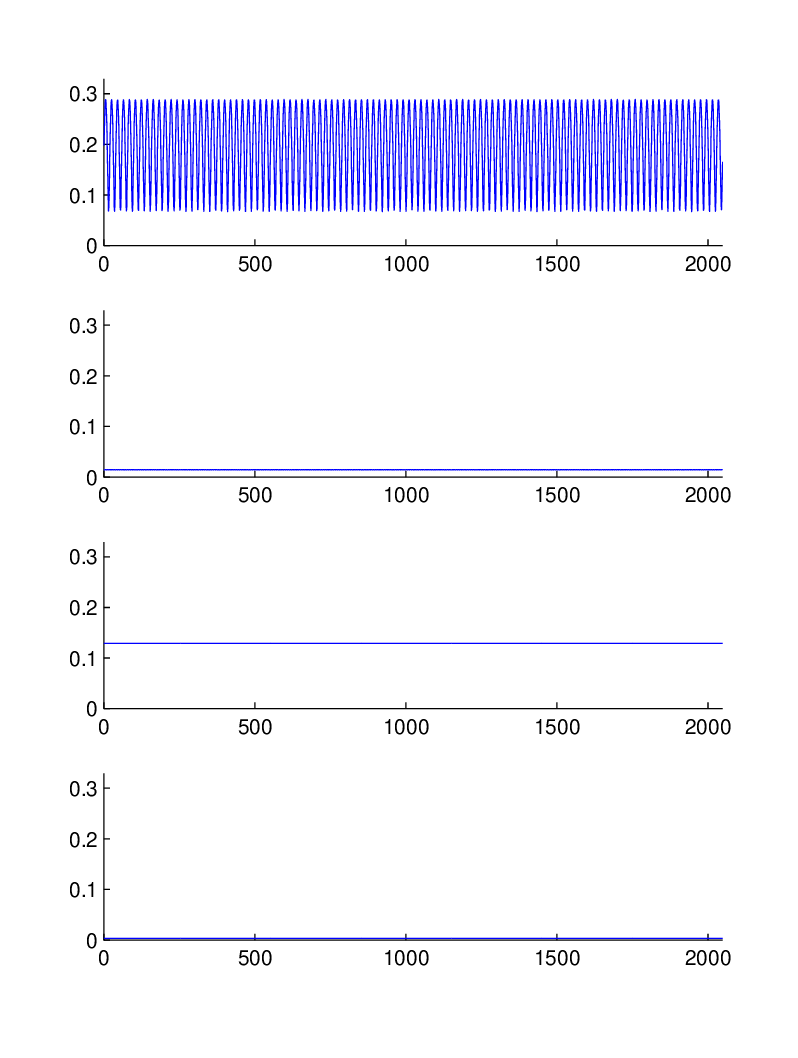}\\
(b)
\end{center}
\end{minipage}
\hskip 0.3cm
\begin{minipage}[b]{0.3\textwidth}
\begin{center}
\includegraphics[width=\textwidth]{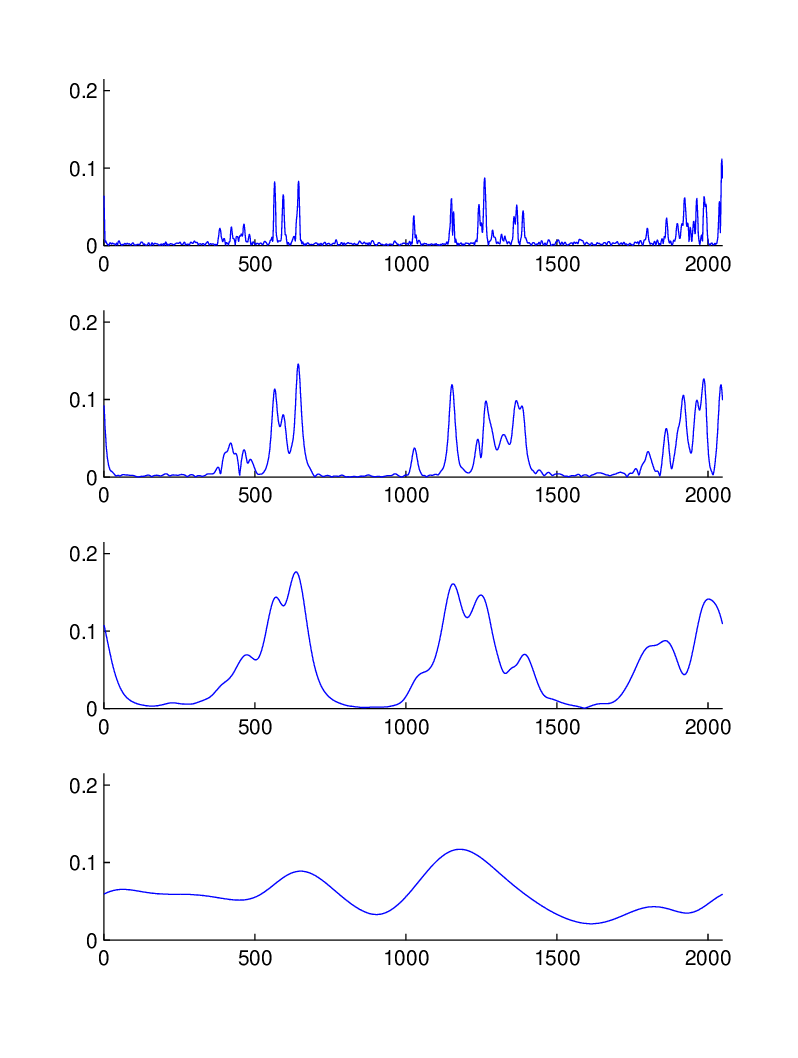}\\
(c)
\end{center}
\end{minipage}
\caption{wavelet transforms, in modulus, of the signals of the figure \ref{fig:three_classes}: (a) realization of a gaussian process (b) sum of sinusoids (c) piecewise regular\\
Each column represents the wavelet transform of one signal. Each graph corresponds to one frequency component of the wavelet transform. For sake of lisibility, only $4$ components are shown, although nine were used in the calculation.
\label{fig:wavelet_transforms}}
\end{figure}

For almost each of the signals that we consider, there exist $x$'s such that $f\star\psi_{k_j}(x)\approx 0$ but the number of such points vary greatly, depending on which class the signal belongs. As an example, the wavelet transforms of the three signals of the figure \ref{fig:three_classes} are displayed in \ref{fig:wavelet_transforms}.

For gaussian signals, there are generally not many points at which the wavelet transform vanishes. The positions of these points do not seem to be correlated in either space or frequency.

For piecewise regular signals, there are more of this points but they are usually distributed in such a way that if $f\star\psi_j(x)\approx 0$, then $f\star\psi_k(x)\approx 0$ for all wavelets $\psi_k$ of higher frequencies than $\psi_j$. This repartition makes the reconstruction easier.

When the signals are sums of sinusoids, it often happens that some components of the wavelet transform are totally negligible: for some $j$, $f\star\psi_j(x)\approx 0$ for any $x$. The negligible frequencies may be either high, low or intermediate.

\nl

From the results shown in \ref{fig:results}, it is clear that the number of zeros influences the reconstruction, but also that isolated zeroes do not prevent reconstruction. The algorithm performs well on gaussian or piecewise regular signals. The distance in modulus between the wavelet transform
 of the reconstructed signal and of the original one is proportional to the amount of noise (and generally significantly smaller). This holds up to large levels of noise (10\%). By comparison, the classical Gerchberg-Saxton algorithm is much less efficient.

However, the algorithm often fails when the input signal is a sum of sinusoids. Not surprisingly, the most difficult signals in this class are the ones for which the sinusoids are not equally distributed among frequency bands and the wavelet transform has a lot of zeroes. The relative error over the modulus of the wavelet transform is then often of several percent, even when the relative error induced by the noise is of the order of 0.1\%.

\nl
In the section \ref{s:non_uniform_continuity}, we explained why, for any function $f$, it is generally possible to construct $g$ such that $f$ and $g$ are not close but their wavelet transform have almost the same modulus. This construction holds provided that the time and frequency support of $f$ is large enough.

Increasing the time and frequency support of $f$ amounts here to increase the size $N$ of the signals. Thus, we expect the function error to increase with $N$. It is indeed the case but this effect is very weakly perceptible on gaussian signals. It is stronger on piecewise regular functions, probably because the wavelet transforms of these signals have more zeroes; their reconstruction is thus less stable.



In the case of the sums of sinusoids, because of the failure of the algorithm, we can not draw firm conclusions regarding the stability of the reconstruction. We nevertheless suspect that this class of signals is the less stable of all and that these instabilities are the cause of the incorrect behavior of our algorithm.

\section{Conclusion}

In this text, we have studied the phase retrieval problem in which one tries to reconstruct a function from the modulus of its Cauchy wavelet transform. We have shown that the reconstruction was unique, up to a global phase, and that the reconstruction operator was continuous but not uniformly continuous. Indeed, if we modulate the wavelet transform of a function by slow-varying phases, we can construct very different functions with almost the same wavelet transform, in modulus. Moreover, in the case where the wavelet transform does not take values too close to zero, all the instabilities of the reconstruction are of this form.

Our proofs are specific to Cauchy wavelets and cannot be extended to generic wavelets because they strongly use the link between Cauchy wavelets and holomorphic functions. Only the description of instabilities of the reconstruction operator (section \ref{s:non_uniform_continuity}) is independent of the choice of the wavelet family (actually, it could also be extended to other time-frequency representations that the wavelet transform). However, in practice, the Cauchy wavelets do not seem to behave differently from other wavelets. We expect that the unicity and stability results are true for much generic wavelets than Cauchy ones but we do not know how to prove it.

\nl
\noindent
\textbf{Acknowledgedments.} Many thanks to A. Bandeira and D. Mixon for their helpful correspondence. This work has been partially supported by ERC grant InvariantClass 320959.

\appendix

\section{Lemmas of the proof of theorem \ref{thm:unicity_hol} \label{app:lemmas_unicity_hol}}

\begin{proof}[Proof of lemma \ref{lem:prolongation}]
We recall equation \eqref{eq:relation_zeros}:
\begin{equation*}
\tag{\ref{eq:relation_zeros}}
\frac{B_F(z+i\alpha)\tilde B_G(z)}{B_G(z+i\alpha)\tilde B_F(z)}
=e^{iC+iBz}\frac{S_G(z+i\alpha)}{S_F(z+i\alpha)}
\end{equation*}
We want to show that the left part of this equality admits a meromorphic extension to $\C$. We also want this meromorphic extension to have the same poles (with multiplicity) than it would if all four fonctions $B_F,B_G,\tilde B_F$ and $\tilde B_G$ were meromorphically defined over all $\C$.

We first remark that $\tilde B_F$ and $\tilde B_G$ admit meromorphic extensions to $\C$. Indeed, if the $(z_k)_k$ are the zeros of $F(.+i\alpha)$ in $\H$, this set has no accumulation point in $\overline{H}$: if $z_\infty$ was an accumulation point, $z_\infty+i\alpha\in\H$ would be an accumulation point of the zeros of $F$ and, as $F$ is holomorphic, it would be the null function. From the classical properties of Blaschke products, $\tilde B_F$ converge over $\C$ and so does $\tilde B_G$.

On the contrary, $B_F$ and $B_G$ may not admit meromorphic extensions over $\C$. But their quotient $B_F/B_G$ does.

We define:
\begin{equation*}
B_F'(z)=\left(\frac{z-i}{z+i}\right)^{m_F}\underset{k}{\prod}\frac{|z^F_k-i|}{z^F_k-i}\frac{|z^F_k+i|}{z^F_k+i}\frac{z-z^F_k}{z-\overline{z}^F_k}
\end{equation*}
where the $(z^F_k)$'s are the zeros of $F$, each $z^F_k$ being counted, not with multiplicity $\mu_F(z^F_k)$, but with multiplicity $\max(0,\mu_F(z^F_k)-\mu_G(z^F_k))$ (and $m_F$ is still the multiplicity of $i$ as a zero of $F$).

Similarly:
\begin{equation*}
B_G'(z)=\left(\frac{z-i}{z+i}\right)^{m_G}\underset{k}{\prod}\frac{|z^G_k-i|}{z^G_k-i}\frac{|z^G_k+i|}{z^G_k+i}\frac{z-z^G_k}{z-\overline{z}^G_k}
\end{equation*}
where the $(z_k^G)$'s are the zeros of $G$ counted with multiplicity $\max(0,\mu_G(z_k^G)-\mu_F(z_k^G))$.

We define:
\begin{equation*}
B_{F,G}(z)=\underset{k}{\prod}\frac{|z^{F,G}_k-i|}{z^{F,G}_k-i}\frac{|z^{F,G}_k+i|}{z^{F,G}_k+i}\frac{z-z^{F,G}_k}{z-\overline{z}^{F,G}_k}
\end{equation*}
where the $z_k^{F,G}$ are the zeros of $F$ or $G$, counted with multiplicity $\min(\mu_F(z_k^{F,G}),\mu_G(z_k^{F,G}))$. The function $B_{F,G}$ corresponds to the ``common part'' of $B_F$ and $B_G$, which we may factorize in the quotient $B_F/B_G$.

The products $B_F',B_G',B_{F,G}$ converge over $\H$ and, for all $z\in\H$:
\begin{equation*}
B_F(z)=B'_F(z)B_{F,G}(z)\quad\quad
B_G(z)=B'_G(z)B_{F,G}(z)
\end{equation*}
So for all $z\in\H$:
\begin{equation*}
\frac{B_F(z+i\alpha)\tilde B_G(z)}{B_G(z+i\alpha)\tilde B_F(z)}=
\frac{B'_F(z+i\alpha)\tilde B_G(z)}{B'_G(z+i\alpha)\tilde B_F(z)}
\end{equation*}
If we show that $B'_F$ and $B'_G$ converge over $\C$, we can take $B_w(z)=\frac{B'_F(z+i\alpha)\tilde B_G(z)}{B'_G(z+i\alpha)\tilde B_F(z)}$. It will be meromorphic over $\C$.

\nl

To prove this, we first establish a relation between the zeros of $F$ and $G$.

Let $z$ be such that $0<\Im z\leq \alpha$. The zeros of $B_F$ are the zeros of $F$ in $\H$, counted with multiplicity. Thus, $z-i\alpha$ is a zero of $B_F(.+i\alpha)$ with multiplicity $\mu_F(z)$. It is a zero of $B_G(.+i\alpha)$ with multiplicity $\mu_G(z)$.

Because $\Im (z-i\alpha)\leq 0$, it is not a zero of $\tilde B_F$ (resp. $\tilde B_G$) but may be a pole. As a pole, its multiplicity is the multiplicity of $\overline{z-i\alpha}=\overline{z}+i\alpha$ as a zero of $F(.+i\alpha)$ (resp. $G(.+i\alpha)$): it is $\mu_F(\overline{z}+2i\alpha)$ (resp. $\mu_G(\overline{z}+2i\alpha)$).

The right part of \eqref{eq:relation_zeros}, $e^{iC+iBz} \frac{S_G(z+i\alpha)}{S_F(z+i\alpha)}$ has no zero neither pole over $\{z\in\C\mbox{ s.t. }\Im z>-\alpha\}$ (from the definition of $S_G$ and $S_F$ given in \eqref{eq:def_S}). So neither does the left part. In particular, $z-i\alpha$ is not a zero and is not a pole:
\begin{equation}\label{eq_app:part_relation_zeros}
\mu_F(z)-\mu_G(z)-\mu_G(\overline{z}+2i\alpha)+\mu_F(\overline{z}+2i\alpha)=0
\end{equation}

We now explain why $B'_F$ converges over $\C$. The same result will hold for $B'_G$. From the properties of Blaschke products, $B'_F$ converges over $\C$ if $(z_k^F)$ has no accumulation point in $\R$.

By contradiction, we assume that some subsequence of $(z_k^F)$, denoted by $(z_{\phi(k)}^F)$, converges to $\lambda\in\R$. Because the $z_k^F$'s appear in $B'_F$ with multiplicity $\max(0,\mu_F(z_k^F)-\mu_G(z_k^F))$, we must have:
\begin{equation*}
\mu_F(z_{\phi(k)}^F)-\mu_G(z_{\phi(k)}^F)>0\quad\quad
\forall k\in\N
\end{equation*}
We can assume that, for all $k$, $0<\Im z_{\phi(k)}^F\leq\alpha$. From \eqref{eq_app:part_relation_zeros}:
\begin{equation*}
\mu_G(\overline{z}_{\phi(k)}^F+2i\alpha)-\mu_F(\overline{z}_{\phi(k)}^F+2i\alpha)
=\mu_F(z_{\phi(k)}^F)-\mu_G(z_{\phi(k)}^F)>0
\end{equation*}
Consequently, $\overline{z}_{\phi(k)}^F+2i\alpha$ is a zero of $G$ for all $k$. As $z_{\phi(k)}^F\to\lambda\in\R$, $\lambda+2i\alpha\in\H$ is an accumulation point of the zeros of $G$. This is impossible because $G$ is holomorphic over $\H$ and we have assumed that it was not the null function.

\nl
To conclude, we have to prove the equation \eqref{eq:formule_mult}.

For any $z\in\H$, the multiplicity of $\overline{z}-i\alpha$ as a pole of $B_F'(.+i\alpha)$ is the multiplicity of $z$ as a zero of $B_F'$, that is $\max(0,\mu_F(z)-\mu_G(z))$. Its multiplicity as a pole of $B_G'(.+i\alpha)$ is $\max(0,\mu_G(z)-\mu_F(z))$. As a pole of $\tilde B_F$ (resp. $\tilde B_G$), it is $\mu_F(z+2i\alpha)$ (resp. $\mu_G(z+2i\alpha)$).

The multiplicity of $\overline{z}-i\alpha$ as a pole of $B_w$ is then, as required:
\begin{gather*}
\max(0,\mu_F(z)-\mu_G(z))-\max(0,\mu_G(z)-\mu_F(z))-\mu_F(z+2i\alpha)+\mu_G(z+2i\alpha)\\
=(\mu_F(z)-\mu_G(z))-(\mu_F(z+2i\alpha)-\mu_G(z+2i\alpha))
\end{gather*}

\end{proof}

\begin{proof}[Proof of lemma \ref{lem:S_1}]
We call $dE_F$ and $dE_G$ the singular measures appearing in the definitions of $S_F$ and $S_G$ (see \eqref{eq:def_S}).

From equation \eqref{eq:relation_zeros} and lemma \ref{lem:prolongation}, for any $z\in\H$:
\begin{equation*}
\exp\left(\frac{i}{\pi}\int_\R\frac{1+tz}{t-z}(dE_G-dE_F)(t)\right)
=\frac{S_G(z)}{S_F(z)}=B_w(z-i\alpha)e^{-iC-iB(z-i\alpha)}
\end{equation*}
The function $z\to B_w(z-i\alpha)e^{-iC-iB(z-i\alpha)}$ is meromorphic over $\C$. From the following lemma, $dE_G-dE_F$ must then be the null measure, so $S_G=S_F$ over $\H$.

\begin{lem}\label{lem_app:dE_zero}
Let $dE$ be a real bounded measure, singular with respect to Lebesgue measure. We define:
\begin{equation*}
S(z)=\exp\left(\frac{i}{\pi}\int_\R\frac{1+tz}{t-z}dE(t)\right)
\quad\quad\forall z\in\H
\end{equation*}
If $S$ admits a meromorphic extension in the neighborhood of each point of $\R$, then $dE=0$.
\end{lem}
\begin{proof}
Let $s(z)=-\log|S(z)|$ for all $z\in\H$. This is well-defined and:
\begin{equation*}
s(x+iy)=\frac{1}{\pi}\int_\R\frac{y}{(t-x)^2+y^2}(1+t^2)dE(t)
\quad\quad\forall x,y\in\R\mbox{ s.t. }y>0
\end{equation*}

The lemma \ref{lem_app:lim_L1} states that $(1+t^2)dE(t)$ is the limit, in the sense of distributions, of $s(t+iy)dt$ when $y\to 0^+$. The principle of the proof will then be to show that $s(.+iy)$ also converges to $-\log|S_{|\R}|$, where $S_{|\R}$ is the extension of $S$ to $\R$, so $dE=-\frac{\log|S(t)|dt}{1+t^2}$. The singularity of $dE$ will imply $\log|S_{|\R}|=0$ and $dE=0$.

\begin{lem}\label{lem_app:lim_L1}
Let $dE$ be a real measure such that $\frac{dE(t)}{1+t^2}$ is bounded. Let:
\begin{equation*}
s(x+iy)=\frac{1}{\pi}\int_\R\frac{y}{(t-x)^2+y^2}dE(t)
\quad\quad\forall x,y\in\R\mbox{ s.t. }y>0
\end{equation*}
For all continuous compactly-supported $f\in\mathcal{C}^0_c(\R)$:
\begin{equation*}
\int_\R f(t)dE(t)=\underset{y\to 0^+}{\lim}\int_\R s(t+iy)f(t)dt
\end{equation*}
\end{lem}
\begin{proof}
\begin{align}
\int_\R s(t+iy)f(t)dt&=\frac{1}{\pi}\iint_{\R}\frac{yf(t)}{(t'-t)^2+y^2}dE(t')dt
\nonumber\\
&=\frac{1}{\pi}\iint_\R \frac{yf(t')}{(t'-t)^2+y^2}dE(t')dt+\frac{1}{\pi}\iint_\R\frac{y(f(t')-f(t))}{(t'-t)^2+y^2}dE(t')dt
\nonumber\\
&=\int_\R f(t)dE(t)+\frac{1}{\pi}\iint_\R\frac{y(f(t')-f(t))}{(t'-t)^2+y^2}dE(t')dt
\label{eq_app:convergence_L1}
\end{align}
For all $y,\epsilon>0,t'\in\R$:
\begin{align*}
\left|\int_\R\frac{y(f(t')-f(t))}{(t'-t)^2+y^2}dt\right|
&\leq \left(\underset{|x_1-x_2|\leq\epsilon}{\sup}|f(x_1)-f(x_2)|\right)
\int_\R\frac{y}{(t'-t)^2+y^2}dt\\
&\quad +2(\sup|f|)\int_{|t-t'|>\epsilon}\frac{y}{(t'-t)^2+y^2}dt\\
&= \pi\left(\underset{|x_1-x_2|\leq\epsilon}{\sup}|f(x_1)-f(x_2)|\right)\\
&\quad +2(\sup|f|)\int_{|u|>\epsilon/y}\frac{1}{1+u^2}du
\end{align*}
The second term of the last sum tends to $0$ when $y\to 0^+$, uniformly in $t'$ so:
\begin{equation*}
\underset{y\to 0^+}{\limsup}\left|\int_\R\frac{y(f(t')-f(t))}{(t'-t)^2+y^2}dt\right|
\leq\pi\left(\underset{|x_1-x_2|\leq\epsilon}{\sup}|f(x_1)-f(x_2)|\right)
\end{equation*}
which tends to $0$ when $\epsilon\to 0$ because $f$ is uniformly continuous. Convergence is uniform in $t'$.

Moreover, if $K$ is the compact support of $f$ and $|K|$ is its Lebesgue measure, then, for all $t'\notin K$:
\begin{align*}
\left|\int_\R\frac{y(f(t')-f(t))}{(t'-t)^2+y^2}dt\right|
&=\left|\int_\R\frac{yf(t)}{(t'-t)^2+y^2}dt\right|\\
&\leq(\sup|f|)\int_K\frac{y}{(t'-t)^2+y^2}dt\\
&\leq (\sup|f|)\int_K\frac{y}{d(t',K)^2+y^2}dt\\
&=(\sup |f|)\frac{y}{d(t',K)^2+y^2}|K|
\end{align*}

It implies that the second term of \eqref{eq_app:convergence_L1} tends to $0$. Let $K'=\{t'\in\R\mbox{ s.t. }d(t',K)\leq 1\}$.
\begin{align*}
\iint_\R\frac{y(f(t')-f(t))}{(t'-t)^2+y^2}dE(t')dt
&\leq\underset{t'\in K'}{\sup} \left|\int_\R\frac{y(f(t')-f(t))}{(t'-t)^2+y^2}dt\right|dE(K')\\
&\quad+|K|(\sup|f|)\int_{t'\notin K'}\frac{y}{d(t',K)^2+y^2}dE(t')
\end{align*}
Because $\frac{dE(t')}{1+t'^2}$ is bounded, $\int_{t'\notin K'}\frac{dE(t')}{d(t',K)^2+y^2}$ is bounded when $y\to 0^+$. So the last expression tends to $0$.

The equation \eqref{eq_app:convergence_L1} then implies the result.
\end{proof}

We still denote by $S(t)$ the meromorphic extension of $S$ to a neighborhood of $\overline{H}$. Let $\{r_k\}$ be the zeros or poles of $S$.

When $y\to 0^+$, $s(.+iy)$ tends to $-\log|S|$ almost everywhere. On every compact of $\R-\{r_k\}$, the convergence is uniform, and thus in $L^1$.

Let $r_k$ be any zero or pole and $\epsilon>0$ be such that $S$ admits a meromorphic extension over a neighborhood of $[r_k-\epsilon;r_k+\epsilon]\times[-\epsilon;\epsilon]$ and $r_j\notin [r_k-\epsilon;r_k+\epsilon]$ for all $j\ne k$. There exist $h:[r_k-\epsilon;r_k+\epsilon]\times[-\epsilon;\epsilon]\to\C$ holomorphic and $m\in\Z$ such that:
\begin{equation*}
S(z)=(z-r_k)^mh(z)\quad\forall z\in [r_k-\epsilon;r_k+\epsilon]\times[-\epsilon;\epsilon]
\quad\quad\mbox{and}\quad
h(r_k)\ne 0
\end{equation*}
For all $y\in]0;\epsilon[$:
\begin{align}
\int_{r_k-\epsilon}^{r_k+\epsilon}|s(t+iy)+\log|S(t)||dt
&=\int_{r_k-\epsilon}^{r_k+\epsilon}\big|m\log|t-r_k+iy|+\log|h(t+iy)| \nonumber\\
&\quad\quad-m\log|t-r_k|-\log|h(t)|\big|dt\nonumber\\
&\leq m\int_{r_k-\epsilon}^{r_k+\epsilon}\big|\log|t-r_k+iy|-\log|t-r_k|\big|dt\nonumber\\
&\quad+\int_{r_k-\epsilon}^{r_k+\epsilon}\big|\log|h(t+iy)|-\log|h(t)|\big|dt
\label{eq_app:convergence_S_L1}
\end{align}
As $\log|h|$ is continuous, $\log|h(.+iy)|$ converges uniformly to $\log|h_{|\R}|$ over $[r_k-\epsilon;r_k+\epsilon]$:
\begin{equation*}
\int_{r_k-\epsilon}^{r_k+\epsilon}\big|\log|h(t+iy)|-\log|h(t)|\big|dt\to 0\quad
\mbox{when }y\to 0^+
\end{equation*}
As $\log|.-r_k+iy|$ converges to $\log|.-r_k|$ in $L^1([r_k-\epsilon;r_k+\epsilon])$:
\begin{equation*}
\int_{r_k-\epsilon}^{r_k+\epsilon}\big|\log|t-r_k+iy|-\log|t-r_k|\big|dt\to 0
\end{equation*}
So, by \eqref{eq_app:convergence_S_L1}, $s(.+iy)$ converges in $L^1$ to $t\in\R\to-\log|S(t)|$, over $[r_k-\epsilon;r_k+\epsilon]$. As the sequence $(r_k)$ has no accumulation point in $\R$, $s(.+iy)\to-\log|S_\R|$ (in $L^1$) over each compact set of $\R$.

By the lemma \ref{lem_app:lim_L1}, for all $f\in\mathcal{C}^0_c(\R)$:
\begin{equation*}
\int_\R f(t)(1+t^2)dE(t)=\underset{y\to 0^+}{\lim}\int_\R s(t+iy)f(t)dt=-\int_\R\log|S(t)|f(t)dt
\end{equation*}
We deduce that $dE(t)=-\frac{\log|S(t)|dt}{1+t^2}$. As $dE$ is singular with respect to Lebesgue measure, we must have $\log|S(t)|=0$ for all $t\in\R$ and $dE=0$.
\end{proof}

\end{proof}

\section{Lemmas of the proof of theorem \ref{thm:weak_stability} \label{app:lemmas_weak_stability}}

\begin{proof}[Proof of lemma \ref{lem:tot_bound}]
We first recall the Riesz-Fr\'echet-Kolmogorov theorem.
\begin{thm*}[Riesz-Fr\'echet-Kolomogorov]
Let $p\in[1;+\infty[$. Let $\mathcal{F}$ be a subset of $L^p(\R)$. The set $\mathcal{F}$ is relatively compact if and only if:
\begin{enumerate}
\item[(i)] $\mathcal{F}$ is bounded.
\item[(ii)] For every $\epsilon>0$, there exists some compact $K\subset \R$ such that:
\begin{equation*}
\underset{f\in\mathcal{F}}{\sup}\,||f||_{L^p(\R-K)}\leq\epsilon
\end{equation*}
\item[(iii)] For every $\epsilon>0$, there exists $\delta>0$ such that:
\begin{equation*}
\underset{f\in\mathcal{F}}{\sup}\,||f(.+h)-f||_p\leq\epsilon
\quad\quad\forall h\in[-\delta;\delta]
\end{equation*}
\end{enumerate}
\end{thm*}

We want to apply this theorem to $p=2$ and $\mathcal{F}=\{f_n\star\psi_j\}_{n\in\N}$.

First of all, $\mathcal{F}$ is bounded: actually, from \eqref{eq:U_bounds}, $(f_n)_{n\in\N}$ itself is bounded (because $(U(f_n))_n$ converges and thus is bounded). It implies that $\{f_n\star\psi_j\}_n$ is bounded because $||f_n\star\psi_j||_2\leq||f_n||_2||\psi_j||_1$ (by Young's inequality).

Let us now prove (ii). Let any $\epsilon>0$ be fixed.

The sequence $\left(|f_n\star\psi_j|\right)_n$ converges in $L^2(\R)$ (to $h_j$, because $U(f_n)\to (h_j)_{j\in\Z}$ in $L^2_\Z(\R)$). So $\{|f_n\star\psi_j|\}_n$ is relatively compact in $L^2(\R)$. By the Riesz-Fr\'echet-Kolmogorov theorem, there exists $K\subset\R$ a compact set such that:
\begin{equation*}
\underset{n\in\N}{\sup\,}||\,|f_n\star\psi_j|\,||_{L^2(\R-K)}\leq \epsilon
\end{equation*}
But, for all $n$, $||\,|f_n\star\psi_j|\,||_{L^2(\R-K)}=||f_n\star\psi_j||_{L^2(\R-K)}$ so (ii) holds:
\begin{equation*}
\underset{n\in\N}{\sup\,}||f_n\star\psi_j||_{L^2(\R-K)}\leq \epsilon
\end{equation*}

We finally check (iii). Let $\epsilon>0$ be fixed. For any $h\in\R$:
\begin{equation*}
||(f_n\star\psi_j)(.+h)-(f_n\star\psi_j)||_2=||f_n\star(\psi_j(.-h)-\psi_j)||_2\leq||f_n||_2||\psi_j(.-h)-\psi_j||_1
\end{equation*}
As $\underset{n}{\sup\,}||f_n||_2<+\infty$ and $\underset{h\to 0}{\lim}\,||\psi_j(.-h)-\psi_j||_1=0$ (this property holds for any $L^1$ function), we have, for $\delta>0$ small enough:
\begin{equation*}
\underset{n}{\sup\,}||(f_n\star\psi_j)(.+h)-(f_n\star\psi_j)||_2\leq\epsilon
\quad\quad\forall h\in[-\delta;\delta]
\end{equation*}
\end{proof}

\begin{proof}[Proof of lemma \ref{lem:existence_g}]
We want to find $g\in L^2_+(\R)$ such that $\hat l_j=\hat g\hat\psi_j$ for every $j\in\Z$.

If $\omega\leq 0$, we set $\hat g(\omega)=0$. Then, for each $j$, we set $\hat g=\hat l_j/\hat\psi_j$ on the support of $\hat\psi_j$, which we denote by $\mbox{Supp }\hat\psi_j$. This definition is correct in the sense that:
\begin{equation*}
\mbox{if }j_1\ne j_2,\quad\frac{\hat l_{j_1}}{\hat\psi_{j_1}}=\frac{\hat l_{j_2}}{\hat\psi_{j_2}}\mbox{ a.e. on }\mbox{Supp }\hat\psi_{j_1}\cap \mbox{Supp }\hat\psi_{j_2}
\end{equation*}
Indeed, for all $n$, $(f_{\phi(n)}\star\psi_{j_1})\star\psi_{j_2}=(f_{\phi(n)}\star\psi_{j_2})\star\psi_{j_1}$ so, by taking the limit in $n$, $l_{j_1}\star\psi_{j_2}=l_{j_2}\star\psi_{j_1}$ and $\hat l_{j_1}\hat\psi_{j_2}=\hat l_{j_2}\hat\psi_{j_1}$.

We can note that, for all $j$, $\hat g\hat\psi_j=\hat l_j$. It is true on $\mbox{Supp }\hat\psi_j$, by definition. And, on $\R-\mbox{Supp }\hat\psi_j$, $\hat l_j=0=\hat g\hat\psi_j$ because $\hat l_j$ is the $L^2$-limit of $\hat f_{\phi(n)}\hat\psi_j$ and $\hat f_{\phi(n)}\hat\psi_j=0$ on $\R-\mbox{Supp }\hat\psi_j$.

The $\hat g$ we just defined belongs to $L^2(\R)$. Indeed, by \eqref{eq:approx_littlewood_paley}:
\begin{equation*}
||\hat g||_2^2\leq \frac{1}{B}\int_{\R^+}|\hat g|^2\underset{j}{\sum}|\hat\psi_j|^2=\frac{1}{B}\int_{\R^+}\underset{j}{\sum}|\hat l_j|^2=\frac{1}{B}\underset{j}{\sum}||l_j||_2^2
\end{equation*}
As $f_{\phi(n)}\star\psi_j$ goes to $l_j$ when $n$ goes to $\infty$ and $U(f_{\phi(n)})=\{|f_{\phi(n)}\star\psi_{j}|\}_j$ goes to $(h_j)_{j\in\Z}\in L^2_\Z(\R)$, we must have $|l_j|=h_j$ for each $j$. So $\frac{1}{B}\underset{j}{\sum}||l_j||_2^2=\frac{1}{B}\underset{j}{\sum}||h_j||_2^2=\frac{1}{B}||(h_j)_{j\in\Z}||_2^2<+\infty$ and $\hat g$ belongs to $L^2(\R)$.

As $\hat g\in L^2(\R)$, it is the Fourier transform of some $g\in L^2(\R)$. For all $j\in\Z$, as $\hat g\hat\psi_j=\hat l_j$, we have $g\star\psi_j=l_j$.

We now show that $f_{\phi(n)}\to g$ when $n\to\infty$.

For every $J,n\in\N$:
\begin{align*}
\sqrt{\underset{|j|>J}{\sum}||f_{\phi(n)}\star\psi_j||_2^2}
&=\sqrt{\underset{|j|>J}{\sum}||\,U(f_{\phi(n)})_j\,||_2^2}\\
&\leq\sqrt{\underset{|j|>J}{\sum}||U(f_{\phi(n)})_j-h_j||_2^2}+
\sqrt{\underset{|j|>J}{\sum}||h_j||_2^2}\\
&\leq||U(f_{\phi(n)})-(h_j)||_2+\sqrt{\underset{|j|>J}{\sum}||h_j||_2^2}
\end{align*}
So $\underset{n}{\limsup}\left(\underset{|j|>J}{\sum}||f_{\phi(n)}\star\psi_j||_2^2\right)\leq \underset{|j|>J}{\sum}||h_j||_2^2$ and:
\begin{align*}
\underset{n}{\limsup}\left(\underset{j\in\Z}{\sum}||f_{\phi(n)}\star\psi_j-g\star\psi_j||_2^2\right)&
\leq\underset{n}{\limsup}\left(\underset{|j|\leq J}{\sum}||f_{\phi(n)}\star\psi_j-g\star\psi_j||_2^2\right)\\
&\quad+\underset{n}{\limsup}\left(\underset{|j|>J}{\sum}||f_{\phi(n)}\star\psi_j-g\star\psi_j||_2^2\right)\\
&=\underset{n}{\limsup}\left(\underset{|j|>J}{\sum}||f_{\phi(n)}\star\psi_j-g\star\psi_j||_2^2\right)\\
&\leq \underset{|j|>J}{\sum}||h_j||_2^2
\end{align*}
This last quantity may be as small as desired, for $J$ large enough, so $\underset{j\in\Z}{\sum}||f_{\phi(n)}\star\psi_j-g\star\psi_j||_2^2\to 0$.

\noindent By \eqref{eq:approx_littlewood_paley}:
\begin{align*}
\underset{j\in\Z}{\sum}||f_{\phi(n)}\star\psi_j-g\star\psi_j||_2^2
&=\int_{\R}\left|\hat f_{\phi(n)}-\hat g\right|^2(\underset{j}{\sum}|\hat\psi_j|^2)\\
&\geq A \int_{\R}\left|\hat f_{\phi(n)}-\hat g\right|^2\\
&=\frac{A}{2\pi}||f_{\phi(n)}-g||_2^2
\end{align*}
so $||f_{\phi(n)}-g||_2\to 0$.
\end{proof}

\section{Proof of theorem \ref{thm:stability_dyadic_case}\label{app:stability_dyadic_case}}

In this section, we prove the theorem \ref{thm:stability_dyadic_case}, which gives a stability result for the case of dyadic wavelets.

For all $y>0$, we define:
\begin{equation*}
\mathcal{N}(y)=\underset{x\in\R,s=1,2}{\sup}|F^{(s)}(x+iy)|
\end{equation*}

The following lemma is not necessary to our proof but we will use it to progressively simplify our inequalities.
\begin{lem}\label{lem:relations_for_N}
For all $y_1,y_2\in\R^*_+$, if $y_1<y_2$:
\begin{equation}\label{eq:N_non_increasing}
\mathcal{N}(y_1)\geq\mathcal{N}(y_2)
\end{equation}
and for all $y_3\in[y_1;y_2]$:
\begin{equation}\label{eq:Ny123}
\mathcal{N}(y_3)\leq\mathcal{N}(y_1)^{\frac{y_2-y_3}{y_2-y_1}}\mathcal{N}(y_2)^{\frac{y_3-y_1}{y_2-y_1}}
\end{equation}
\end{lem}
\begin{proof}
The second inequality comes directly from theorem \ref{thm:bounded_on_a_band}, applied to functions $F^{(1)}$ and $F^{(2)}$ on the band $\{z\in\C\mbox{ s.t. }y_1<\Im z<y_2\}$.

The first inequality may be derived from the first one. The function $\mathcal{N}(y)$ is bounded when $y\to+\infty$. 
Keeping $y_1$ and $y_3$ fixed in \eqref{eq:Ny123} and letting $y_2$ go to $+\infty$ then gives:
\begin{equation*}
\mathcal{N}(y_3)\leq\mathcal{N}(y_1)
\end{equation*}
\end{proof}

We can now prove the theorem.

\begin{proof}[Proof of theorem \ref{thm:stability_dyadic_case}]
From the relation \eqref{eq:recall_F_f_psi_j} between $F^{(s)}$ and the $f^{(s)}\star\psi_j$ and from the hypotheses, the following inequalities hold for all $x\in[-M 2^j;M2^j]$:
\begin{gather*}
\left||F^{(1)}(x+i2^j)|^2-|F^{(2)}(x+i2^j)|^2\right|\leq\epsilon\mathcal{N}(2^j)^2\\
\left||F^{(1)}(x+i2^{j+1})|^2-|F^{(2)}(x+i2^{j+1})|^2\right|\leq\epsilon\mathcal{N}(2^{j+1})^2\\
|F^{(1)}(x+i2^j)|^2,|F^{(2)}(x+i2^j)|^2\geq c\mathcal{N}(2^j)^2\\
|F^{(1)}(x+i2^{j+1})|^2,|F^{(2)}(x+i2^{j+1})|^2\geq c\mathcal{N}(2^{j+1})^2
\end{gather*}

Let us set, for all $z$ such that $-2^{j+1}<\Im z<2^{j+1}$:
\begin{equation*}
G(z)=F^{(1)}(z+i2^{j+1})\overline{F^{(1)}(\overline{z}+i2^{j+1})}-F^{(2)}(z+i2^{j+1})\overline{F^{(2)}(\overline{z}+i2^{j+1})}
\end{equation*}
For all $z$ such that $\Im z=0$:
\begin{align*}
|G(z)|=\left||F^{(1)}(z+i2^{j+1})|^2-|F^{(2)}(z+i2^{j+1})|^2\right|
&\leq \epsilon\mathcal{N}(2^{j+1})^2
&\mbox{if }|\Re z|\leq M2^j\\
&\leq \mathcal{N}(2^{j+1})^2
&\mbox{if }|\Re z|>M2^j
\end{align*}
and for all $z$ such that $\Im z=3.2^{j-1}$:
\begin{align*}
|G(z)|&=|F^{(1)}(\Re z+7.2^{j-1}i)\overline{F^{(1)}(\Re z+2^{j-1}i)}-F^{(2)}(\Re z+7.2^{j-1}i)\overline{F^{(2)}(\Re z+2^{j-1}i)}|\\
&\leq 2\,\mathcal{N}(7.2^{j-1})\mathcal{N}(2^{j-1})
\end{align*}
We apply the lemma \ref{lem:hole_on_one_line} for $a=0,b=3.2^{j-1},t=2/3,A=\mathcal{N}(2^{j+1})^2,B=2\,\mathcal{N}(7.2^{j-1})\mathcal{N}(2^{j-1})$. It implies that, for all $x\in[-\lambda M2^j;\lambda M2^j]$:
\begin{equation*}
|G(x+i 2^j)|\leq 2^{2/3}\epsilon^{1/3-\alpha_M}\mathcal{N}(2^{j+1})^{2/3}\mathcal{N}(2^{j-1})^{2/3}\mathcal{N}(7.2^{j-1})^{2/3}
\end{equation*}
where $\alpha_M=\frac{4}{3}\frac{\exp\left(-\frac{2\pi}{3}(1-\lambda)M\right)}{1-\exp\left(-\frac{2\pi}{3}(1-\lambda)M\right)}$.

Replacing $G$ by its definition gives, for all $x\in[-\lambda M2^j;\lambda M2^j]$:
\begin{align*}
|F^{(1)}(x+i3.2^{j})\overline{F^{(1)}(x+i2^j)}&-F^{(2)}(x+i3.2^{j})\overline{F^{(2)}(x+i2^j)}|\\
&\leq 2^{2/3}\epsilon^{1/3-\alpha_M}\mathcal{N}(2^{j+1})^{2/3}\mathcal{N}(2^{j-1})^{2/3}\mathcal{N}(7.2^{j-1})^{2/3}\\
&\leq 2\epsilon^{1/3-\alpha_M}\mathcal{N}(2^{j+1})^{4/3}\mathcal{N}(2^{j-1})^{2/3}
\end{align*}
We used the equation \eqref{eq:N_non_increasing} to obtain the last inequality.

So, for all $x\in[-\lambda M2^j;\lambda M2^j]$:
\begin{align*}
\left|F^{(1)}\right.&(x+i3.2^{j})\overline{F^{(1)}(x+i2^j)}F^{(2)}(x+i2^j)\overline{F^{(2)}(x+i2^j)} \\
&\left.-F^{(2)}(x+i3.2^{j})\overline{F^{(2)}(x+i2^j)}F^{(1)}(x+i2^j)\overline{F^{(1)}(x+i2^j)}\right|\\
&\leq
|F^{(1)}(x+i3.2^{j})\overline{F^{(1)}(x+i2^j)}-F^{(2)}(x+i3.2^{j})\overline{F^{(2)}(x+i2^j)}|.|F^{(2)}(x+i2^j)\overline{F^{(2)}(x+i2^j)}|\\
&+|F^{(2)}(x+i3.2^{j})\overline{F^{(2)}(x+i2^j)}||F^{(2)}(x+i2^j)\overline{F^{(2)}(x+i2^j)}-F^{(1)}(x+i2^j)\overline{F^{(1)}(x+i2^j)}|\\
&\leq
2\epsilon^{1/3-\alpha_M}\mathcal{N}(2^{j+1})^{4/3}\mathcal{N}(2^{j-1})^{2/3}|F^{(2)}(x+i2^j)|^2 
 + \epsilon\mathcal{N}(2^j)^2|F^{(2)}(x+i3.2^{j})\overline{F^{(2)}(x+i2^j)}|
\end{align*}
Dividing by $|\overline{F^{(1)}(x+i2^j)}\overline{F^{(2)}(x+i2^j)}|$ gives:
\begin{align*}
|F^{(1)}(x+i3.2^j)&F^{(2)}(x+i2^j)-F^{(2)}(x+i3.2^j)F^{(1)}(x+i2^j)|\\
&\leq
2\epsilon^{1/3-\alpha_M}\mathcal{N}(2^{j+1})^{4/3}\mathcal{N}(2^{j-1})^{2/3}\frac{|F^{(2)}(x+i2^j)|}{|F^{(1)}(x+i2^j)|} \\
& + \epsilon\mathcal{N}(2^j)^2\frac{|F^{(2)}(x+i3.2^{j})|}{|F^{(1)}(x+i2^j)|}
\end{align*}
For each $x\in[-\lambda M2^j;\lambda M2^j]$, this relation also holds if we switch the roles of $F^{(1)}$ and $F^{(2)}$. Thus, we can assume that $|F^{(2)}(x+i2^j)|\leq|F^{(1)}(x+i2^j)|$. Using also the fact that $|F^{(1)}(x+i2^j)|\geq\sqrt{c}\mathcal{N}(2^j)$ yields (always for $x\in[-\lambda M2^j;\lambda M2^j]$):
\begin{align}
|F^{(1)}(x+i3.2^j)&F^{(2)}(x+i2^j)-F^{(2)}(x+i3.2^j)F^{(1)}(x+i2^j)|\nonumber\\
&\leq
2\epsilon^{1/3-\alpha_M}\mathcal{N}(2^{j+1})^{4/3}\mathcal{N}(2^{j-1})^{2/3}
 + \frac{\epsilon}{\sqrt{c}}\mathcal{N}(2^j)\mathcal{N}(3.2^j)\nonumber\\
&=2\mathcal{N}(2^j)\mathcal{N}(3.2^j)\left(\frac{\mathcal{N}(2^{j+1})^{4/3}\mathcal{N}(2^{j-1})^{2/3}}{\mathcal{N}(2^j)\mathcal{N}(3.2^j)} \epsilon^{1/3-\alpha_M}+\frac{\epsilon}{2\sqrt{c}} \right)
\nonumber\\
&\leq
2\mathcal{N}(2^j)\mathcal{N}(3.2^j)\left(\left(\frac{\mathcal{N}(2^{j-1})}{\mathcal{N}(2^{j+1})}\right)^{2/3} \epsilon^{1/3-\alpha_M}+\frac{\epsilon}{2\sqrt{c}} \right)\nonumber\\
&\leq
3\mathcal{N}(2^j)\mathcal{N}(3.2^j)\left(\frac{\mathcal{N}(2^{j-1})}{\mathcal{N}(2^{j+1})}\right)^{2/3} \epsilon^{1/3-\alpha_M}
\label{eq:dyadic_intermediate}
\end{align}
In the middle, we used the equation \eqref{eq:Ny123}: $\mathcal{N}(2^{j+1})\leq\mathcal{N}(2^j)^{1/2}\mathcal{N}(3.2^j)^{1/2}$. For the last inequality, we used the fact that $c\geq\epsilon$ so $\frac{\epsilon}{2\sqrt{c}}\leq\frac{\sqrt{\epsilon}}{2}\leq\frac{\epsilon^{1/3-\alpha_M}}{2}\leq\left(\frac{\mathcal{N}(2^{j-1})}{\mathcal{N}(2^{j+1})}\right)^{2/3} \frac{\epsilon^{1/3-\alpha_M}}{2}$.

For all $z$ such that $\Im z>-2^j$, we set:
\begin{equation*}
H(z)=F^{(1)}(z+i3.2^j)F^{(2)}(z+i2^j)-F^{(2)}(z+i3.2^j)F^{(1)}(z+i2^j)
\end{equation*}
From \eqref{eq:dyadic_intermediate}:
\begin{align*}
|H(z)|&\leq 2\mathcal{N}(2^j)\mathcal{N}(3.2^j)&\mbox{if }\Im z=0\mbox{ and }|\Re z|>\lambda M2^j\\
&\leq 2\mathcal{N}(2^j)\mathcal{N}(3.2^j)
\min\left(1,\frac{3}{2}\left(\frac{\mathcal{N}(2^{j-1})}{\mathcal{N}(2^{j+1})}\right)^{2/3} \epsilon^{1/3-\alpha_M}\right)&\mbox{if }\Im z=0\mbox{ and }|\Re z|\leq\lambda M2^j\\
&\leq 2\mathcal{N}(2^{j+3})\mathcal{N}(6.2^j)
&\mbox{if }\Im z=5.2^j
\end{align*}
We may apply the lemma \ref{lem:hole_on_one_line} again. For all $x\in[-\lambda^2M2^j;\lambda^2M2^j]$:
\begin{align*}
|H(x+i2^j)|&\leq 2 \min\left(1,\frac{3}{2}\left(\frac{\mathcal{N}(2^{j-1})}{\mathcal{N}(2^{j+1})}\right)^{2/3} \epsilon^{1/3-\alpha_M}\right)^{4/5-\alpha_M'}
\mathcal{N}(2^j)^{4/5}\mathcal{N}(3.2^j)^{4/5}\mathcal{N}(2^{j+3})^{1/5}\mathcal{N}(6.2^j)^{1/5}\\
&\leq 2 \min\left(1,\frac{3}{2}\left(\frac{\mathcal{N}(2^{j-1})}{\mathcal{N}(2^{j+1})}\right)^{2/3} \epsilon^{1/3-\alpha_M}\right)^{4/5-\alpha_M'}
\mathcal{N}(2^j)^{4/5}\mathcal{N}(2^{j+1})^{6/5}
\end{align*}
where $\alpha_M'=\frac{2}{5}\frac{\exp\left(-\frac{\pi}{5}\lambda(1-\lambda)M)\right)}{1-\exp\left(-\frac{\pi}{5}\lambda(1-\lambda)M)\right)}$.

Replacing $H$ by its definition and dividing by |$F^{(1)}(x+i2^{j+1})F^{(2)}(x+i2^{j+1})|$ (which is greater that $c\mathcal{N}(2^{j+1})^2$) gives:
\begin{align*}
\left|\frac{F^{(1)}(x+i2^{j+2})}{F^{(1)}(x+i2^{j+1})}-\frac{F^{(2)}(x+i2^{j+2})}{F^{(2)}(x+i2^{j+1})}\right|
&\leq
\frac{2}{c} \min\left(1,\left(\frac{3}{2}\frac{\mathcal{N}(2^{j-1})}{\mathcal{N}(2^{j+1})}\right)^{2/3} \epsilon^{1/3-\alpha_M}\right)^{4/5-\alpha_M'}\left(\frac{\mathcal{N}(2^j)}{\mathcal{N}(2^{j+1})}\right)^{4/5}
\end{align*}
As soon as $4/5-\alpha_M'>0$ and $1/3-\alpha_M>0$:
\begin{align*}
\left|\frac{F^{(1)}(x+i2^{j+2})}{F^{(1)}(x+i2^{j+1})}-\frac{F^{(2)}(x+i2^{j+2})}{F^{(2)}(x+i2^{j+1})}\right|
&\leq
\frac{3}{c} \left(\frac{\mathcal{N}(2^{j-1})}{\mathcal{N}(2^{j+1})}\right)^{8/15} \left(\frac{\mathcal{N}(2^j)}{\mathcal{N}(2^{j+1})}\right)^{4/5}\epsilon^{(1/3-\alpha_M)(4/5-\alpha_M')}\\
&\leq
\frac{3}{c} \left(\frac{\mathcal{N}(2^{j-1})}{\mathcal{N}(2^{j+1})}\right)^{4/3}\epsilon^{(1/3-\alpha_M)(4/5-\alpha_M')}\\
&=
\frac{3}{c} \left(\frac{N_{j-1}}{N_{j+1}}2^{2p}\right)^{4/3}\epsilon^{(1/3-\alpha_M)(4/5-\alpha_M')}
\end{align*}

So:
\begin{equation*}
\left|\frac{f^{(1)}\star\psi_{j+2}(x)}{f^{(1)}\star\psi_{j+1}(x)}-\frac{f^{(2)}\star\psi_{j+2}(x)}{f^{(2)}\star\psi_{j+1}(x)}\right|
\leq
\frac{3}{c}2^{\frac{11p}{3}} \left(\frac{N_{j-1}}{N_{j+1}}\right)^{4/3}\epsilon^{(1/3-\alpha_M)(4/5-\alpha_M')}
\end{equation*}
which is the desired result for $A=3.2^{\frac{11p}{3}}$.
\end{proof}

\section{Proof of the theorem \ref{thm:stability_a_less_than_2}\label{app:stability_a_less_than_2}}





In this whole section, as in the paragraph \ref{ss:case_a_less_than_2}, $k$ is assumed to be a fixed integer such that:
\begin{equation*}
a^{-k}< 2-a
\end{equation*}
and we define:
\begin{equation*}
c=1-\frac{a-1}{1-a^{-k}}
\end{equation*}

\begin{lem}\label{lem:estimation_from_one_modulus}
Let the following numbers be fixed:
\begin{equation*}
\epsilon\in]0;1[\quad\quad M>0\quad\quad \mu\in[0;M[\quad\quad j\in\Z
\end{equation*}

We assume that, for all $x\in[-Ma^j;Ma^j]$:
\begin{equation*}
\left||F^{(1)}(x+ia^j)|^2-|F^{(2)}(x+ia^j)|^2\right|\leq\epsilon \mathcal{N}(a^j)^2
\end{equation*}
Then, for all $x\in[-(M-\mu)a^j;(M-\mu) a^j]$:
\begin{align*}
\left|\overline{F^{(1)}(x+i(2a^j-a^{j+1}))}F^{(1)}(x+\right.&\left.ia^{j+1})-
\overline{F^{(2)}(x+i(2a^j-a^{j+1}))}F^{(2)}(x+ia^{j+1})\right|\\
&\leq 
\mathcal{N}(a^j)^{2c} \left(2\mathcal{N}(a^{j+1})\mathcal{N}(a^{j-k})\right)^{1-c}\epsilon^{c-\alpha}
\end{align*}
where:
\begin{equation*}
\alpha = 2\,\frac{e^{-\pi\mu}}{1-e^{-\pi \mu}} 
\end{equation*}
\end{lem}
\begin{proof}
We set:
\begin{equation*}
H(z)=\overline{F^{(1)}(\overline{z}+ia^j)}F^{(1)}(z+ia^j)-\overline{F^{(2)}(\overline{z}+ia^j)}F^{(2)}(z+ia^j)
\end{equation*}
When $y=0$, $|H(x+iy)|=\left||F^{(1)}(x+ia^j)|^2-|F^{(2)}(x+ia^j)|^2\right|$. So:
\begin{align*}
|H(x+iy)|&\leq \epsilon\mathcal{N}(a^j)^2\mbox{ if }x\in[-Ma^j;Ma^j]\\
&\leq \mathcal{N}(a^j)^2\mbox{ if }x\notin[-Ma^j;Ma^j]
\end{align*}

When $y=a^j-a^{j-k}$:
\begin{align*}
|H(x+iy)|&=|\overline{F^{(1)}(x+ia^{j-k})}F^{(1)}(x+i(2a^j-a^{j-k}))
-\overline{F^{(2)}(x+ia^{j-k})}F^{(2)}(x+i(2a^j-a^{j-k}))|\\
&\leq 2\mathcal{N}(2a^j-a^{j-k})\mathcal{N}(a^{j-k})
\hskip 2cm (\forall x\in\R)
\end{align*}

We apply the lemma \ref{lem:hole_on_one_line} to $H$, restricted to the band $\{z\in\C\mbox{ s.t. }\Im z\in[0;a^j-a^{j-k}]\}$.

From this lemma, when $y=a^{j+1}-a^j$ and $x\in[-\mu Ma^j;\mu Ma^j]$:
\begin{equation*}
|H(x+iy)|\leq
\epsilon^{f(x+iy)}\mathcal{N}(a^j)^{2c} \left(2\mathcal{N}(2a^j-a^{j-k})\mathcal{N}(a^{j-k})\right)^{1-c}
\end{equation*}
where $c=1-\frac{a-1}{1-a^{-k}}$ and:
\begin{align*}
f(x+iy)\geq c - 2\frac{a-1}{1-a^{-k}}\frac{e^{-\pi\frac{Ma^j-|x|}{a^j-a^{j-k}}}}{1-e^{-\pi\frac{Ma^j-|x|}{a^j-a^{j-k}}}}
\end{align*}
Because of the definition of $k$, $\frac{a-1}{1-a^{-k}}\leq 1$. Moreover, $\frac{Ma^j-|x|}{a^j-a^{j-k}}\geq \frac{\mu}{1-a^{-k}}\geq \mu$, so:
\begin{equation*}
f(x+iy)\geq c-2\frac{e^{-\pi \mu}}{1-e^{-\pi \mu}}=c-\alpha
\end{equation*}
Replacing $H$ by its definition yields:
\begin{align*}
\left|\overline{F^{(1)}(x+i(2a^j-a^{j+1}))}F^{(1)}(x+ia^{j+1}) -
\right.&\left.
\overline{F^{(2)}(x+i(2a^j-a^{j+1}))}F^{(2)}(x+ia^{j+1})\right|\\
&=|H(x+i(a^{j+1}-a^j))|\\
&\leq\epsilon^{c-\alpha}\mathcal{N}(a^j)^{2c} \left(2\mathcal{N}(2a^j-a^{j-k})\mathcal{N}(a^{j-k})\right)^{1-c}
\end{align*}
To conclude, it suffices to note that, because of the way we chose $k$, $2a^j-a^{j-k}\geq a^{j+1}$ so, from \ref{lem:relations_for_N}, $\mathcal{N}(2a^j-a^{j-k})\leq\mathcal{N}(a^{j+1})$.
\end{proof}

\begin{thm}\label{thm:recurrence}
Let the following numbers be fixed:
\begin{equation*}
\epsilon,\kappa\in]0;1[\mbox{ with }\kappa\geq\epsilon^{2(1-c)}
\hskip 1cm
M>0
\hskip 1cm
\mu\in[0;M[
\hskip 1cm
j\in\Z
\hskip 1cm
K\in\N
\end{equation*}
We assume that, for any $n\in\{j+1,...,j+K\}$ and $x\in[-Ma^{j+K};Ma^{j+K}]$:
\begin{gather}
\left||F^{(1)}(x+ia^n)|^2-|F^{(2)}(x+ia^n)|^2 \right|\leq\epsilon\mathcal{N}(a^n)^2\label{eq:hyp_rec_1}\\
|F^{(1)}(x+ia^n)|^2,|F^{(2)}(x+ia^n)|^2\geq\kappa\mathcal{N}(a^n)^2\label{eq:hyp_rec_2}
\end{gather}

We define recursively:
\begin{align*}
n_0&=j+K&w_0&=a^{j+K}\\
\forall l\in\N\hskip 1cm
n_{l+1}&=n_l-2&w_{l+1}&=w_l-(a-1)^2a^{n_{l+1}}
\end{align*}
We define:
\begin{equation*}
D_l=\underset{s=0}{\overset{l-1}{\prod}}\left(\frac{\mathcal{N}(a^{n_s-1-k})}{\mathcal{N}(a^{n_s-2})}\right)
\hskip 1cm\mbox{and}\hskip 1cm
c_l=c-2\left(1+\frac{2}{a}\frac{a^2-1}{a+2}\left(\underset{s=0}{\overset{l-1}{\sum}}a^{-2s}\right)\right)\left(\frac{e^{-\pi\mu}}{1-e^{-\pi\mu}}\right)
\end{equation*}
For any $l\geq 0$ such that $n_l\geq j$ and $M-(l+1)\mu>0$, we have, provided that $c_l<1$:
\begin{align}
\frac{1}{\mathcal{N}(w_l)\mathcal{N}(a^{n_l})}\left|\overline{F^{(1)}(x+iw_l)}F^{(1)}(x+ia^{n_l})-\right.&\left.\overline{F^{(2)}(x+iw_l)}
F^{(2)}(x+ia^{n_l})\right|\nonumber\\
&\leq 3 D_l\left(\frac{2\kappa^{-l/2}-\kappa^{-(l-1)/2}-1}{1-\sqrt{\kappa}}\right)\epsilon^{c_l}\nonumber\\
&\left(\forall x\in[-(M-(l+1)\mu)a^{j+K};(M-(l+1)\mu)a^{j+K}]\right)\label{eq:recurrence}
\end{align}
\end{thm}
\begin{proof}
We procede by induction over $l$.

For $l=0$, \eqref{eq:recurrence} is a direct consequence of \eqref{eq:hyp_rec_1}. Indeed, $w_0=a^{n_0},D_0=1,c_0<1$ so, for $x\in[-Ma^{j+K};Ma^{j+K}]$:
\begin{equation*}
\frac{1}{\mathcal{N}(a^{n_0})^2}
\left||F^{(1)}(x+ia^{n_0})|^2-|F^{(2)}(x+ia^{n_0})|^2\right|
\leq \epsilon\leq 3D_0\epsilon^{c_0}
\end{equation*}

We now suppose that \eqref{eq:recurrence} holds for $l$ and prove it for $l+1$.

We procede in two parts. First, we use the induction hypothesis to bound the function $\left|\overline{F^{(1)}(x+iw_l)}F^{(1)}(x+i(2a^{n_l-1}-a^{n_l}))-\overline{F^{(2)}(x+iw_l)}F^{(2)}(x+i(2a^{n_l-1}-a^{n_l}))\right|$. In a second part, we use this bound to obtain the desired result.

\nl
\textbf{First part:} by triangular inequality,
\begin{align}
\left|\overline{F^{(1)}(x+iw_l)}\right.&\left.F^{(1)}(x+i(2a^{n_l-1}-a^{n_l}))-\overline{F^{(2)}(x+iw_l)}
F^{(2)}(x+i(2a^{n_l-1}-a^{n_l}))\right|\nonumber\\
&\leq \left|\overline{F^{(1)}(x+iw_l)}F^{(1)}(x+ia^{n_l})-\overline{F^{(2)}(x+iw_l)}F^{(2)}(x+ia^{n_l})\right|\label{eq:line_1}\\
&\hskip 2cm\times\left|\frac{F^{(1)}(x+i(2a^{n_l-1}-a^{n_l}))}{F^{(1)}(x+ia^{n_l})}\right|\nonumber\\
&+\left|\frac{F^{(2)}(x+ia^{n_l})}{F^{(1)}(x+ia^{n_l})}-\frac{\overline{F^{(1)}(x+ia^{n_l})}}{\overline{F^{(2)}(x+ia^{n_l})}}
\right|\label{eq:line_2}\\
&\hskip 2cm\times\left|
\overline{F^{(2)}(x+iw_l)}F^{(1)}(x+i(2a^{n_l-1}-a^{n_l}))
\right|\nonumber\\
&+\left|\overline{F^{(1)}(x+ia^{n_l})}F^{(1)}(x+i(2a^{n_l-1}-a^{n_l}))-\overline{F^{(2)}(x+ia^{n_l})}F^{(2)}(x+i(2a^{n_l-1}-a^{n_l}))\right|\label{eq:line_3}\\
&\hskip 2cm\times\left| \frac{\overline{F^{(2)}(x+iw_l)}}{\overline{F^{(2)}(x+ia^{n_l})}}
\right|\nonumber
\end{align}
By the induction hypothesis, for $x\in[-(M-2l\mu)a^{j+K};(M-2l\mu)a^{j+K}]$, \eqref{eq:line_1} is bounded by:
\begin{align*}
\left|\overline{F^{(1)}(x+iw_l)}F^{(1)}(x+ia^{n_l})\right.&\left.-\overline{F^{(2)}(x+iw_l)}F^{(2)}(x+ia^{n_l})\right|\\
&\leq 3D_l\left(\frac{2\kappa^{-l/2}-\kappa^{-(l-1)/2}-1}{1-\sqrt{\kappa}}\right)\mathcal{N}(w_l)\mathcal{N}(a^{n_l})\epsilon^{c_l}
\end{align*}
Because of \eqref{eq:hyp_rec_1} and \eqref{eq:hyp_rec_2} (for $n=n_l$), \eqref{eq:line_2} is bounded by:
\begin{align*}
\left|\frac{F^{(2)}(x+ia^{n_l})}{F^{(1)}(x+ia^{n_l})}-\frac{\overline{F^{(1)}(x+ia^{n_l})}}{\overline{F^{(2)}(x+ia^{n_l})}}
\right|
=\left|\frac{|F^{(2)}(x+ia^{n_l})|^2-|F^{(1)}(x+ia^{n_l})|^2}{F^{(1)}(x+ia^{n_l})F^{(2)}(x+ia^{n_l})}\right|
\leq \frac{\epsilon}{\kappa}
\end{align*}
Finally, from the lemma \ref{lem:estimation_from_one_modulus} applied to $j=n_l-1$, \eqref{eq:line_3} is bounded by:
\begin{align*}
\left|\overline{F^{(1)}(x+ia^{n_l})}F^{(1)}(x+i(2a^{n_l-1}-a^{n_l}))
\right.&-\left.
\overline{F^{(2)}(x+ia^{n_l})}F^{(2)}(x+i(2a^{n_l-1}-a^{n_l}))\right|\\
&\leq \mathcal{N}(a^{n_l-1})^{2c}(2\mathcal{N}(a^{n_l})\mathcal{N}(a^{n_l-1-k}))^{1-c}\epsilon^{c-\alpha}
\end{align*}
for all $x\in[-Ma^{j+K}+\mu a^j;Ma^{j+K}-\mu a^j]\supset [-(M-(l+1)\mu)a^{j+K};(M-(2l+1)\mu)a^{j+K}]$.

We insert these bounds into the triangular inequality. We also use the fact that $|F^{(1)}(x+ia^{n_l})|,|F^{(2)}(x+ia^{n_l})|\geq\sqrt{\kappa}\mathcal{N}(a^{n_l})$. We get, for any $x\in[-(M-(l+1)\mu)a^{j+K};(M-(l+1)\mu)a^{j+K}]$:
\begin{align*}
\left|\overline{F^{(1)}(x+iw_l)}\right.&\left.F^{(1)}(x+i(2a^{n_l-1}-a^{n_l}))-\overline{F^{(2)}(x+iw_l)}
F^{(2)}(x+i(2a^{n_l-1}-a^{n_l}))\right|\\
&\leq \frac{1}{\sqrt{\kappa}} 3D_l \left(\frac{2\kappa^{-l/2}-\kappa^{-(l-1)/2}-1}{1-\sqrt{\kappa}}\right)\mathcal{N}(w_l)\mathcal{N}(2a^{n_l-1}-a^{n_l})\epsilon^{c_l}\\
&\quad+ \frac{\epsilon}{\kappa}\mathcal{N}(w_l)\mathcal{N}(2a^{n_l-1}-a^{n_l})\\
&\quad+ \frac{2^{1-c}}{\sqrt{\kappa}}\frac{\mathcal{N}(w_l)}{\mathcal{N}(a^{n_l})^c}\mathcal{N}(a^{n_l-1})^{2c}\mathcal{N}(a^{n_l-1-k})^{1-c}\epsilon^{c-\alpha}
\end{align*}
We must now simplify this inequality.

First, $2a^{n_l-1}-a^{n_l}=ca^{n_l-1}+(1-c)a^{n_l-1-k}$ so, from the lemma \ref{lem:relations_for_N}, $\mathcal{N}(2a^{n_l-1}-a^{n_l})\leq\mathcal{N}(a^{n_l-1})^c\mathcal{N}(a^{n_l-1-k})^{1-c}$. So:
\begin{align*}
\left|\overline{F^{(1)}(x+iw_l)}\right.&\left.F^{(1)}(x+i(2a^{n_l-1}-a^{n_l}))-\overline{F^{(2)}(x+iw_l)}
F^{(2)}(x+i(2a^{n_l-1}-a^{n_l}))\right|\\
&\leq \mathcal{N}(w_l)\mathcal{N}(a^{n_l-1})^c\mathcal{N}(a^{n_l-1-k})^{1-c}\\
&\quad\times\left(
\frac{1}{\sqrt{\kappa}} 3D_l \left(\frac{2\kappa^{-l/2}-\kappa^{-(l-1)/2}-1}{1-\sqrt{\kappa}}\right)\epsilon^{c_l}+ \frac{\epsilon}{\kappa}
+ \frac{2^{1-c}}{\sqrt{\kappa}}\frac{\mathcal{N}(a^{n_l-1})^{c}}{\mathcal{N}(a^{n_l})^c}\epsilon^{c-\alpha}\right)
\end{align*}
Now we note that $1\leq\frac{\mathcal{N}(a^{n_l-1})^c}{\mathcal{N}(a^{n_l})^c}$ (from the lemma \ref{lem:relations_for_N} again, because $a^{n_l-1}\leq a^{n_l}$). Because $\kappa\geq\epsilon^{2(1-c)}$, we also have $\frac{\epsilon}{\kappa}\leq\frac{\epsilon^c}{\sqrt{\kappa}}\leq\frac{\epsilon^{c-\alpha}}{\sqrt{\kappa}}$. And as $c-\alpha\geq c_l$, $\epsilon^{c-\alpha}\leq\epsilon^{c_l}$. This gives:
\begin{align*}
\left|\overline{F^{(1)}(x+iw_l)}\right.&\left.F^{(1)}(x+i(2a^{n_l-1}-a^{n_l}))-\overline{F^{(2)}(x+iw_l)}
F^{(2)}(x+i(2a^{n_l-1}-a^{n_l}))\right|\\
&\leq \frac{\epsilon^{c_l}}{\sqrt{\kappa}}\mathcal{N}(w_l)\frac{\mathcal{N}(a^{n_l-1})^{2c}\mathcal{N}(a^{n_l-1-k})^{1-c}}{\mathcal{N}(a^{n_l})^c}
\left( 3D_l \left(\frac{2\kappa^{-l/2}-\kappa^{-(l-1)/2}-1}{1-\sqrt{\kappa}}\right)+1
+{2^{1-c}}\right)
\end{align*}
If we bound $2^{1-c}$ by $2$ and notice that $D_l\geq 1$ (because, from \ref{lem:relations_for_N}, it is a product of terms bigger that $1$), we have:
\begin{align*}
\left|\overline{F^{(1)}(x+iw_l)}\right.&\left.F^{(1)}(x+i(2a^{n_l-1}-a^{n_l}))-\overline{F^{(2)}(x+iw_l)}
F^{(2)}(x+i(2a^{n_l-1}-a^{n_l}))\right|\\
&\leq \frac{\epsilon^{c_l}}{\sqrt{\kappa}}3D_l\mathcal{N}(w_l)\frac{\mathcal{N}(a^{n_l-1})^{2c}\mathcal{N}(a^{n_l-1-k})^{1-c}}{\mathcal{N}(a^{n_l})^c}
\left( \left(\frac{2\kappa^{-l/2}-\kappa^{-(l-1)/2}-1}{1-\sqrt{\kappa}}\right)+1\right)\\
&\leq 3\epsilon^{c_l}D_l\mathcal{N}(w_l)\frac{\mathcal{N}(a^{n_l-1})^{2c}\mathcal{N}(a^{n_l-1-k})^{1-c}}{\mathcal{N}(a^{n_l})^c}
\left(\frac{2\kappa^{-(l+1)/2}-\kappa^{-l/2}-1}{1-\sqrt{\kappa}}\right)
\end{align*}

Finally, from \ref{lem:relations_for_N}, we have $\mathcal{N}(a^{n_l-1})\leq\mathcal{N}(a^{n_l})^{1/2}\mathcal{N}(2a^{n_l-1}-a^{n_l})^{1/2}$ so:
\begin{align*}
\left|\overline{F^{(1)}(x+iw_l)}\right.&\left.F^{(1)}(x+i(2a^{n_l-1}-a^{n_l}))-\overline{F^{(2)}(x+iw_l)}
F^{(2)}(x+i(2a^{n_l-1}-a^{n_l}))\right|\\
&\leq 3\epsilon^{c_l}D_l\mathcal{N}(w_l)\mathcal{N}(2a^{n_l-1}-a^{n_l})
\left(\frac{\mathcal{N}(a^{n_l-1-k})}{\mathcal{N}(2a^{n_l-1}-a^{n_l})}\right)^{1-c}
\left(\frac{2\kappa^{-(l+1)/2}-\kappa^{-l/2}-1}{1-\sqrt{\kappa}}\right)
\end{align*}

\textbf{Second part:} we define, for any $z\in\C$ such that $-2a^{n_l-1}+a^{n_l}<\Im z<w_l$:
\begin{equation*}
H(z)=\overline{F^{(1)}(\overline{z}+iw_l)}F^{(1)}(z+i(2a^{n_l-1}-a^{n_l}))
-\overline{F^{(2)}(\overline{z}+iw_l)}F^{(2)}(z+i(2a^{n_l-1}-a^{n_l}))
\end{equation*}
We write:
\begin{equation*}
B = \frac{3}{2}D_l\left(\frac{\mathcal{N}(a^{n_l-1-k})}{\mathcal{N}(2a^{n_l-1}-a^{n_l})}\right)^{1-c}
\left(\frac{2\kappa^{-(l+1)/2}-\kappa^{-l/2}-1}{1-\sqrt{\kappa}}\right)
\end{equation*}
From the first part:
\begin{align*}
|H(x+iy)|&\leq 2\mathcal{N}(w_l)\mathcal{N}(2a^{n_l-1}-a^{n_l})B\epsilon^{c_l}
&\mbox{if }y=0,x\in[-(M-(l+1)\mu)a^{j+K};(M-(l+1)\mu)a^{j+K}]\\
&\leq 2\mathcal{N}(w_l)\mathcal{N}(2a^{n_l-1}-a^{n_l})
&\mbox{if }y=0,x\notin[-(M-(l+1)\mu)a^{j+K};(M-(l+1)\mu)a^{j+K}]
\end{align*}
Moreover, if we set $y_l=w_l-2a^{n_l-1}+a^{n_l}$:
\begin{align*}
H(x+iy_l)&=\overline{F^{(1)}(x+i(2a^{n_l-1}-a^{n_l}))}F^{(1)}(x+iw_l)-\overline{F^{(2)}(x+i(2a^{n_l-1}-a^{n_l}))}F^{(2)}(x+iw_l)\\
&=\overline{H(x)}
\end{align*}
Thus, we also have:
\begin{align*}
|H(x+iy)|&\leq 2\mathcal{N}(w_l)\mathcal{N}(2a^{n_l-1}-a^{n_l})B\epsilon^{c_l}
&\mbox{if }y=y_l,x\in[-(M-(l+1)\mu)a^{j+K};(M-(l+1)\mu)a^{j+K}]\\
&\leq 2\mathcal{N}(w_l)\mathcal{N}(2a^{n_l-1}-a^{n_l})
&\mbox{if }y=y_l,x\notin[-(M-(l+1)\mu)a^{j+K};(M-(l+1)\mu)a^{j+K}]
\end{align*}
We apply the lemma \ref{lem:holes_on_both_lines} with $a=0,b=y_l$. For $\Im z=(a-1)^2a^{n_l-2}$ and $|\Re z|\leq (M-(l+1)\mu)a^{j+K}$:
\begin{equation}
|H(z)|\leq 2\mathcal{N}(w_l)\mathcal{N}(2a^{n_l-1}-a^{n_l})(B\epsilon^{c_l})^{f(z)}\label{eq:maj_H}
\end{equation}
with $f(z)\geq 1-4\frac{(a-1)^2a^{n_l-2}}{y_l}\left(\frac{e^{-\pi\frac{(M-(l+1)\mu)a^{j+K}-|\mbox{\scriptsize\rm Re }z|}{y_l}}}{1-e^{-\pi\frac{(M-(l+1)\mu)a^{j+K}-|\mbox{\scriptsize\rm Re }z|}{y_l}}}\right)$.

From the definition of $w_l$, one may check that $(w_l)$ is a decreasing sequence which converges to $\frac{2a^{j+K}}{a+1}$ when $l$ goes to $\infty$. So, for any $l\geq 0$:
\begin{gather*}
y_l\leq w_l\leq w_0=a^{j+K}\\
y_l\geq \frac{2a^{j+K}}{a+1}-2a^{n_l-1}+a^{n_l}
\geq \frac{2a^{j+K}}{a+1}-2a^{j+K-1}+a^{j+K}=\frac{(a-1)(a+2)}{(a+1)}a^{j+K-1}
\end{gather*}
From this we deduce:
\begin{equation*}
f(z)\geq 1-4\,\frac{a^2-1}{a+2}a^{-2l-1}
\left(\frac{e^{-\pi\frac{(M-(l+1)\mu)a^{j+K}-|\mbox{\scriptsize\rm Re }z|}{a^{j+K}}}}{1-e^{-\pi\frac{(M-(l+1)\mu)a^{j+K}-|\mbox{\scriptsize\rm Re }z|}{a^{j+K}}}}\right)
\end{equation*}
So, when $|\Re z|\leq (M-(l+2)\mu)a^{j+K}$, $f(z)\geq 1-4\,\frac{a^2-1}{a+2}a^{-2l-1}\left(\frac{e^{-\pi\mu}}{1-e^{-\pi\mu}}\right)$.

As $B\geq 1$ and $f(z)\leq 1$, $B^{f(z)}\leq B$. Moreover, $c_l\leq 1$ so $c_lf(z)\geq c_l -(1-f(z))$ if $1-f(z)\geq 0$. The equation \eqref{eq:maj_H} thus gives:
\begin{align*}
|H(z)|&\leq 2\mathcal{N}(w_l)\mathcal{N}(2a^{n_l-1}-a^{n_l})B\epsilon^{c_l-4\,\frac{a^2-1}{a+2}a^{-2l-1}\left(\frac{e^{-\pi\mu}}{1-e^{-\pi\mu}}\right)}\\
&=2\mathcal{N}(w_l)\mathcal{N}(2a^{n_l-1}-a^{n_l})B\epsilon^{c_{l+1}}\\
&=3D_l\mathcal{N}(w_l)\mathcal{N}(2a^{n_l-1}-a^{n_l})^c\mathcal{N}(a^{n_l-1-k})^{1-c}\left(\frac{2\kappa^{-(l+1)/2}-\kappa^{-l/2}-1}{1-\sqrt{\kappa}}\right)\epsilon^{c_{l+1}}
\end{align*}
Because $w_l\geq w_{l+1}$ and $2a^{n_l-1}-a^{n_l}\geq a^{n_l-1-k}$, we have $\mathcal{N}(w_l)\leq\mathcal{N}(w_{l+1})$ and $\mathcal{N}(2a^{n_l-1}-a^{n_l})\leq\mathcal{N}(a^{n_l-1-k})$. Thus:
\begin{align*}
|H(z)|&\leq
3D_l\mathcal{N}(w_{l+1})\mathcal{N}(a^{n_l-2})\frac{\mathcal{N}(a^{n_l-1-k})}{\mathcal{N}(a^{n_l-2})}\left(\frac{2\kappa^{-(l+1)/2}-\kappa^{-l/2}-1}{1-\sqrt{\kappa}}\right)\epsilon^{c_{l+1}}\\
&=3D_{l+1}\mathcal{N}(w_{l+1})\mathcal{N}(a^{n_l-2})\left(\frac{2\kappa^{-(l+1)/2}-\kappa^{-l/2}-1}{1-\sqrt{\kappa}}\right)\epsilon^{c_{l+1}}
\end{align*}

So, for any $x\in[-(M-(l+2)\mu)a^{j+K};(M-(l+2)\mu)a^{j+K}]$:
\begin{align*}
\frac{1}{\mathcal{N}(w_{l+1})\mathcal{N}(a^{n_{l+1}})}&\left|\overline{F^{(1)}(x+iw_{l+1})}F^{(1)}(x+ia^{n_{l+1}})-\overline{F^{(2)}(x+iw_{l+1})}
F^{(2)}(x+ia^{n_{l+1}})\right|\\
&=|H(x+i (a-1)^2a^{n_l-2})|\\
&\leq 3D_{l+1}\left(\frac{2\kappa^{-(l+1)/2}-\kappa^{-l/2}-1}{1-\sqrt{\kappa}}\right)\epsilon^{c_{l+1}}
\end{align*}
This is exactly the induction hypothesis at the order $l+1$.
\end{proof}

\begin{proof}[Proof of the theorem \ref{thm:stability_a_less_than_2}]
We will obtain the desired theorem as a corollary of the previous one (\ref{thm:recurrence}).

The conditions \eqref{eq:hyp_rec_1} and \eqref{eq:hyp_rec_2} in the statement of the theorem \ref{thm:recurrence} are equivalent to \eqref{eq:thm_stab_hyp_1} and \eqref{eq:thm_stab_hyp_2}, required in the theorem \ref{thm:stability_a_less_than_2}.

Thus, if we fix $\mu\in[0;M[$, we have that, for any $l\geq 0$ such that $n_l\geq j$ and $M-(l+1)\mu>0$, under the condition that $c_l<1$:
\begin{align*}
\frac{1}{\mathcal{N}(w_l)\mathcal{N}(a^{n_l})}\left|\overline{F^{(1)}(x+iw_l)}F^{(1)}(x+ia^{n_l})-\right.&\left.\overline{F^{(2)}(x+iw_l)}
F^{(2)}(x+ia^{n_l})\right|\\
&\leq 3 D_l\left(\frac{2\kappa^{-l/2}-\kappa^{-(l-1)/2}-1}{1-\sqrt{\kappa}}\right)\epsilon^{c_l}\\
&\left(\forall x\in[-(M-(l+1)\mu)a^{j+K};(M-(l+1)\mu)a^{j+K}]\right)
\end{align*}
where the constants are defined as in \ref{thm:recurrence}.

We can check that, for any $l$, $w_l=\frac{a^{j+K}}{a+1}\left(2+(a-1)a^{-2l}\right)$.

We take $l=K/2$. We then have $w_l=\frac{2}{a+1}a^{j+K}+\frac{a-1}{a+1}a^j=a^J$ and $n_l=j$. For this $l$, the previous inequality is equivalent to:
\begin{align*}
\frac{1}{N_JN_j}\left|\overline{f^{(1)}\star\psi_{J}(x)}f^{(1)}\star\psi_j(x)-\overline{f^{(2)}\star\psi_{J}(x)}\right.&\left.f^{(2)}\star\psi_j(x)
\right|\\
&\leq 3 D_l\left(\frac{2\kappa^{-K/4}-\kappa^{-(K-2)/4}-1}{1-\sqrt{\kappa}}\right)\epsilon^{c_l}
\end{align*}
We observe that $c_l\geq \underset{l\to\infty}{\lim}c_l=c-2\left(1+2\,\frac{a}{a+2}\right)\left(\frac{e^{-\pi\mu}}{1-e^{-\pi\mu}}\right)\geq c\,-\,4\left(\frac{e^{-\pi\mu}}{1-e^{-\pi\mu}}\right)$ and $\frac{2\kappa^{-K/4}-\kappa^{-(K-2)/4}-1}{1-\sqrt{\kappa}}\leq \frac{2\kappa^{-K/4}}{1-\sqrt{\kappa}}$.

So, for any $x\in[-(M-\mu(1+K/2))a^{j+K};(M-\mu(1+K/2))a^{j+K}]$:
\begin{align*}
\frac{1}{N_JN_j}\left|\overline{f^{(1)}\star\psi_{J}(x)}f^{(1)}\star\psi_j(x)-\overline{f^{(2)}\star\psi_{J}(x)}\right.&\left.f^{(2)}\star\psi_j(x)
\right|\\
&\leq 6 D_l\frac{\kappa^{-K/4}}{1-\sqrt{\kappa}}\epsilon^{c-4\left(\frac{e^{-\pi\mu}}{1-e^{-\pi\mu}}\right)}
\end{align*}
From the equation \eqref{eq:recall_F_f_psi_j}:
\begin{equation*}
D_l=\underset{s=0}{\overset{K/2-1}{\prod}}\left(\frac{\mathcal{N}(a^{n_s-1-k})}{\mathcal{N}(a^{n_s-2})}\right)
=\underset{s=0}{\overset{K/2-1}{\prod}}\left(a^{p(k-1)}\frac{N_{n_s-1-k}}{N_{n_s-2}}\right)
\end{equation*}
For $\mu=\frac{M}{K+2}$, our last inequality is exactly the desired result.
\end{proof}

\section{Bounds for holomorphic functions\label{app:bounds}}

In the proofs of the section \ref{s:strong_stability_result}, we often have to consider holomorphic functions defined on a band of the complex plane. We want to obtain informations about their values inside the band from their values on the boundary of the band. This is the purpose of the three theorems contained in this section.

In the whole section, $a,b$ are fixed real numbers such that $a<b$. We write $B_{a,b}=\{z\in\C\mbox{ s.t. }a<\Im z<b\}$. We consider a holomorphic function $W:B_{a,b}\to \C$ which satisfies the following properties:
\begin{enumerate}
\item[(i)] $W$ is bounded on $B_{a,b}$.
\item[(ii)] $W$ admits a continuous extension over $\overline{B}_{a,b}$, which we still denote by $W$.
\end{enumerate}

The first theorem we need is a well-known fact. We recall its demonstration because it is very short and relies on the same idea that will also be used in the other proofs.
\begin{thm}\label{thm:bounded_on_a_band}
We suppose that, for some $A,B>0$:
\begin{gather*}
|W(z)|\leq A\mbox{ if }\Im z=a\\
|W(z)|\leq B\mbox{ if }\Im z=b
\end{gather*}
Then, for all $t\in]0;1[$ and all $z\in\C$ such that $\Im z=(1-t)a+tb$:
\begin{equation*}
|W(z)|\leq A^{1-t}B^t
\end{equation*}
\end{thm}
\begin{proof}
For every $\epsilon>0$ and $z\in\overline{B}_{a,b}$:
\begin{equation*}
L(z)=\log(|W(z)|)-\frac{(b-\Im z)\log(A)+(\Im z-a)\log(B)}{b-a}-\epsilon\log|z+i(1-a)|
\end{equation*}
is subharmonic on $B_{a,b}$ and continuous on $\overline{B}_{a,b}$. It is upper-bounded and takes negative values on $\partial B_{a,b}$. Moreover, $L(z)\to-\infty$ when $\Re(z)\to\pm\infty$. From the maximum principle, this function must be negative on $B_{a,b}$.

Letting $\epsilon$ go to $0$ implies:
\begin{gather*}
\log(|W(z)|)\leq\frac{(b-\Im z)\log(A)+(\Im z-a)\log(B)}{b-a}
\quad\quad\forall z\in\overline{B}_{a,b}\\
\Rightarrow\quad
|W(z)|\leq A^{\frac{b-\mbox{\tiny Im } z}{b-a}}B^{\frac{\mbox{\tiny Im } z-a}{b-a}}
\end{gather*}
\end{proof}

\begin{lem}\label{lem:hole_on_one_line}
Let $A,B,\epsilon>0$ be fixed real numbers, with $\epsilon\leq 1$. We assume that:
\begin{gather*}
|W(z)|\leq B\mbox{ if }\Im z=b\\
|W(z)|\leq A\mbox{ if }\Im z=a\mbox{ and }\Re z\notin[-M;M]\\
|W(z)|\leq \epsilon A\mbox{ if }\Im z=a\mbox{ and }\Re z\in[-M;M]
\end{gather*}
Then, for all $z$ such that $a<\Im z<b$, if $t\in[0;1]$ is such that $\Im z=(1-t)a+tb$:
\begin{equation*}
|W(z)|\leq\epsilon^{f(z)}A^{1-t}B^{t}
\end{equation*}
where:
\begin{equation*}
f(z)=\frac{1}{\pi}\arg\left(\frac{e^{\pi M/(b-a)}-e^{\pi (z-ia)/(b-a)}}{e^{-\pi M/(b-a)}-e^{\pi (z-ia)/(b-a)}}\right)
\end{equation*}
and this function satisfies, when $|\Re z|\leq M$: $f(z)\geq (1-t)-2t\frac{e^{-\pi\frac{M-\left|\mbox{\scriptsize\rm Re }z\right|}{b-a}}}{1-e^{-\pi\frac{M-\left|\mbox{\scriptsize\rm Re }z\right|}{b-a}}}$.
\end{lem}
\begin{proof}
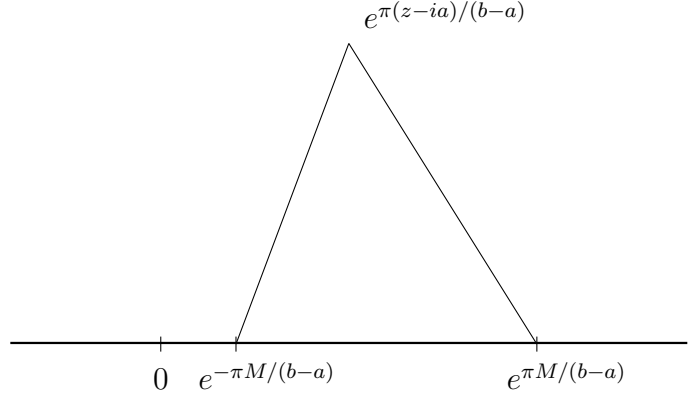
\begin{figure}
\setlength{\unitlength}{1mm}
\begin{picture}(20,50)(-35,-7)
\put(0,0){\rule{90mm}{0.3mm}}
\put(20,-1){\line(0,1){2}}
\put(30,-1){\line(0,1){2}}
\put(70,-1){\line(0,1){2}}
\put(25,-6){$e^{-\pi M/(b-a)}$}
\put(19,-6){$0$}
\put(66,-6){$e^{\pi M/(b-a)}$}
\put(30,0){\line(15,40){15}}
\put(70,0){\line(-25,40){25}}
\put(47,42){$e^{\pi (z-ia)/(b-a)}$}
\end{picture}
\caption{\label{fig:angles}Positions of the points used in the definition of $f$}
\end{figure}
The function $f$ may be continuously extended to $\overline{B}_{a,b}-\{-M+ia;M+ia\}$. By looking at the figure \ref{fig:angles}, one sees that:
\begin{align*}
f(x+ia)&=0\mbox{ for all }x\in\R-[-M;M]\\
&=1\mbox{ for all }x\in]-M;M[\\
f(x+ib)&=0\mbox{ for all }x\in\R\\
\end{align*}

We set:
\begin{equation*}
f(-M+ia)=f(M+ia)=1
\end{equation*}
This definition makes the extension of $f$ upper semi-continuous on $\overline{B}_{a,b}$ (because $f\leq 1$ on all $B_{a,b}$).

For any $\eta>0$, the following function is subharmonic on $B_{a,b}$:
\begin{equation*}
L(z)=\log(|W(z)|)-\log(\epsilon)f(z)-\frac{(b-\Im z)\log(A)+(\Im z-a)\log(B)}{b-a}-\eta\log|z+i(1-a)|
\end{equation*}
It is upper semi-continuous on $\overline{B}_{a,b}$ and tends to $-\infty$ when $\Re z\to\pm\infty$. Thus, this function admits a local maximum over $\overline{B}_{a,b}$. This maximum is attained on $\partial B_{a,b}$, because $L$ is subharmonic.

From the hypotheses, one can check that $L(z)\leq 0$ for all $z\in\partial B_{a,b}$. The function $L$ is thus negative on the whole band $\overline{B}_{a,b}$. Letting $\eta$ go to zero gives, for all $z\in B_{a,b}$ such that $\Im z=(1-t)a+tb$:
\begin{equation*}
|W(z)|\leq\epsilon^{f(z)}A^{1-t}B^{t}
\end{equation*}
We are only left to show that $f(z)\geq(1-t)-2t \frac{e^{-\pi\frac{M-\left|\mbox{\scriptsize\rm Re }z\right|}{b-a}}}{1-e^{-\pi\frac{M-\left|\mbox{\scriptsize\rm Re }z\right|}{b-a}}}$ when $\Im z=(1-t)a+tb$.

\noindent If we write $x=\Re(z)$, we have:
\begin{align*}
f(z)&=\frac{1}{\pi}\arg\left(-e^{-i\pi t}\frac{1-e^{\pi(x-M)/(b-a)}e^{\pi it})}{1-e^{-\pi(M+x)/(b-a)}e^{-\pi it}}\right)\\
&=(1-t)+\frac{1}{\pi}\arg\left(\frac{1-e^{\pi(x-M)/(b-a)}e^{\pi it})}{1-e^{-\pi(M+x)/(b-a)}e^{-\pi it}}
\right)
\end{align*}
We note that:
\begin{align*}
\left|\arg\left(1-e^{\pi\frac{x-M}{b-a}}e^{i\pi t}\right)\right|
&\leq\left|\tan\left(1-e^{\pi\frac{x-M}{b-a}}e^{i\pi t}\right)\right|\\
&=\left|\sin(\pi t)\right|\frac{e^{\pi\frac{x-M}{b-a}}}{1-e^{\pi\frac{x-M}{b-a}}\cos(\pi t)}\\
&\leq\left|\sin(\pi t)\right|\frac{e^{\pi\frac{x-M}{b-a}}}{1-e^{\pi\frac{x-M}{b-a}}}\\
&\leq \pi t \frac{e^{\pi\frac{x-M}{b-a}}}{1-e^{\pi\frac{x-M}{b-a}}}
\leq \pi t\frac{e^{-\pi\frac{M-\left|\mbox{\scriptsize\rm Re }z\right|}{b-a}}}{1-e^{-\pi\frac{M-\left|\mbox{\scriptsize\rm Re }z\right|}{b-a}}}
\end{align*}
And the same inequality holds for $\left|\arg\left(1-e^{-\pi\frac{M+x}{b-a}}e^{-i\pi t}\right)\right|$. This implies the result.

\end{proof}

The proof of the third result is similar to the proof of the second one. We do not reproduce it.
\begin{lem}\label{lem:holes_on_both_lines}
Let $M,A,\epsilon>0$ be fixed real numbers, with $\epsilon\leq 1$. We assume that:
\begin{align*}
|W(x+ia)|\leq A\quad\quad&|W(x+ib)|\leq A\quad\quad\forall x\in\R-[-M;M]\\
|W(x+ia)|\leq \epsilon A\quad\quad&|W(x+ib)|\leq\epsilon A\quad\quad\forall x\in[-M;M]
\end{align*}
Then, for all $z$ such that $a<\Im z<b$:
\begin{equation*}
|W(z)|\leq\epsilon^{f(z)}A
\end{equation*}
where:
\begin{equation*}
f(z)=\frac{1}{\pi}\arg\left(\frac{e^{\pi M/(b-a)}-e^{\pi (z-ia)/(b-a)}}{e^{-\pi M/(b-a)}-e^{\pi (z-ia)/(b-a)}}.\frac{-e^{-\pi M/(b-a)}-e^{\pi (z-ia)/(b-a)}}{-e^{\pi M/(b-a)}-e^{\pi (z-ia)/(b-a)}} \right)
\end{equation*}
and this function satisfies, when $|\Re z|\leq M$: $f(z)\geq 1-4 t\left(\frac{e^{-\pi\frac{M-\left|\mbox{\scriptsize\rm Re }z\right|}{b-a}}}{1-e^{-\pi\frac{M-\left|\mbox{\scriptsize\rm Re }z\right|}{b-a}}}\right) $, for $t=\frac{\mbox{\footnotesize\rm Im } z-a}{b-a}$.
\end{lem}



\bibliographystyle{plainnat}
\bibliography{../bib.bib}

\end{document}

%% file: defs_p.tex
\usepackage{hyperref}

\usepackage{algorithmic,algorithm}

\newcommand{\nl}{\vskip 0.3cm}

\newcommand{\C}{{\mathbb C}}
\newcommand{\R}{{\mathbb R}}
\newcommand{\Z}{{\mathbb Z}}
\newcommand{\N}{{\mathbb N}}

\renewcommand{\H}{\mathbb{H}}

\renewcommand{\Im}{\mbox{\rm Im}\,}
\renewcommand{\Re}{\mbox{\rm Re}\,}

\newtheorem{thm}{Theorem}[section]
\newtheorem*{thm*}{Theorem}
\newtheorem{cor}[thm]{Corollary}
\newtheorem{lem}[thm]{Lemma}
\newtheorem*{lem*}{Lemma}

\newtheorem{rem}[thm]{Remark}



%% file: modules_arxiv.bbl
\begin{thebibliography}{19}
\providecommand{\natexlab}[1]{#1}
\providecommand{\url}[1]{\texttt{#1}}
\expandafter\ifx\csname urlstyle\endcsname\relax
  \providecommand{\doi}[1]{doi: #1}\else
  \providecommand{\doi}{doi: \begingroup \urlstyle{rm}\Url}\fi

\bibitem[Akutowicz(1956)]{akutowicz}
E.~J. Akutowicz.
\newblock On the determination of the phase of a {F}ourier integral, {I}.
\newblock \emph{Trans. Am. Math. Soc.}, 83:\penalty0 179--192, 1956.

\bibitem[Akutowicz(1957)]{akutowicz_compact}
E.~J. Akutowicz.
\newblock On the determination of the phase of a {F}ourier integral, {II}.
\newblock \emph{Proc. Am. Math. Soc.}, 8:\penalty0 234--238, 1957.

\bibitem[And\'en and Mallat(2011)]{anden}
J.~And\'en and S.~Mallat.
\newblock Multiscale scattering for audio classification.
\newblock \emph{Proceedings of the ISMIR 2011 conference}, 2011.

\bibitem[Balan and Wang(2013)]{balan_wang}
R.~Balan and Y.~Wang.
\newblock Invertibility and robustness of phaseless reconstruction.
\newblock \emph{Preprint}, August 2013.

\bibitem[Balan et~al.(2006)Balan, Casazza, and Edidin]{balan}
R.~Balan, P.~Casazza, and D.~Edidin.
\newblock On signal reconstruction without noisy phase.
\newblock \emph{Appl. Comp. Harm. Anal.}, 20:\penalty0 345--356, 2006.

\bibitem[Bandeira et~al.(2013)Bandeira, Cahill, Mixon, and
  Nelson]{bandeira_stability}
A.~S. Bandeira, J.~Cahill, D.~G. Mixon, and A.~A. Nelson.
\newblock Saving phase: Injectivity and stability for phase retrieval.
\newblock \emph{Preprint}, February 2013.

\bibitem[Bodmann and Hammen(2013)]{bodmann}
B.~G. Bodmann and N.~Hammen.
\newblock Stable phase retrieval with low-redundancy frames.
\newblock \emph{Preprint}, 2013.
\newblock http://arxiv.org/abs/1302.5487.

\bibitem[Cand\`es et~al.(2011)Cand\`es, Strohmer, and Voroninski]{candes2}
E.~J. Cand\`es, T.~Strohmer, and V.~Voroninski.
\newblock Phaselift: exact and stable signal recovery from magnitude
  measurements via convex programming.
\newblock \emph{To appear in Communications in Pure and Applied Mathematics},
  2011.

\bibitem[Cand\`es et~al.(2013)Cand\`es, Li, and Soltanolkotabi]{candes_li2}
E.~J. Cand\`es, X.~Li, and M.~Soltanolkotabi.
\newblock Phase retrieval from coded diffraction patterns.
\newblock \emph{Preprint}, 2013.
\newblock http://arxiv.org/abs/1310.3240.

\bibitem[Conca et~al.(2013)Conca, Edidin, Hering, and Vinzant]{conca}
A.~Conca, D.~Edidin, M.~Hering, and C.~Vinzant.
\newblock Algebraic characterization of injectivity in phase retrieval.
\newblock \emph{Preprint}, 2013.
\newblock http://arxiv.org/abs/1312.0158.

\bibitem[Eldar and Mendelson(2013)]{eldar}
Y.~C. Eldar and S.~Mendelson.
\newblock Phase retrieval: stability and recovery guarantees.
\newblock \emph{Applied and Computation harmonic analysis}, September 2013.

\bibitem[Fickus et~al.(2013)Fickus, Mixon, Nelson, and Wang]{fickus}
M.~Fickus, D.~G. Mixon, A.~A. Nelson, and Y.~Wang.
\newblock Phase retrieval from very few measurements.
\newblock \emph{Preprint}, 2013.
\newblock http://arxiv.org/abs/1307.7176.

\bibitem[Garnett(1981)]{garnett}
John~B. Garnett.
\newblock \emph{Bounded analytic functions}.
\newblock Academic Press, 1981.

\bibitem[Gerchberg and Saxton(1972)]{gerchberg}
R.~Gerchberg and W.~Saxton.
\newblock A practical algorithm for the determination of phase from image and
  diffraction plane pictures.
\newblock \emph{Optik}, 35:\penalty0 237--246, 1972.

\bibitem[Griffin and Lim(1984)]{griffin_lim}
D.~Griffin and J.~S. Lim.
\newblock Signal estimation from modified short-time fourier transform.
\newblock \emph{IEEE Transactions on acoustics, speech and signal processing},
  32:\penalty0 236--243, 1984.

\bibitem[Kryloff(1939)]{kryloff}
W.~Kryloff.
\newblock On functions which are regular in a half-plane.
\newblock \emph{Rec.math. (Mat. Sbornik)}, 6, 1939.

\bibitem[Mesgarani et~al.(2006)Mesgarani, Slaney, and Shamma]{shamma}
N.~Mesgarani, M.~Slaney, and S.~A. Shamma.
\newblock Discrimination of speech from nonspeech based on multiscale
  spectr-temporal modulations.
\newblock \emph{IEEE Transactions on audio, speech and language processing},
  14, May 2006.

\bibitem[Rudin(1921)]{rudin}
W.~Rudin.
\newblock \emph{Real and complex analysis, third edition}.
\newblock McGraw-Hill International Editions, 1921.

\bibitem[Walther(1963)]{walther}
A.~Walther.
\newblock The question of phase retrieval in optics.
\newblock \emph{Opt. Acta}, 10:\penalty0 41--49, 1963.

\end{thebibliography}
